\documentclass[onefignum,onetabnum,reqno]{siamart190516}

\usepackage{amsfonts}
\usepackage{graphicx,epstopdf}
\usepackage{mathrsfs}
\usepackage{bm}
\usepackage{color}
\usepackage{hyperref} 
\usepackage{xcolor}

\usepackage{caption, subcaption}%
\usepackage{enumitem}  

\usepackage{tikz,lipsum,lmodern}
\usepackage[most]{tcolorbox} 

\usepackage{geometry}
\allowdisplaybreaks[3]

\ifpdf
  \DeclareGraphicsExtensions{.eps,.pdf,.png,.jpg}
\else
  \DeclareGraphicsExtensions{.eps}
\fi

\newsiamremark{remark}{Remark}
\newsiamremark{example}{Example}
\newsiamremark{hypothesis}{Hypothesis}
\crefname{hypothesis}{Hypothesis}{Hypotheses}
\newsiamthm{claim}{Claim}
\newsiamthm{assum}{Assumption}

\title{Provably Positive Central DG Schemes via Geometric Quasilinearization 
for Ideal MHD Equations\thanks{Funding: The work of Kailiang Wu is supported in part by NSFC grant 12171227. The work of Chi-Wang Shu is supported in part by NSF grant DMS-2010107 and AFOSR grant FA9550-20-1-0055.
}}

\author{Kailiang Wu\thanks{Department of Mathematics, Southern University of Science and Technology, and National Center for Applied Mathematics Shenzhen (NCAMS), Shenzhen, Guangdong 518055, China ({\tt wukl@sustech.edu.cn}).} 
	\and Haili Jiang\thanks{School of Mathematical Sciences, Peking University, Beijing 100871, China  ({\tt jianghaili@pku.edu.cn}).} 
	\and 
Chi-Wang Shu\thanks{Division of Applied Mathematics, Brown University, Providence, RI 02912, USA
	({\tt chi-wang\_shu@brown.edu}).}
}

\usepackage{amsopn}

\DeclareMathAlphabet\mathbfcal{OMS}{cmsy}{b}{n}

\usepackage{stmaryrd} 
\def\jump#1{\llbracket #1 \rrbracket }

\usepackage{xparse} 

\NewDocumentCommand{\dgal}{sO{}m}{%
	\IfBooleanTF{#1}
	{\dgalext{#3}}
	{\dgalx[#2]{#3}}%
}

\NewDocumentCommand{\dgalext}{m}{%
	\sbox0{%
		\mathsurround=0pt 
		$\left\{\vphantom{#1}\right.\kern-\nulldelimiterspace$%
	}%
	\sbox2{\{}%
	\ifdim\ht0=\ht2
	\{\kern-.625\wd2 \{#1\}\kern-.625\wd2 \}%
	\else
	\left\{\kern-.7\wd0\left\{#1\right\}\kern-.7\wd0\right\}%
	\fi
}

\NewDocumentCommand{\dgalx}{om}{%
	\sbox0{\mathsurround=0pt$#1\{$}%
	\sbox2{\{}%
	\ifdim\ht0=\ht2
	\{\kern-.625\wd2 \{#2\}\kern-.625\wd2 \}%
	\else
	\mathopen{#1\{\kern-.7\wd0 #1\{}
	#2
	\mathclose{#1\}\kern-.7\wd0 #1\}}
	\fi
}

\begin{document}

\maketitle

\begin{abstract}
In the numerical simulation of ideal magnetohydrodynamics (MHD), keeping the pressure and density always positive is essential for both physical considerations and numerical stability. This is however a challenging task, due to the underlying relation between such positivity-preserving (PP) property and the magnetic divergence-free (DF) constraint as well as the strong nonlinearity of the MHD equations. In this paper, we present the first rigorous PP analysis of the central discontinuous Galerkin (CDG) methods and construct arbitrarily high-order provably PP CDG schemes for ideal MHD. By the recently developed geometric quasilinearization (GQL) approach, our analysis reveals that the PP property of standard CDG methods is closely related to a discrete magnetic DF condition, whose form was yet unknown prior to our analysis and differs from that for the non-central DG and finite volume methods in [K.~Wu, {\em SIAM J. Numer. Anal.}, 56 (2018), pp.~2124--2147]. The discovery of this relation lays the foundation for the design of our PP CDG schemes. In the 1D case, the discrete DF condition is naturally satisfied, and we rigorously prove that the standard CDG method is PP under a condition that can be enforced easily with an existing PP limiter. However, in the multidimensional cases, the corresponding discrete DF condition is highly nontrivial yet critical, and we analytically prove that the standard CDG method, even with the PP limiter, is not PP in general, as it generally fails to meet the discrete DF condition. We address this issue by carefully analyzing the structure of the discrete divergence terms and then constructing new locally DF CDG schemes for Godunov’s modified MHD equations with an additional source term. The key point is to find out the suitable discretization of the source term such that it exactly cancels out all the terms in the discovered discrete DF condition. Based on the GQL approach, we prove in theory the PP property of the new multidimensional CDG schemes under a CFL condition. The robustness and accuracy of the proposed PP CDG schemes are further validated by several demanding numerical MHD examples, including the high-speed jets and blast problems with very low plasma beta. 


\end{abstract}



\section{Introduction}
This paper is devoted to exploring robust high-order numerical methods for simulating the 
compressible ideal  magnetohydrodynamics (MHD), which has wide applications in plasma physics, 
astrophysics, and space physics. 
Let $\rho$, ${\bf m}$, and $E$ denote the fluid density, momentum vector, and total energy, respectively. 
Denote the magnetic field by ${\bf B}=(B_1,B_2,B_3)$. 
The mathematical equations that govern ideal MHD can be formulated as   
\begin{equation}\label{eq:MHD}
	\partial_t {\bf U} +  \nabla \cdot {\bf F} ( {\bf U} ) = {\bf 0}, 
\end{equation}	
where $\nabla \cdot {\bf F} = \sum_{i=1}^{d} \frac{\partial {\bf F}_{i}}{\partial x_{i}}$ with $d$ being the spatial dimensionality, and the conservative vector and fluxes are 
\begin{equation*}
		\quad {\bf U}= 	
	\begin{pmatrix}
		\rho  
		\\
		{\bf m}
		\\
		{\bf B}
		\\
		E	
	\end{pmatrix}, 
	\qquad 
	{\bf F}_i ({\bf U})= 	
	\begin{pmatrix}
	\rho v_i
	\\
	v_i {\bf m} - B_i {\bf B} +  
	+ \left( p + \frac12 | {\bf B} |^2  \right) {\bf e}_i
	\\
	v_i {\bf B} - B_i {\bf v}
	\\
	v_i \left( E + p + \frac12 | {\bf B} |^2 \right) 
	- B_i ( {\bf v} \cdot {\bf B} )  
\end{pmatrix}. 	
\end{equation*}
Here 
${\bf v}=(v_1,v_2,v_3)={\bf m}/\rho$ denotes the fluid velocity, 
$p$ is the thermal pressure, and ${\bf e}_i$ is the $i$th column of the $3\times 3$ identity matrix.  
The total energy consists of the kinetic, magnetic, and internal energies, 
namely,
$E= \frac12 ( \rho |{\bm v} |^2 + | {\bf B} |^2 ) + \rho e$, where $e$ is the specific internal energy. 
The equations \eqref{eq:MHD} are closed by an equation of state (EOS), which relates the thermodynamic
variables in the following general form 
\begin{equation}\label{eq:EOS}
	p = p(\rho, e). 
\end{equation}
The classical EOS for ideal gases is $p=(\gamma - 1)\rho e$, where $\gamma >1$ is a constant denoting the adiabatic index.  
The ideal MHD equations \eqref{eq:MHD} with \eqref{eq:EOS} are a nonlinear system of hyperbolic conservation laws, whose solutions may contain discontinuities such as shocks even if the initial data is smooth. 
This renders it difficult to simulate ideal compressible MHD flows accurately and robustly. 

The magnetic field $\bf B$ should satisfy an extra {\em divergence-free (DF)} constraint:   
\begin{equation}\label{eq:divergence-free}
	\nabla \cdot {\bf B} := \sum_{i=1}^{d} \frac{\partial B_{i}}{\partial x_{i}}  = 0, %
\end{equation} 
which describes the physical principle of non-existence of magnetic monopoles. 
Although not explicitly included in the MHD equations \eqref{eq:MHD}, 
the DF constraint \cref{eq:divergence-free} is automatically preserved by   
the exact solution of \eqref{eq:MHD} for all $t>0$ if the initial condition at $t=0$ satisfies \eqref{eq:divergence-free}. 
Numerically, it is important to carefully respect this constraint, because serious violation of \eqref{eq:divergence-free} may cause numerical instability and/or nonphysical structures in the approximated solutions (cf.~\cite{Brackbill1980,Evans1988,BalsaraSpicer1999,Toth2000,Li2005}). In the 1D case ($d=1$), the constraint 
\eqref{eq:divergence-free} and the fifth equation 
of \eqref{eq:MHD} become $\partial_{x_1} B_1=0=\partial_t B_1$, namely, $B_1$ is a constant, which can be easily preserved in the numerical simulation. 
However, in the multidimensional cases ($d\ge 2$), 
it is very difficult to exactly preserve \eqref{eq:divergence-free} in the numerical design. 
To address this need, 
researchers have developed various numerical techniques to reduce the divergence errors  
or explicitly enforce some approximate DF conditions at the discrete level; see, for example, 
\cite{Brackbill1980,Evans1988,Powell1994,Powell1995,BalsaraSpicer1999,Dedner2002,Li2005,Li2011,Li2012,Yakovlev2013,XuLiu2016,Fu2018}, the early survey article \cite{Toth2000}, and references therein. 
Among those techniques, 
the eight-wave approach \cite{Powell1994,Powell1995} 
 is based on suitably discretizing 
the Godunov's modified form \cite{Godunov1972} of the ideal MHD system  
\begin{equation}\label{eq:MHD:GP}
	{\bf U}_t + \nabla \cdot {\bf F} ({\bf U}) 
	= (- \nabla \cdot {\bf B} )~{\bf S} ( {\bf U} )
\end{equation} 
with an additional source term, where ${\bf S} ({\bf U}) := ( 0, {\bf B}, {\bf v}, {\bf v} \cdot {\bf B} )^\top.$ 
Notice that the source term in \eqref{eq:MHD:GP} is proportional 
to $\nabla \cdot {\bf B}$ and thus vanishes under the condition \eqref{eq:divergence-free}. 
This implies, 
for DF initial conditions, the exact solutions of 
the standard MHD system \eqref{eq:MHD} and the modified MHD system \eqref{eq:MHD:GP} 
{\em are the same}.  
In other words,  for DF initial conditions, the two forms \eqref{eq:MHD} and \eqref{eq:MHD:GP} are equivalent at the continuous level. 
However, the modified form \eqref{eq:MHD:GP} has the following advantages. 
As discovered by Godunov \cite{Godunov1972},  
the standard form \eqref{eq:MHD} of MHD is not symmetrizable, 
while the modified form \eqref{eq:MHD:GP} is the unique symmetrizable form for the ideal MHD system. 
Since the symmetrizable form \eqref{eq:MHD:GP} admits entropy pairs, it is useful for studying the entropy stability of numerical methods \cite{Chandrashekar2016,LiuShuZhang2017}. 
Moreover, Powell \cite{Powell1994} noticed that the standard form \eqref{eq:MHD} of MHD  
is incompletely hyperbolic and suggested to add the source term in \eqref{eq:MHD:GP} to recover the missing eigenvector. 
Although this non-conservative source term may lead to some drawbacks \cite{Toth2000}, 
Powell demonstrated that adding 
a proper discrete version of the source term could make the numerical schemes more stable to prevent the accumulation of divergence errors \cite{Powell1995}. 
Besides, the modified form \eqref{eq:MHD:GP} has another significant advantage in terms of positivity, which will be discussed later.

%




In addition to the DF constraint \cref{eq:divergence-free}, the solutions of the MHD equations \eqref{eq:MHD} should also satisfy several algebraic constraints on positivity:
\begin{equation}\label{eq:positive}
	\rho>0,\quad p>0,\quad e>0, 
\end{equation}
because these three quantities are positive in physics. As  in \cite{zhang2011} we assume that  
$p > 0 \Leftrightarrow e > 0$, which is satisfied by quite general EOS. 
For both physical considerations and robust computations, it is significant and essential to 
{\em positivity-preserving (PP)} numerical methods for system \eqref{eq:MHD}, which always keep the numerical solutions satisfying \eqref{eq:positive}. 
However, most numerical schemes for MHD are generally not PP and may produce negative density or pressure,  when simulating problems involving strong discontinuity, high march number, low internal energy, low density, low plasma beta, and/or strong magnetic field. 
As well-known, once the numerical density and/or pressure become negative, the hyperbolicity 
of the system is lost, causing serious numerical instability and the breakdown of the simulation. 
In fact, this issue also occurs in the pure hydrodynamic case ({\em i.e.}~the simulation of the compressible Euler equations), but gets much worse for MHD, due to the underlying influence of the magnetic divergence errors on the positivity. 
Over the past two decades, researchers have made some efforts to reduce such risk; see, for example, 
\cite{BalsaraSpicer1999a,Janhunen2000,Waagan2009,klingenberg2010relaxation,Balsara2012,cheng,Christlieb,Christlieb2016,ZouYuDai2019,liu2021new} and some recent works on provably PP schemes
\cite{Wu2017a,WuShu2018,WuShu2019,WuShu2020NumMath,zhang2021provably}. 
For the 1D ideal MHD equations, 
several PP multi-state approximate Riemann solvers were developed in \cite{Janhunen2000,MIYOSHI2005315,Bouchut2007,Bouchut2010}. 
Waagan proposed a positive second-order MUSCL--Hancock scheme \cite{Waagan2009} based on a PP linear reconstruction and the relaxation Riemann solvers of \cite{Bouchut2007,Bouchut2010}; see also \cite{klingenberg2010relaxation} for a review.  
Waagan, Federrath, and Klingenberg  \cite{waagan2011robust} 
 systematically demonstrated the robustness of that scheme by benchmark numerical tests, and they \cite{Waagan2009,waagan2011robust} noticed the importance of a stable discretization of the Powell type source term, which was added in only the magnetic induction equations and thus different from \eqref{eq:MHD:GP}. 
In recent years, researchers have made remarkable progress in constructing high-order PP or bound-preserving schemes for conservation laws; see, for example, 
 \cite{zhang2010,zhang2010b,zhang2011,Xu2014,WuTang2015,xiong2016parametrized,ZHANG2017301,Wu2017,WuTangM3AS,wu2021uniformly,wu2021minimum} and references therein. 
Christlieb et al.~\cite{Christlieb,Christlieb2016} proposed 
high-order PP finite difference schemes for ideal MHD, 
based on the parametrized flux limiters \cite{Xu2014,xiong2016parametrized} and the presumed positivity of the Lax--Friedrichs scheme (which was later rigorously proved in \cite{Wu2017a}). 
In high-order finite volume or discontinuous Galerkin (DG) schemes, it is well-known that 
the PP property may be lost in two cases: 
one case is that the reconstructed or DG solution polynomials fail to be positive, and the other is 
the cell averages evolved to the next time step become negative in the updating process; see the framework by Zhang and Shu  \cite{zhang2010,zhang2010b}. 
The positivity lost in the first case can be effectively recovered by a simple PP limiter; 
see, for example, the local scaling PP limiters \cite{cheng} for DG and central DG MHD schemes generalized from \cite{zhang2010,zhang2010b}, and the self-adjusting PP limiter \cite{Balsara2012}.  
However, it is very challenging to fully 
guarantee the positivity of the cell averages in the updating process, 
which is also critical to obtain a genuinely PP scheme.  
In fact, the validity of the PP limiters \cite{Balsara2012,cheng} relies on  
the positivity of the cell averages in the updating process, 
which was, however, not rigorously proved for the methods in \cite{Balsara2012,cheng}; it was formally shown for only the 1D methods in \cite{cheng} by invoking some assumptions on the exact Riemann solutions and also conjectured for the multi-dimensional methods of \cite{cheng}.  
As finite numerical tests might be insufficient to genuinely and fully demonstrate the PP property under all circumstances,  exploring {\em provably PP} schemes \cite{Wu2017a,WuShu2018,WuShu2019,WuShu2020NumMath} for MHD and developing the related mathematical theory become very important and highly desirable.

In a series of recent work \cite{Wu2017a,WuShu2018,WuShu2019}, high-order {\em provably PP} numerical methods were systematically developed for ideal MHD. 
Interestingly, it was discovered that the positivity preservation (which is an {\em algebraic} property) 
and the DF condition \cref{eq:divergence-free} (which is a {\em differential} constraint) 
are tightly linked, at both the discrete \cite{Wu2017a} and continuous levels \cite{WuShu2018}. 
More specifically, the theoretical analysis in \cite{Wu2017a} first showed, for the regular (non-central) DG and finite volume schemes of the standard MHD system \eqref{eq:MHD}, that their PP property is closely connected with a discrete DF condition. 
Moreover, slightly violating the discrete DF condition may lose the PP property of cell averages in the updating process \cite{Wu2017a}. 
On the other hand, it was shown in \cite[Appendix A]{WuShu2018} that if the continuous DF constraint \cref{eq:divergence-free} is slightly violated, {even the exact smooth solutions} 
of the standard MHD system \eqref{eq:MHD} may fail to be PP. 
Fortunately, the modified MHD system \eqref{eq:MHD:GP} does not suffer from this issue \cite{WuShu2019}, namely, its exact solutions are always PP, no matter whether the DF condition \cref{eq:divergence-free} is satisfied or not. 
Inspired by these findings, high-order accurate {\em provably PP} schemes were studied for ideal MHD within the (non-central) DG and finite volume frameworks via the standard form \eqref{eq:MHD} \cite{Wu2017a} and in the multidimensional cases \cite{WuShu2018,WuShu2019} via the modified form  \eqref{eq:MHD:GP}. See also some recent extensions and applications in \cite{WuShu2020NumMath,liu2021new,zhang2021provably}.

This paper aims to explore and rigorously analyze high-order provably PP schemes for ideal MHD in the central DG (CDG) framework. 
It is a sequel to the previous effort \cite{Wu2017a,WuShu2018,WuShu2019} on the non-central DG methods. 
The CDG method was originally introduced in \cite{liu2007central}, as a variant of the DG method  \cite{cockburn1998runge} 
to the central scheme framework \cite{nessyahu1990non,liu2005central}. 
Different from the regular DG method \cite{cockburn1998runge}, the CDG method 
evolves two copies of numerical solutions on two sets of overlapping meshes (namely, the primal mesh and its dual mesh), thereby possessing the distinct 
advantage of avoiding the use of any exact or approximate Riemann solvers, which can be computationally expensive for complicated systems such as MHD. 
Another advantage is that the CDG method allows much larger time step-sizes \cite{li2015operator}. 
It is also worth mentioning that 
Li {et al.}~\cite{Li2011,Li2012} systematically proposed a novel CDG method which exactly maintains the globally DF property of 
the numerical magnetic field for ideal MHD; see also \cite{YangLi2016,Fu2018} for more related works. 
Recently, bound-preserving CDG schemes were constructed for
the scalar conservation laws and the Euler equations \cite{li2016maximum}, the shallow water equations \cite{li2017CDGWB}, and the relativistic hydrodynamics \cite{WuTang2017ApJS}. 
Although the PP limiter \cite{zhang2010b,li2016maximum} was extended to the CDG methods for ideal MHD
in \cite{cheng}, the validity of the PP limiter \cite{cheng} is based on the 
positivity of the cell averages in the updating process, which was, however, not 
rigorously proved but was formally shown in only the 1D case \cite{cheng} by invoking some assumptions on the exact Riemann solutions. 
{\em The rigorous PP property of the CDG methods for MHD is still unclear in theory, especially in the multidimensional cases.}  
It is natural and interesting to ask the following important questions: 
\begin{tcolorbox}[colback=blue!2!white,colframe=blue!50!black]
Is the PP property of the CDG methods on overlapping meshes for ideal MHD also related to some discrete DF conditions? If so, what is the corresponding discrete DF conditions in the CDG case? In theory, how to establish the relation for the CDG schemes? 
\end{tcolorbox}
All of these questions have no answers yet. 
This paper will settle these questions by rigorous theoretical analysis, which further leads to our provably PP CDG schemes for ideal MHD. 
Specifically, the main efforts and findings in this work include:
\begin{itemize}
	\item We present the first rigorous PP analysis of the standard CDG methods for the MHD equations \eqref{eq:MHD}. The analysis is based on the geometric quasilinearization (GQL) approach, which was proposed in \cite{Wu2017a} with its general framework established in \cite{WuShu2021GQL}. 
	Our new analysis establishes the theoretical relation between the PP property of the CDG method and a discrete DF condition, which distinctly differs from that of the non-central DG and finite volume methods in \cite{Wu2017a}. 
	This finding lays the foundation for the design of our provably PP CDG schemes. 
	\item In the 1D case, the discrete DF 
	condition is naturally satisfied, and we rigorously prove that the   
	standard CDG method is PP under a condition on the CDG polynomials. This condition can be simply enforced by an existing PP limiter \cite{cheng}. 
	\item In the 2D case, however, the corresponding discrete DF condition becomes highly nontrivial, and we prove by an analytical counterexample that  
	the standard CDG method for \eqref{eq:MHD}, even with the PP limiter, is not PP in general, 
	as it may fail to meet the discrete DF condition. 
	\item By studying the structure of the 2D discrete DF condition, 
	we further construct a new 2D locally DF CDG method based on the modified MHD equations \eqref{eq:MHD:GP}. 
	The key point is to  
	carefully discretize the extra source term in \eqref{eq:MHD:GP} to exactly control the effect of nonzero divergence on the PP property. 
	Based on the GQL approach, we rigorously prove in theory 
	the positivity of the new 2D CDG schemes under a CFL condition. 
	The new CDG schemes carry many features of the standard CDG method, e.g., avoiding the use of any Riemann solvers and being uniformly high-order accurate and of high resolution. 
	\item We implement the proposed provably PP CDG schemes and demonstrate their  
	robustness and accuracy by 
	several demanding numerical MHD examples, including the high-speed jets and bast problems of very low plasma beta. 
\end{itemize}
It is worth noting that the analysis and design of our PP CDG schemes have distinct difficulties different from the regular DG case \cite{Wu2017a,WuShu2018} or other hyperbolic systems \cite{zhang2010b,li2016maximum}. 
One key difficulty in our quest is to analytically establish the intrinsic relation between the PP property 
and discrete DF condition on 2D overlapping meshes, {\em whose form remained unknown prior to our analysis and is very different from the non-central DG case.} 
Due to the relation, the states involved in CDG schemes are intrinsically coupled by the discrete DF condition, making the PP analysis very nontrivial. Consequently, some standard PP techniques, which typically rely on reformulating a multidimensional scheme into convex combination of formal 1D PP schemes \cite{zhang2010b,li2016maximum}, are inapplicable in our multidimensional MHD cases. 
Another new challenge in this work is to find out the suitable discretization of the source term in \eqref{eq:MHD:GP} such that it exactly offsets the divergence terms in the discovered discrete DF condition.  Our novel source term discretization in the CDG framework is based on the information from the corresponding dual mesh and distinctly different from the non-central DG case \cite{WuShu2018}.

The paper is organized as follows: We review the GQL approach and some auxiliary theoretical results in \cref{sec:GQL}. \Cref{section:pp1d,section:pp2d} present the rigorous PP analysis of the standard 
CDG schemes 
in 1D and 2D, respectively. 
The provably PP, locally DF 2D CDG schemes are constructed and analyzed in 
\cref{section:newpp2d}. 
The 3D extension is straightforward and omitted in this paper.  
\Cref{section:numerical-example} gives numerical examples to verify the PP property, robustness, and effectiveness of our schemes, before concluding the paper in \cref{section:conclusion}.

\section{Admissible state set and geometric quasilinearization}\label{sec:GQL} 
This section briefly reviews the GQL approach \cite{Wu2017a,WuShu2021GQL} and a few related results in the MHD case, which will be useful in the PP analysis. 

The positivity constraints \eqref{eq:positive} demand that 
the conservative 
vector ${\bf U}$ must belong to the following physically {\em admissible state set}
\begin{equation}\label{set:admissible}
	G = \left\{ {\bf U} = (\rho, {\bf m}, {\bf B}, E)^{\top}:~  \rho > 0, ~~
	\mathcal{E}({\bf U}) > 0 \right\}, 
\end{equation}
which is a {\em convex} set \cite{cheng}, with 
$	\mathcal{E}({\bf U}):=E - \frac{|{\bf m}|^2}{2\rho} - \frac{|{\bf B}|^2}{2}=\rho e.	$

A numerical scheme for \eqref{eq:MHD} is called PP if it {\em always preserves} the numerical solutions in the set $G$. 
From \eqref{set:admissible}, we can see that 
it is more difficult to preserve 
the positivity of $\mathcal{E}({\bf U})$, 
which is a highly nonlinear function depending on all the conservative quantities $\{\rho, {\bf m}, {\bf B}, E\}$. 
In a typical scheme for \eqref{eq:MHD}, 
$\{\rho, {\bf m}, {\bf B}, E\}$ are themselves evolved via 
their own conservation laws, which are seemingly 
independent of each other. 
As such, it may not always guarantee the positivity of $\mathcal{E}({\bf U})$ due to numerical errors, especially when the kinetic or/and magnetic energies are huge and very close to the total energy. 
In order to analyze the PP property of a numerical scheme, one should substitute all 
the discrete evolution equations of $\{\rho, {\bf m}, {\bf B}, E\}$ into 
the highly nonlinear function $\mathcal{E}({\bf U})$, and then analytically check whether the resulting $\mathcal{E}$ is positive or not.

To overcome the difficulties arising from the nonlinearity of $\mathcal{E}({\bf U})$, we introduce an {\em equivalent linear} representation of $G$, which 
skillfully transfers the intractable {\em nonlinear} constraint $\mathcal{E}({\bf U})>0$ into {\em linear} ones.

\begin{lemma}[GQL representation \cite{Wu2017a}]
	\label{theo:eqDefG}
	The admissible state set $G$ is exactly equivalent to
	\begin{equation}\label{eq:newDefG}
		G_{\ast} = \left\{   {\bf U} = (\rho,{\bf m},{\bf B},E)^\top:~ {\bf U}\cdot {\bf n}_1 > 0,~~~
		{\bf U} \cdot {\bf n}^{\ast} + \frac{|{\bf B}^{\ast}|^2}{2} > 0~~\forall \, {\bf v}^{\ast}, {\bf B}^{\ast} \in {\mathbb {R}}^3 \right\},
	\end{equation}
	where ${\bf n}_1:=(1,0,\dots,0)^\top$, the extra variables $\{{\bf v}^{\ast}, {\bf B}^{\ast}\}$ are called {\bf {\em free auxiliary variables}}, and ${\bf n}^{\ast}$ is a function of only the free auxiliary variables:   
	\begin{equation}\label{eq:DefN}
		{\bf n}^{\ast} := \bigg( \frac{|{\bf v}^{\ast}|^2}2,~- {\bf v}^{\ast},~-{\bf B}^{\ast},~1 \bigg)^\top.
	\end{equation} 
\end{lemma}
A proof of \cref{theo:eqDefG} was first given in \cite{Wu2017a}, and its geometric interpretation was presented in \cite{WuShu2021GQL}. 
Notice that in the equivalent form \eqref{eq:newDefG}, all the constraints become {\bf linear} with respect to $\bf U$, yielding a highly effective way to theoretically study the positive numerical MHD schemes. 
Such an equivalent linear form is called {\em GQL representation}, and can be derived for general convex sets within the GQL framework \cite{WuShu2021GQL}. 
The GQL representation \eqref{eq:newDefG} will be a crucial tool in our PP analysis and design.

Let us recall the following inequality \cref{eq:MHD:LLFsplit}, which was constructed in \cite{Wu2017a} and will be useful for the PP analysis based on the GQL approach.    %

\begin{lemma}[\cite{Wu2017a}]\label{theo:MHD:LLFsplit}
	 For any free auxiliary
	variables ${\bf v}^{\ast}, {\bf B}^{\ast} \in {\mathbb{R}}^3$ and any two admissible states ${\bf U}, \tilde{\bf U} \in G$, the following inequality
	\begin{equation}\label{eq:MHD:LLFsplit}
		\bigg( {\bf U} - \frac{ {\bf F}_i({\bf U})}{\alpha}
		+
		\tilde{\bf U} + \frac{ {\bf F}_i(\tilde{\bf U})}{\alpha}
		\bigg) \cdot {\bf n}^{\ast} + |{\bf B}^{\ast}|^2
		+  \frac{  B_i - \tilde B_i }{\alpha} ({\bf v}^{\ast} \cdot {\bf B}^{\ast})  > 0 ,
	\end{equation}
	holds if $\alpha > \alpha_{i} ({\bf U},\tilde{\bf U}) $, where $i\in\{1,2,3\}$, and 
	\begin{equation} \label{eq:alpha_i}
		\alpha_{i} ({\bf U},\tilde{\bf U})  =
		\max\bigg\{ |v_i|+ {\mathcal{C}}_i, |\tilde v_i| +  \tilde {\mathcal{C}}_i , \frac{| \sqrt{\rho} v_i + \sqrt{\tilde \rho} \tilde v_i |} {\sqrt{\rho}+\sqrt{\tilde \rho}} + \max\{ {\mathcal{C}}_i , \tilde {\mathcal{C}}_i \}  \bigg\}
		+ \frac{ |{\bf B}-\tilde{\bf B}| }{ \sqrt{\rho} + \sqrt{\tilde \rho}  }
	\end{equation}	
	with 
	\begin{align*}
		{\mathcal {C}}_i = \frac{1}{\sqrt{2}}  \left[   \frac{ |{\bf B}|^2}{\rho} + {\mathcal {C}}_s^2 + \sqrt{ \left(   \frac{ |{\bf B}|^2}{\rho} + {\mathcal {C}}_s^2 \right)^2 - 4 \frac{ B_i^2{\mathcal {C}}_s^2 }{\rho}  } \right]^\frac12, \qquad {\mathcal {C}}_s=\frac{p}{\rho \sqrt{2e}}.
	\end{align*}
\end{lemma}

\begin{remark}
Let  ${\mathscr{R}}_i ({\bf U})$ be the spectral radius of the Jacobian matrix, in the $x_i$-direction, $i=1,2,3$,  
of the MHD equations \eqref{eq:MHD:GP}. 
For the ideal EOS $p=(\gamma -1) \rho e$, it was derived in \cite{Powell1994} that 
$$
{\mathscr{R}}_i ({\bf U}) = |v_i| + \frac{1}{\sqrt{2}} \left[   \frac{ |{\bf B}|^2}{\rho} + c^2_s + \sqrt{ \left(  \frac{ |{\bf B}|^2}{\rho} + c^2_s \right)^2 - 4 \frac{  B_i^2c_s^2}{\rho}  }  \right]^\frac12
$$
where $c_s=\sqrt{\gamma p/\rho}$ is the sound speed. 
Let  $\alpha_i^{\rm std} := \max\{ {\mathscr{R}}_i ({\bf U}), {\mathscr{R}}_i(\tilde{\bf U}) \}$. It was shown in \cite{Wu2017a} that 
\begin{equation} \label{eq:aaaaWKL}
	\alpha_{i} ({\bf U},\tilde{\bf U}) 
	\le \alpha_i^{\rm std} + 
	{\mathcal O} \big( |{\bf U} -\tilde {\bf U}| \big). 
\end{equation}
\end{remark}

\begin{remark}
	In the PP analysis of many other hyperbolic systems (see, e.g., \cite{zhang2010b,zhang2011,WuTang2015,Wu2017}), 
	one usually expects the following property for any ${\bf U}\in G$,
	\begin{equation}\label{eq:LFproperty}
		{\bf U} \pm \frac{ {\bf F}_i ( {\bf U} ) }{\alpha} \in G \qquad \text{with}~~ \alpha \ge {\mathscr{R}}_i ({\bf U}). 
	\end{equation}
	If true, this property would imply $\frac12 \big( {\bf U} - \frac{ {\bf F}_i({\bf U})}{\alpha}
	+ \tilde{\bf U} + \frac{ {\bf F}_i(\tilde{\bf U})}{\alpha}
	\big) \in G$ for $\alpha \ge  \alpha_i^{\rm std}$ and then by \eqref{eq:newDefG} would lead to  
\begin{equation}\label{notrue}
		 \bigg( {\bf U} - \frac{ {\bf F}_i({\bf U})}{\alpha}
	 +
	 \tilde{\bf U} + \frac{ {\bf F}_i(\tilde{\bf U})}{\alpha}
	 \bigg) \cdot {\bf n}^{\ast} + |{\bf B}^{\ast}|^2 >0.
\end{equation}
	 {\em Unfortunately for the MHD system, the usually-expected property \eqref{eq:LFproperty} {\bf {\em is  not true}} in general}, even if the condition $\alpha \ge {\mathscr{R}}_i$ is replaced with $\alpha \ge \chi {\mathscr{R}}_i$ for any given constant $\chi \ge 1$; see a proof in \cite[Proposition 2.5]{Wu2017a}. 
	Therefore, the PP analysis of numerical MHD schemes has distinct challenges significantly different from that for other hyperbolic systems such as the Euler equations \cite{zhang2010b,li2016maximum}.
\end{remark}

\begin{remark}
	As 	\eqref{eq:LFproperty}, the resulting inequality \eqref{notrue} is also invalid in general \cite{Wu2017a}. 
	Different from \eqref{notrue}, the correct inequality \eqref{eq:MHD:LLFsplit} in \cref{theo:MHD:LLFsplit} 
	has an extra term $\frac{  B_i - \tilde B_i }{\alpha} ({\bf v}^{\ast} \cdot {\bf B}^{\ast})$, which is essential and critical.  
		Without this term the inequality \eqref{eq:MHD:LLFsplit} would reduce to \eqref{notrue} and become incorrect. This term is not always positive but helps offset the ``possible negativity'' of \eqref{notrue}. 
		More importantly, this technical term will be canceled out skillfully under a discrete DF condition, 
		and it will be a key to establish the intrinsic relation of the PP property 
		to the discrete DF condition. 	
\end{remark}

\section{Rigorous PP analysis of 1D standard CDG method}\label{section:pp1d}
In this section, we apply the GQL approach to rigorously analyze the positivity of the standard CDG method for the 1D MHD equations.  
In the 1D case, the DF constraint \eqref{eq:divergence-free} simply reduces to that $B_1$ is a constant, denoted by 
${\rm B}_{\rm const}$. The 1D analysis is fairly trivial compared to the multidimensional cases, 
but it may help us to gain some insights. 

For convenience, we employ the symbol $x$ to represent the 1D spatial coordinate variable. 
The spatial domain $\Omega$ is uniformly divided into 
$\{ I_{j} := (x_{j-\frac{1}{2}}, x_{j+\frac{1}{2}}) \}$ 
with constant stepsize $\Delta x = x_{j+\frac{1}{2}} - x_{j-\frac{1}{2}}$. 
We denote $x_{j} = \frac{1}{2}(x_{j-\frac{1}{2}}+x_{j+\frac{1}{2}})$, then  
$\{ I_{j+\frac{1}{2}} := (x_{j}, x_{j+1})\}$ forms a dual partition. 
Define 
\begin{equation*}
	\mathbb{V}_{h}^{C,k} = \left\{ {\bf w} \in [L^2(\Omega)]^8: w_\ell |_{I_{j}} \in \mathbb{P}^{k}(I_{j}) ~\forall j,\ell \right\}, \quad
	\mathbb{V}_{h}^{D,k} = \left\{ {\bf u} \in [L^2(\Omega)]^8: u_\ell |_{I_{j+\frac{1}{2}}} \in \mathbb{P}^{k}(I_{j+\frac{1}{2}}) ~\forall j,\ell \right\},
\end{equation*}
where $\mathbb{P}^{k}(I)$ denote the space of the polynomials with degree less than or equal to $k$ on the cells $I$. 
The standard semi-discrete CDG method 
seeks the numerical solutions ${\bf U}_{h}^{C} \in \mathbb{V}_{h}^{C,k}$ and ${\bf U}_{h}^{D} \in \mathbb{V}_{h}^{D,k}$ such that for any test functions ${\bf w} \in \mathbb{V}_{h}^{C,k}$ and ${\bf u} \in \mathbb{V}_{h}^{D,k}$, 
\begin{align}\label{eq:CDG-1d-primal}
\begin{aligned}
		\int_{I_{j}} \frac{\partial {\bf U}_{h}^{C}}{\partial t} \cdot {\bf w} {\rm d}x
	& = \frac{1}{\tau_{\max}}\int_{I_{j}} ( {\bf U}_{h}^{D} - {\bf U}_{h}^{C}) \cdot {\bf w} {\rm d}x  
	+ \int_{I_{j}} {\bf F}_1 ({\bf U}_{h}^{D})\cdot \partial_{x}  {\bf w} {\rm d}x    \\ 
	& \quad + {\bf F}_1 ({\bf U}_{h}^{D}(x_{j-\frac{1}{2}})) \cdot {\bf w}(x_{j-\frac{1}{2}}^{+}) -  {\bf F}_1 ({\bf U}_{h}^{D}(x_{j+\frac{1}{2}})) \cdot {\bf w}(x_{j+\frac{1}{2}}^{-}), 
\end{aligned}
\\ \label{eq:CDG-1d-dual}
\begin{aligned} 
		\int_{I_{j+\frac{1}{2}}} \frac{\partial {\bf U}_{h}^{D}}{\partial t} \cdot {\bf u} {\rm d}x
	& = \frac{1}{\tau_{\max}}\int_{I_{j+\frac{1}{2}}} ({\bf U}_{h}^{C} - {\bf U}_{h}^{D}) \cdot {\bf u} {\rm d}x  
	+ \int_{I_{j+\frac{1}{2}}} {\bf F}_1 ({\bf U}_{h}^{C})\cdot \partial_{x}   {\bf u} {\rm d}x   \\ 
	& \quad  + {\bf F}_1 ({\bf U}_{h}^{C}(x_{j},t)) \cdot {\bf u}(x_{j}^{+}) -  {\bf F}_1 ({\bf U}_{h}^{C}(x_{j+1},t)) \cdot {\bf u}(x_{j+1}^{-}). 
\end{aligned}
\end{align}
Here $\tau_{\max}$ is the maximum time stepsize determined by certain CFL condition, which will be specified in the PP analysis, and $f(x^{\pm}) = \lim_{\epsilon \to 0^+}f(x\pm\epsilon)$ denotes 
the limits at $x$ from the left or the right side.

Based on Zhang-Shu's framework \cite{zhang2010b}, to achieve a PP high-order scheme, the main task is to preserve the evolved cell averages in the set $G$ during the updating process. 
Once such a property is guaranteed, one can then use a simple PP limiter to enforce the PP property of the solution polynomials at any specified points. Denote 
$\overline{{\bf U}}_{j}^{C}(t) := \frac{1}{\Delta x}  \int_{I_{j}} {\bf U}_{h}^{C}(x,t) \mathrm{d}x$ and    
$\overline{{\bf U}}_{j+\frac{1}{2}}^{D}(t) := \frac{1}{\Delta x}  \int_{I_{j+\frac{1}{2}}} {\bf U}_{h}^{D}(x,t) \mathrm{d}x.
$
Taking ${\bf w} = {\bf 1}$ in \eqref{eq:CDG-1d-primal} and ${\bf u} = {\bf 1}$ in \eqref{eq:CDG-1d-dual}, we obtain the semi-discrete scheme satisfied by the cell averages of the  CDG solution: 
\begin{align}\label{eq:CDG-average-primal}
		\frac{ {\rm d} \overline{{\bf U}}_{j}^{C} }{ {\rm d} t} &=  {\mathbfcal L}_{j} \big( {\bf U}_{h}^{C}, {\bf U}_{h}^{D} \big)
	: = \frac{ \overline{ {\bf U} }_{j}^{D} - \overline{ {\bf U} }_{j}^{C} }{\tau_{\max}}   
	- \frac{ {\bf F}_1 ({\bf U}_{h}^{D}(x_{j+\frac{1}{2}})) - {\bf F}_1 ({\bf U}_{h}^{D}(x_{j-\frac{1}{2}}))  }{\Delta x},
	\\
\label{eq:CDG-average-dual} 
	\frac{ {\rm d} \overline{{\bf U}}_{j+\frac{1}{2}}^{D} }{ {\rm d} t} &=  {\mathbfcal L}_{j+\frac{1}{2}} \big( {\bf U}_{h}^{D}, {\bf U}_{h}^{C} \big): = \frac{ \overline{ {\bf U} }_{j+\frac{1}{2}}^{C} - \overline{ {\bf U} }_{j+\frac{1}{2}}^{D} }{\tau_{\max}}   
	- \frac{{\bf F}_1 ({\bf U}_{h}^{C}(x_{j+1})) - {\bf F}_1 ({\bf U}_{h}^{C}(x_{j}))}{\Delta x},
\end{align}
where and below we omit the $t$ dependence of all quantities for convenience. The scheme \eqref{eq:CDG-average-primal}--\eqref{eq:CDG-average-dual} is desired to satisfy the following PP property 
	\begin{equation}\label{eq:PP-1D}
	\overline{ {\bf U} }_{j}^{C} + \Delta t  {\mathbfcal L}_{j} \big( {\bf U}_{h}^{C}, {\bf U}_{h}^{D} \big) \in G, \qquad
	\overline{ {\bf U} }_{j+\frac{1}{2}}^{D} + \Delta t  {\mathbfcal L}_{j+\frac{1}{2}} \big( {\bf U}_{h}^{D}, {\bf U}_{h}^{C} \big) \in G  \qquad  \forall j,
\end{equation} 
under certain suitable CFL condition on the time stepsize $\Delta t$ and some proper conditions on 
the CDG solution polynomials which can be accessible by the PP limiter. The property \eqref{eq:PP-1D}  guarantees the cell averages staying in $G$ during the updating process, if one uses a strong-stability-preserving (SSP) method for time discretization, which is a convex combination of the  forward Euler method.

We now use the GQL approach to derive a theoretical analysis on property \eqref{eq:PP-1D} for the cell-averaged CDG scheme  \eqref{eq:CDG-average-primal}--\eqref{eq:CDG-average-dual}. We only focus on the case $k\ge 1$, because 
when $k=0$ the scheme \eqref{eq:CDG-average-primal}--\eqref{eq:CDG-average-dual} reduces to a first-order Lax--Friedrichs-like scheme, whose PP property can be concluded from \cite{Wu2017a}. 
Let $ \{ \hat{x}_{j-\frac14}^{(\nu)} \}_{\nu = 1}^{L} $ 
and $ \{ \hat{x}_{j+\frac14}^{(\nu)} \}_{\nu = 1}^{L} $
be the Gauss--Lobatto quadrature nodes transformed into the intervals $[x_{j-\frac{1}{2}}, x_{j}]$ and  
$[x_{j}, x_{j+\frac{1}{2}}]$, respectively. 
Denote $\hat{\mathbb{Q}}_{j}^{x} = \{ \hat{x}_{j-\frac14}^{(\nu)} \}_{\nu = 1}^{L}  \cup \{ \hat{x}_{j+\frac14}^{(\nu)} \}_{\nu = 1}^{L}$. 
Let $\{\hat{\omega}_{\nu}\}_{\nu = 1}^{L}$
be the associated weights satisfying $\sum_{\nu=1}^{L} \hat{\omega}_{\nu} = 1$ and $\hat{\omega}_{1}=\hat{\omega}_{L}$. 
We take $L=\lceil \frac{k+3}2 \rceil$, which gives $2L-3\geq k$, so that
the $L$-point Gauss--Lobatto quadrature rule is exact for polynomials of degree up to $k$. This implies 
	\begin{align}\label{eq:GLexactness}
	\overline{ {\bf U} }_{j}^{D}  &=  \frac{1}{\Delta x} \Big(
   \int_{x_{j-\frac12}}^{x_j} {\bf U}_{h}^{D} \mathrm{d}x +  \int_{x_{j}}^{x_{j+\frac12}} {\bf U}_{h}^{D} \mathrm{d}x \Big)  = 
	\sum_{ \sigma = \pm 1 } \sum_{\nu=1}^{L} \frac{\hat{\omega}_{\nu}}{2} {\bf U}_{h}^{D}( \hat{x}_{j + \frac{\sigma}4  }^{(\nu)} ) 
	= \frac{\hat{\omega}_{1}}{2} \left( {\bf U}_{j-\frac{1}{2}}^{D} + 
	{\bf U}_{j+\frac{1}{2}}^{D} \right) +{\bf \Pi}_j^D
\end{align}
with ${\bf U}_{j\pm\frac{1}{2}}^{D}:= {\bf U}_{h}^{D}(x_{j\pm\frac12})$ and ${\bf \Pi}_j^D:= \sum_{\nu=2}^{L} \frac{\hat{\omega}_{\nu}}{2} {\bf U}_{h}^{D}( \hat{x}_{j-\frac14}^{(\nu)} )  +
\sum_{\nu=1}^{L-1} \frac{\hat{\omega}_{\nu}}{2} {\bf U}_{h}^{D}( \hat{x}_{j+\frac14}^{(\nu)} )$.

\begin{theorem}[PP property of 1D standard CDG method]\label{thm:1DCDG}
	Assume that $\overline{ {\bf U} }_{j}^{C}, \overline{ {\bf U} }_{j+\frac12}^{D} \in G$ for all $j$ and the numerical solutions ${\bf U}_{h}^{C}(x)$ and ${\bf U}_{h}^{D}(x)$ satisfy
	\begin{align}\label{pp-condition}
		{\bf U}_{h}^{C}(x)  \in {G} \,, \quad {\bf U}_{h}^{D}(x)  \in {G}  \qquad  
		\forall  x \in \mathop{\cup}_j  \hat{\mathbb{Q}}_{j}^{x} , 
		\\ \label{eq:1D-DDF}
		B_{1,h}^{D}(x_{j\pm\frac{1}{2}})  = {\rm B_{const}} = B_{1,h}^{C}(x_{j\pm 1})  \qquad \forall j \,,
	\end{align}
	then the PP property \eqref{eq:PP-1D} holds under the CFL condition
\begin{equation}\label{eq:1DCFL}
		a_1 \frac{\Delta t}{\Delta x}   < \frac{\theta \hat{\omega}_{1}}{2} \,, \qquad
	\theta := \frac{\Delta t}{\tau_{\max}} \in (0,1] \,,
\end{equation}
	where  $a_1 :=  \max_{j} \{ \alpha_{1} \big( {\bf U}_{h}^{C}(x_{j+1}), {\bf U}_{h}^{C}(x_{j})\big), \alpha_{1} \big({\bf U}_{h}^{D}(x_{j+\frac{1}{2}}), {\bf U}_{h}^{D}(x_{j-\frac{1}{2}}) \big)\} $. 
\end{theorem}

\begin{proof}
	Denote ${\bf U}_{\Delta t}^{C}:=\overline{{\bf U}}_{j}^{C} + \Delta t  {\mathbfcal L}_{j} \big( {\bf U}_{h}^{C}, {\bf U}_{h}^{D} \big)$. Thanks to \eqref{eq:GLexactness}, we have 
	\begin{equation}\label{eq:Uc}
		{\bf U}_{\Delta t}^{C} = (1-\theta)\overline{ {\bf U} }_{j}^{C} +  \theta {\bf \Pi}_j^D +  
		\frac{ \theta \hat{\omega}_{1}}{2} \left( 
		{\bf U}_{j+\frac{1}{2}}^{D} + {\bf U}_{j-\frac{1}{2}}^{D}  \right)
		- \frac{\Delta t}{\Delta x}\Big( {\bf F}_1( {\bf U}_{j+\frac{1}{2}}^{D} ) - {\bf F}_1({\bf U}_{j-\frac{1}{2}}^{D} )\Big).
	\end{equation}
    The condition \eqref{pp-condition} implies that ${\bf U}_{j\pm\frac{1}{2}}^{D} \in G$ and $\frac{1}{1-\hat \omega_1}{\bf \Pi}_j^D \in G$. It follows that 
	\begin{align}\notag
		{\bf U}_{\Delta t}^{C} \cdot {\bf n}_1
		= (1-\theta)\overline{\bf U}_{j}^{C} \cdot {\bf n}_1 +   \theta {\bf \Pi}_j^D  \cdot {\bf n}_1 
		 +   
		\big(\frac{\theta\hat{\omega}_{1}}{2} - \frac{\Delta t}{\Delta x} v_{1,j+\frac{1}{2}}^{D} \big)  \rho_{j+\frac{1}{2}}^{D}  + \big(\frac{\theta\hat{\omega}_{1}}{2} 
		+ \frac{\Delta t}{\Delta x}v_{1,j-\frac{1}{2}}^{D} \big)  \rho_{j-\frac{1}{2}}^{D} > 0,  
	\end{align}
	where the condition \eqref{eq:1DCFL} is used in the inequality. 
	Define $\alpha := \frac{\theta \hat{\omega}_{1}}{2} \cdot \frac{\Delta x}{\Delta t}$. We can rewrite \eqref{eq:Uc} as 
	\begin{equation}\label{eq:Uc2}
	{\bf U}_{\Delta t}^{C} = (1-\theta)\overline{ {\bf U} }_{j}^{C} +  \theta {\bf \Pi}_j^D +  
		\frac{ \theta \hat{\omega}_{1}}{2} \left( 
	{\bf U}_{j+\frac{1}{2}}^{D} 
	-\frac{ {\bf F}_1( {\bf U}_{j+\frac{1}{2}}^{D} ) }{\alpha} 
	+ {\bf U}_{j-\frac{1}{2}}^{D} + 
	\frac{ {\bf F}_1({\bf U}_{j-\frac{1}{2}}^{D} ) }{ \alpha }  \right).
\end{equation}	
	Note that the condition \eqref{eq:1DCFL} yields $\alpha > a_1 \ge \alpha_{1} ( {\bf U}_{j+\frac12}^D, {\bf U}_{j-\frac12}^D )$. 
	Thanks to \cref{theo:MHD:LLFsplit}, we have for any free auxiliary variables ${\bf v}^{\ast},{\bf B}^{\ast} \in \mathbb R^3$ that 
	\begin{align*} 
		{\bf U}_{\Delta t}^{C} \cdot {\bf n}^{\ast} + \frac{|{\bf B}^{\ast}|^2}{2} &  \overset{\mbox{\eqref{eq:Uc2}}}{=}
		(1-\theta) \left(  \overline{ {\bf U} }_{j}^{C}\cdot {\bf n}^{\ast} + \frac{|{\bf B}^{\ast}|^2}{2} \right) +  
		\theta  (1-\hat \omega_1) \left( \frac{1}{1-\hat \omega_1} {\bf \Pi}_j^D \cdot {\bf n}^{\ast} + \frac{|{\bf B}^{\ast}|^2}{2} \right)
		\\
		& \quad 
		+ \frac{ \theta \hat{\omega}_{1}}{2} \left[
		\left( 
		{\bf U}_{j+\frac{1}{2}}^{D} 
		-\frac{ {\bf F}_1( {\bf U}_{j+\frac{1}{2}}^{D} ) }{\alpha} 
		+ {\bf U}_{j-\frac{1}{2}}^{D} + 
		\frac{ {\bf F}_1({\bf U}_{j-\frac{1}{2}}^{D} ) }{ \alpha }  \right) \cdot {\bf n}^{\ast} + |{\bf B}^{\ast}|^2
		 \right]
		 \\
		 & \overset{\mbox{\eqref{eq:MHD:LLFsplit}}}{>} \frac{ \theta \hat{\omega}_{1}}{2} \cdot \frac{  B_{1,h}^D( x_{j-\frac12} ) - B_{1,h}^D( x_{j+\frac12} )  }{ \alpha } \overset{\mbox{\eqref{eq:1D-DDF}}} {=} 0.
	\end{align*}
	According to the GQL representation \eqref{eq:newDefG} in \cref{theo:eqDefG}, we obtain ${\bf U}_{\Delta t}^{C} 
	 \in {G_*}=G$. Similar arguments give $\overline{ {\bf U} }_{j+\frac{1}{2}}^{D} + \Delta t  {\mathbfcal L}_{j+\frac{1}{2}} ( {\bf U}_{h}^{D}, {\bf U}_{h}^{C} ) \in {G}$. 
	The proof is completed.
\end{proof}

\begin{remark}
	  \cref{thm:1DCDG} indicates that the PP property of the 1D CDG method 
	  is related to a discrete DF condition \eqref{eq:1D-DDF}, which is trivial and naturally satisfied. In fact, the 1D CDG method \eqref{eq:CDG-1d-primal}--\eqref{eq:CDG-1d-dual} 
	  automatically maintain the 1D globally DF property ${\bf B}_{1,h}^C(x)\equiv {\bf B}_{1,h}^D(x) \equiv {\rm B_{const}}$, because the fifth component of ${\bf F}_1({\bf U})$ is zero. 
	  	The condition \eqref{pp-condition} can be simply enforced by an existing PP limiter \cite{cheng} generalized from \cite{zhang2010,zhang2010b}. 
	Notice that the 1D globally DF property is not affected by the PP limiter.  
	As we will see, in the 2D case, 
	the related discrete DF condition is very different and highly nontrivial. 
\end{remark}

\section{Rigorous PP analysis of 2D standard CDG method}\label{section:pp2d}
In this section, we apply the GQL approach to rigorously analyze the positivity of the standard CDG method for the 2D MHD equations.  
Our analysis will reveal that the PP property is closely related to a discrete DF condition, which is very nontrivial and differs from that for the regular DG method in \cite{Wu2017a}. 
The extension of our analysis to 3D is quite straightforward and will be omitted in this paper. For convenience, we will employ the symbols $(x,y)$ to denote the 2D spatial coordinate variables.

Let $\{ I_{i,j}\}$ and $\{I_{i+\frac{1}{2}, j+\frac{1}{2}}\}$ 
 denote two overlapping uniform meshes for a rectangular domain
$\Omega = [x_{\min},x_{\max}] \times [y_{\min},y_{\max}]$ with $ I_{i,j} = (x_{i-\frac{1}{2}}, x_{i+\frac{1}{2}}) \times (y_{j-\frac{1}{2}}, y_{j+\frac{1}{2}}) $
and $ I_{i+\frac{1}{2}, j+\frac{1}{2}} = ( x_{i}, x_{i+1} ) \times ( y_{j}, y_{j+1} ) $.
The spatial stepsizes are constants, denoted by $\Delta x$ in the $x$-direction and $\Delta y$ in the $y$-direction. Define 
{\small\begin{equation*}
	\mathbb{V}_{h}^{C,k} = \left\{ {\bf w} \in [L^2(\Omega)]^8: w_\ell |_{I_{i,j}} \in \mathbb{P}^{k}(I_{i,j}) ~\forall i,j,\ell \right\},\quad 
	\mathbb{V}_{h}^{D,k} = \left\{ {\bf u} \in [L^2(\Omega)]^8: u_\ell |_{I_{i+\frac{1}{2},j+\frac{1}{2}}} \in \mathbb{P}^{k}(I_{i+\frac{1}{2},j+\frac{1}{2}}) ~\forall i,j,\ell \right\}
\end{equation*}}
with $\mathbb{P}^{k}(I)$ denoting the space of the 2D polynomials in $I$ with the total degree of at most $k$. 
The standard semi-discrete CDG method seeks the numerical solutions ${\bf U}_{h}^{C} \in \mathbb{V}_{h}^{C,k}$ and ${\bf U}_{h}^{D} \in \mathbb{V}_{h}^{D,k}$ such that 
\begin{align} \label{eq:CDG-2d-primal} 
	\int_{ I_{ij} } \frac{\partial {\bf U}_{h}^{C}}{\partial t} \cdot {\bf w} {\rm d}x {\rm d}y
	&= {\mathbfcal G}_{ij} \big( {\bf U}_{h}^{C}, {\bf U}_{h}^{D}, {\bf w} \big) \qquad \forall {\bf w} \in \mathbb{V}_{h}^{C,k}, 
	\\ \label{eq:CDG-2d-dual} 
	\int_{ I_{i+\frac{1}{2},j+\frac{1}{2}} } \frac{\partial {\bf U}_{h}^{D}}{\partial t} \cdot {\bf u} {\rm d}x {\rm d}y
	& = {\mathbfcal G}_{i+\frac12,j+\frac12} \big( {\bf U}_{h}^{D}, {\bf U}_{h}^{C}, {\bf u} \big) \qquad \forall  {\bf u} \in \mathbb{V}_{h}^{D,k}
\end{align}
with 
\begin{align} \nonumber
& {\mathbfcal G}_{ij} \big( {\bf U}_{h}^{C}, {\bf U}_{h}^{D}, {\bf w} \big) 
	:= 
	\frac{1}{\tau_{\max}}\int_{ I_{ij} } ( {\bf U}_{h}^{D} - {\bf U}_{h}^{C}) \cdot {\bf w} {\rm d}x {\rm d}y  
	+ \int_{I_{ij}} {\bf F}({\bf U}_{h}^{D}) \cdot \nabla {\bf w} {\rm d}x {\rm d}y \\ \nonumber
	& \qquad - \int_{y_{j-\frac{1}{2}}}^{y_{j+\frac{1}{2}}}  
	\Big({\bf F}_{1}({\bf U}_{h}^{D}(x_{i+\frac{1}{2}},y,t)) \cdot {\bf w}(x_{i+\frac{1}{2}}^{-},y) 
	-    {\bf F}_{1}({\bf U}_{h}^{D}(x_{i-\frac{1}{2}},y,t)) \cdot {\bf w}(x_{i-\frac{1}{2}}^{+},y) \Big)  {\rm d}y \\  \label{eq:2DCDG-primal-operator}
	& \qquad - \int_{x_{i-\frac{1}{2}}}^{x_{i+\frac{1}{2}}}  \Big({\bf F}_{2}({\bf U}_{h}^{D}(x, y_{j+\frac{1}{2}},t)) 
	\cdot {\bf w}(x, y_{j+\frac{1}{2}}^{-}) 
	- {\bf F}_{2}({\bf U}_{h}^{D}(x, y_{j-\frac{1}{2}},t)) \cdot {\bf w}(x, y_{j-\frac{1}{2}}^{+}) \Big) {\rm d}x,
	\\  \nonumber
	& {\mathbfcal G}_{i+\frac12,j+\frac12} \big( {\bf U}_{h}^{D}, {\bf U}_{h}^{C}, {\bf u} \big) 
	:=  \frac{1}{\tau_{\max}}\int_{ I_{i+\frac{1}{2},j+\frac{1}{2}} } ( {\bf U}_{h}^{C} - {\bf U}_{h}^{D}) \cdot {\bf u} {\rm d}x {\rm d}y + \int_{I_{i+\frac{1}{2},j+\frac{1}{2}}}  {\bf F}({\bf U}_{h}^{C}) \cdot \nabla {\bf u}  {\rm d}x {\rm d}y  \\ \nonumber
	& \qquad - \int_{y_{j}}^{y_{j+1}} \Big({\bf F}_{1}({\bf U}_{h}^{C}(x_{i+1},y,t))\cdot {\bf u}(x_{i+1}^{-},y)
	-    {\bf F}_{1}({\bf U}_{h}^{C}(x_{i},y,t))\cdot {\bf u}(x_{i}^{+},y) \Big)  {\rm d}y 
	\\ \label{eq:2DCDG-dual-operator}
	& \qquad  - \int_{x_{i}}^{x_{i+1}} \Big({\bf F}_{2}({\bf U}_{h}^{C}(x, y_{j+1},t)) \cdot {\bf u}(x, y_{j+1}^{-})
	- {\bf F}_{2}({\bf U}_{h}^{C}(x, y_{j},t )) \cdot {\bf u}(x, y_{j}^{+}) \Big) {\rm d}x.
\end{align}
Let  $\{ {x}_{j-\frac14}^{(\mu)} \}_{\mu = 1}^{N} $ 
and $ \{ {x}_{j+\frac14}^{(\mu)} \}_{\mu = 1}^{N} $
denote the $N$-point Gauss quadrature nodes transformed into the interval $\big[x_{j-\frac{1}{2}}, x_{j}\big]$ 
and  $\big[x_{j}, x_{j+\frac{1}{2}}\big]$, respectively. 
Denote  ${\mathbb{Q}}_{j}^{x} := \{ {x}_{j-\frac14}^{(\mu)} \}_{\mu = 1}^{N} \cup  \{ {x}_{j+\frac14}^{(\mu)} \}_{\mu = 1}^{N}$. 
Let  $\{\omega_{\mu}\}_{\mu = 1}^{N}$
be the associated weights satisfying $\sum_{\mu = 1}^{N} \omega_{\mu} = 1$. 
Similarly, use ${\mathbb{Q}}_{j}^{y} = \{ {y}_{j-\frac14}^{(\mu)} \}_{\mu = 1}^{N} \cup \{ {y}_{j+\frac14}^{(\mu)} \}_{\mu = 1}^{N}$ to denote the Gauss quadrature nodes in the $y$-direction.
For the accuracy requirement, we take $N = k+1$ for a $\mathbb{P}^k$-based CDG method.
With these quadrature rules approximating the cell interface integrals, the semi-discrete equations 
for the cell averages in the CDG method \eqref{eq:CDG-2d-primal}--\eqref{eq:CDG-2d-dual} can be written as
\begin{equation}\label{eq:CDG-2d-high}
	\frac{ {\rm d} \overline{ {\bf U} }_{ij}^{C} }{ {\rm d} t} = {\mathbfcal L}_{ij} \big( {\bf U}_{h}^{C}, {\bf U}_{h}^{D} \big), \qquad \frac{ {\rm d} \overline{ {\bf U} }_{i+\frac{1}{2},j+\frac{1}{2}}^{D} }{ {\rm d} t} = {\mathbfcal L}_{i+\frac{1}{2},j+\frac{1}{2}} \big( {\bf U}_{h}^{D}, {\bf U}_{h}^{C} \big)
\end{equation}
with 
{\small\begin{align}  \nonumber
	{\mathbfcal L}_{ij} \big( {\bf U}_{h}^{C}, {\bf U}_{h}^{D} \big) = &
	\frac{ \overline{ {\bf U} }_{ij}^{D} - \overline{ {\bf U} }_{ij}^{C}  }{\tau_{\max}}   
- \frac{ 1 }{\Delta x} \sum\limits_{ \sigma = \pm 1 }  \sum_{\mu=1}^N \frac{\omega_{\mu} }{2}
	\Big(  {\bf F}_{1}({\bf U}_{h}^{D}(x_{i+\frac{1}{2}}, {y}_{j+\frac{\sigma}4}^{(\mu)} )) 
	- {\bf F}_{1}({\bf U}_{h}^{D}(x_{i-\frac{1}{2}},{y}_{j+\frac{\sigma}4}^{(\mu)})) \Big) 
	\\ \label{eq:CDG-2d-primal-high} 
	&  - \frac{1}{\Delta y}  \sum_{ \sigma = \pm 1 }  \sum_{\mu=1}^N  \frac{\omega_{\mu}}{2}
	\Big(  {\bf F}_{2} ({\bf U}_{h}^{D}( {x}_{i+\frac{\sigma}4}^{(\mu)} , y_{j+\frac{1}{2}})) 
	- {\bf F}_{2} ({\bf U}_{h}^{D}( {x}_{i+\frac{\sigma}4}^{(\mu)}, y_{j-\frac{1}{2}})) \Big),
\\ \nonumber
	 {\mathbfcal L}_{i+\frac{1}{2},j+\frac{1}{2}} \big( {\bf U}_{h}^{D}, {\bf U}_{h}^{C} \big)  = &  
	\frac{  \overline{ {\bf U} }_{i+\frac{1}{2},j+\frac{1}{2}}^{C} 
		- \overline{ {\bf U} }_{i+\frac{1}{2},j+\frac{1}{2}}^{D}  }{\tau_{\max}}  
- \frac{ 1 }{\Delta x} \sum\limits_{ \sigma = \pm 1 }  \sum_{\mu=1}^N \frac{\omega_{\mu} }{2}
	\Big(  {\bf F}_{1}({\bf U}_{h}^{C}(x_{i+1}, {y}_{j+\frac{2+\sigma}4}^{(\mu)} )) 
	- {\bf F}_{1}({\bf U}_{h}^{C}(x_{i},{y}_{j+\frac{2+\sigma}4}^{(\mu)})) \Big) 
	\\  \label{eq:CDG-2d-dual-high} 
	&  - \frac{1}{\Delta y}  \sum_{ \sigma = \pm 1 }  \sum_{\mu=1}^N  \frac{\omega_{\mu}}{2}
	\Big(  {\bf F}_{2} ({\bf U}_{h}^{C}( {x}_{i+\frac{2+\sigma}4}^{(\mu)} , y_{j+1})) 
	- {\bf F}_{2} ({\bf U}_{h}^{C}( {x}_{i+\frac{2+\sigma}4}^{(\mu)}, y_{j})) \Big),
\end{align}}
where and below 
 we omit the $t$ dependence of all quantities for convenience.

As we have discussed in the 1D case, to achieve a PP CDG scheme, the main task is to preserve the evolved cell averages in the set $G$ during the updating process. More specifically, we wish the cell-averaged CDG scheme  
\eqref{eq:CDG-2d-high} satisfies 
the following PP property 
\begin{equation}\label{2D-PP-cellave}
	\overline{ {\bf U} }_{ij}^{C} + \Delta t {\mathbfcal L}_{ij} \big( {\bf U}_{h}^{C}, {\bf U}_{h}^{D} \big) \in G,  \quad \overline{ {\bf U} }_{i+\frac{1}{2},j+\frac{1}{2}}^{D} + \Delta t {\mathbfcal L}_{i+\frac{1}{2},j+\frac{1}{2}} \big( {\bf U}_{h}^{D}, {\bf U}_{h}^{C} \big) \in G  \quad \forall i,j,
\end{equation} 
under certain suitable CFL condition on the time stepsize $\Delta t$ and some proper conditions on 
the CDG solution polynomials. The property \eqref{2D-PP-cellave}  guarantees the cell averages staying in $G$ during the updating process, if one uses a strong-stability-preserving (SSP) method for time discretization, which is a convex combination of the forward Euler scheme.

We now employ the GQL approach to carry out a theoretical analysis on the property \eqref{2D-PP-cellave} for the cell-averaged CDG scheme  \eqref{eq:CDG-2d-high}.
As the 1D case, denote 
by  
$ \{ \hat{x}_{j-\frac14}^{(\nu)} \}_{\nu = 1}^{L} $ 
and $ \{ \hat{x}_{j+\frac14}^{(\nu)} \}_{\nu = 1}^{L} $ 
 the Gauss--Lobatto points in $[x_{j-\frac{1}{2}}, x_{j}]$ and  
$[x_{j}, x_{j+\frac{1}{2}}]$, respectively. Denote $\hat{\mathbb{Q}}_{j}^{x} := \{ \hat{x}_{j-\frac14}^{(\nu)} \}_{\nu = 1}^{L}  \cup \{ \hat{x}_{j+\frac14}^{(\nu)} \}_{\nu = 1}^{L}$. 
The Gauss--Lobatto points 
in the $y$-direction are 
similarly denoted as $\hat{\mathbb{Q}}_{j}^{y} :=  \{ \hat{y}_{j-\frac14}^{(\nu)} \}_{\nu = 1}^{L} \cup \{ \hat{y}_{j+\frac14}^{(\nu)} \}_{\nu = 1}^{L} $. 
We take $L=\lceil \frac{k+3}2 \rceil$, which gives $2L-3\geq k$, so that
the $L$-point Gauss--Lobatto quadrature rule is exact for polynomials of degree up to $k$. The exactness of the quadrature rules implies that 
\begin{equation}\label{eq:decomp}
	\overline{{\bf U}}_{ij}^{D} =  \sum\limits_{\nu = 1}^{L} \frac{\hat{\omega}_{\nu}}{2} \mathbf{\Pi}_{ij}^{\nu,-} 
	+  \sum\limits_{\nu = 1}^{L} \frac{\hat{\omega}_{\nu}}{2} \mathbf{\Pi}_{ij}^{\nu,+}, \quad 
		\overline{{\bf U}}_{i+\frac12,j+\frac12}^{C} =  \sum\limits_{\nu = 1}^{L} \frac{\hat{\omega}_{\nu}}{2} \mathbf{\Pi}_{i+\frac12,j+\frac12}^{\nu,-}
	+  \sum\limits_{\nu = 1}^{L} \frac{\hat{\omega}_{\nu}}{2} \mathbf{\Pi}_{i+\frac12,j+\frac12}^{\nu,+}
\end{equation}
with 
\begin{align*}
	\mathbf{\Pi}_{ij}^{\nu,\pm } &:= \frac{\lambda_{1}}{\lambda} \sum_{\sigma=\pm 1} \sum_{\mu=1}^N \frac{\omega_{\mu}}{2} {\bf U}_{h}^{D}( \hat{x}_{i\pm\frac{1}4}^{(\nu)}, {y}_{j+\frac{\sigma}4}^{(\mu)} )
	+ \frac{\lambda_{2}}{\lambda}  \sum_{\sigma=\pm 1} \sum_{\mu=1}^N  \frac{\omega_{\mu}}{2} {\bf U}_{h}^{D} ( {x}_{i+\frac{\sigma}4}^{(\mu)}, \hat{y}_{j\pm\frac{1}4}^{(\nu)} ),
	\\
	\mathbf{\Pi}_{i+\frac12,j+\frac12}^{\nu,\pm} &:= \frac{\lambda_{1}}{\lambda} \sum_{\sigma=\pm 1} \sum_{\mu=1}^N \frac{\omega_{\mu}}{2} {\bf U}_{h}^{C}( \hat{x}_{i+\frac12\pm\frac{1}4}^{(\nu)}, {y}_{j+\frac12+\frac{\sigma}4}^{(\mu)} )
+ \frac{\lambda_{2}}{\lambda}  \sum_{\sigma=\pm 1} \sum_{\mu=1}^N  \frac{\omega_{\mu}}{2} {\bf U}_{h}^{C} ( {x}_{i+\frac12+\frac{\sigma}4}^{(\mu)}, \hat{y}_{j+\frac12\pm\frac{1}4}^{(\nu)} ).
\end{align*}
Here $\lambda_{1} = \frac{ a_{1} \Delta t}{\Delta x}$, $ \lambda_{2} = \frac{ a_{2}\Delta t}{\Delta y} $, 
$\lambda = \lambda_{1} + \lambda_{2}$, with 
\begin{align}  \label{eq:defa1}
		a_1 & \ge  \max_{i,j} \max_{y \in {\mathbb{Q}}_{j}^{y}  } \left\{ \alpha_{1} \big(
	{\bf U}_{h}^{C}(x_{i+1},y),
	{\bf U}_{h}^{C}(x_{i}, y)
	\big),   \alpha_{1} \big(
	{\bf U}_{h}^{D}(x_{i+\frac12},y),
	{\bf U}_{h}^{D}(x_{i-\frac12}, y)
	\big)
	\right\}=:\hat a_1,
	\\ \label{eq:defa2}
	a_2  & \ge  \max_{i,j} \max_{x \in {\mathbb{Q}}_{i}^{x}  } \left\{ \alpha_{2} \big(
	{\bf U}_{h}^{C}(x,y_{j+1}),
	{\bf U}_{h}^{C}(x, y_j)
	\big),   \alpha_{2} \big(
	{\bf U}_{h}^{D}(x,y_{j+\frac12}),
	{\bf U}_{h}^{D}(x, y_{j-\frac12})
	\big)
	\right\}=:\hat a_2.
\end{align}
We introduce the discrete divergence operators for the numerical magnetic fields 
${\bf B}_{h}^{D}(x,y)$ and ${\bf B}_{h}^{C}(x,y)$:
	\begin{align}\label{eq:divijD}  
\begin{aligned}
		{\rm div}_{ij} {\bf B}_{h}^{D} 
	:&=  \frac{1 }{\Delta x}  \sum_{\sigma=\pm 1} \sum_{\mu=1}^N  \frac{\omega_{\mu} }{2}
	\Big( B_{1,h}^{D}( x_{i+\frac{1}{2}}, {y}_{j+\frac{\sigma}4}^{(\mu)} )  - B_{1,h}^{D}(x_{i-\frac{1}{2}},  {y}_{j+\frac{\sigma}4}^{(\mu)} ) \Big)  \\
	&+  \frac{ 1 }{\Delta y}  \sum_{\sigma=\pm 1} \sum_{\mu=1}^N   \frac{\omega_{\mu} }{2}
	\Big( B_{2,h}^{D}( {x}_{i+\frac{\sigma}4}^{(\mu)} ,y_{j+\frac{1}{2}})  - B_{2,h}^{D}( {x}_{i+\frac{\sigma}4}^{(\mu)} , y_{j-\frac{1}{2}}) \Big) \,, 
\end{aligned}
	\\  \label{eq:divijC} 
\begin{aligned}
		{\rm div}_{i+\frac{1}{2},j+\frac{1}{2}} {\bf B}_{h}^{C} 
	:&=  \frac{1}{\Delta x}  \sum_{\sigma=\pm 1} \sum_{\mu=1}^N  \frac{\omega_{\mu} }{2}
	\Big( B_{1,h}^{C}(x_{i+1}, {y}_{j+\frac{2+\sigma}4}^{(\mu)}  )  - B_{1,h}^{C}(x_{i},{y}_{j+\frac{2+\sigma}4}^{(\mu)})\Big) \\
	&+  \frac{1}{\Delta y} \sum_{\sigma=\pm 1} \sum_{\mu=1}^N   \frac{\omega_{\mu} }{2}
	\Big( B_{2,h}^{C}( {x}_{i+\frac{2+\sigma}4}^{(\mu)} ,y_{j+1})  - B_{2,h}^{C}( {x}_{i+\frac{2+\sigma}4}^{(\mu)} , y_{j}) \Big)  \,,
\end{aligned}
\end{align}
which are numerical approximations to the weak divergence $ \frac{1}{\Delta x \Delta y} \int_{\partial I} 
{\bf B} \cdot {\bm n}_{\partial I} {\rm d} s = 
\frac{1}{\Delta x \Delta y}\iint_{I}\nabla \cdot {\bf B} {\rm d}x{\rm d}y $
on the cells $I_{i,j}$ and $I_{i+\frac{1}{2},j+\frac{1}{2}}$, respectively, where ${\bm n}_{\partial I}$ is the outward pointing unit normal of $\partial I$.

\begin{theorem}[Bridge PP and DF properties for 2D standard CDG method]\label{theorem:positivity-2d-high}	
	Assume $\overline{ {\bf U} }_{ij}^{C}\in G$, $\overline{ {\bf U} }_{i+\frac{1}{2},j+\frac{1}{2}}^{D}\in G$ and that the numerical solutions ${\bf U}_{h}^{C}(x,y), {\bf U}_{h}^{D}(x,y)$ satisfy
	\begin{equation}\label{pp-condition-2d}
		{\bf U}_{h}^{C}(x,y)  \in {G} \,, \quad {\bf U}_{h}^{D}(x,y)  \in {G}  \quad  
		\forall (x,y) \in \mathop{\cup}_{i,j} \mathbb{Q}_{ij}, 
	\end{equation}
where $\mathbb{Q}_{ij} := (\mathbb{Q}_{i}^x\otimes \hat{\mathbb{Q}}_{j}^y ) \cup
( \hat{\mathbb{Q}}_{i}^x\otimes \mathbb{Q}_{j}^y )$. 
For all $i$ and $j$, the updated cell averages 
	 ${\bf U}_{\Delta t}^{C} := \overline{ {\bf U} }_{ij}^{C} + \Delta t {\mathbfcal L}_{ij} ( {\bf U}_{h}^{C}, {\bf U}_{h}^{D} )$ and  ${\bf U}_{\Delta t}^{D} := \overline{ {\bf U} }_{i+\frac{1}{2},j+\frac{1}{2}}^{D} + \Delta t {\mathbfcal L}_{i+\frac{1}{2},j+\frac{1}{2}} ( {\bf U}_{h}^{D}, {\bf U}_{h}^{C} )$ satisfy for any free auxiliary variables ${\bf v}^{\ast},{\bf B}^{\ast} \in \mathbb R^3$ that 
	\begin{align}\label{eq:2Ddensity}
		& {\bf U}_{\Delta t}^{C} \cdot {\bf n}_1 >0, \qquad  {\bf U}_{\Delta t}^{D} \cdot {\bf n}_1 >0,
		\\
		\label{eq:positivity-2d-high-primal}
		&
		{\bf U}_{\Delta t}^{C} \cdot {\bf n}^{\ast} + \frac{|{\bf B}^{\ast}|^2}{2}  > 
		\frac{\theta\hat{\omega}_{1}}{2} \Big(  ( \mathbf{\Pi}_{ij}^{L,-} + \mathbf{\Pi}_{ij}^{1,+} ) \cdot {\bf n}^{\ast} + |{\bf B}^{\ast}|^2 \Big)
		- \Delta t ({\bf v}^{\ast}\cdot{\bf B}^{\ast}) ( {\rm div}_{ij} {\bf B}_{h}^{D} )   \,,
		\\  \label{eq:positivity-2d-high-dual}
		&{\bf U}_{\Delta t}^{D} \cdot {\bf n}^{\ast} + \frac{|{\bf B}^{\ast}|^2}{2}  > 
		\frac{\theta\hat{\omega}_{1}}{2} \Big(  ( \mathbf{\Pi}_{i+\frac12,j+\frac12}^{L,-} + \mathbf{\Pi}_{i+\frac12,j+\frac12}^{1,+} ) \cdot {\bf n}^{\ast} + |{\bf B}^{\ast}|^2 \Big)
		- \Delta t ({\bf v}^{\ast}\cdot{\bf B}^{\ast}) ( {\rm div}_{i+\frac12,j+\frac12} {\bf B}_{h}^{C} )   \,,
	\end{align}
	under the CFL condition	
\begin{equation}\label{eq:2D-CFL}
	\lambda= \frac{ a_1 \Delta t}{\Delta x} +  \frac{ a_2 \Delta t}{\Delta y} 
	< \frac{\theta \hat{\omega}_{1}}{2} \,, \qquad
	\theta = \frac{\Delta t}{\tau_{\max}} \in (0,1] \,.
\end{equation}
	Furthermore, if ${\bf U}_{h}^{C}(x,y)$ and ${\bf U}_{h}^{D}(x,y)$ satisfy the following discrete DF condition
	\begin{equation}\label{div-condition-2d}
		{\rm div}_{i,j} {\bf B}_{h}^{D}  = 0 \,,  \qquad
		{\rm div}_{i+\frac{1}{2},j+\frac{1}{2}} {\bf B}_{h}^{C} = 0    \qquad \forall i,j \,,
	\end{equation}
	then \eqref{eq:2Ddensity}--\eqref{eq:positivity-2d-high-dual} imply ${\bf U}_{\Delta t}^{C}, {\bf U}_{\Delta t}^{D} \in {G}$, namely, the desired PP property \eqref{2D-PP-cellave}. 
\end{theorem}

\begin{proof}
	For $\ell \in \{1,2\}$ and any two admissible states ${\bf U}, \tilde{\bf U} \in G$, we observe that 
	\begin{align}\label{WKL3131}
		- ( {\bf F}_\ell ( {\bf U} ) - {\bf F}_\ell ( \tilde{\bf U} ) ) \cdot {\bf n}_1 &= 
		  \tilde \rho \tilde v_\ell - \rho v_\ell  > -  ( \tilde \rho + \rho ) \alpha_\ell ( {\bf U}, \tilde {\bf U} ) = - \alpha_\ell ( {\bf U}, \tilde {\bf U} )  ( {\bf U} +  \tilde {\bf U} ) \cdot {\bf n}_1,
		  \\ \label{WKL3133}
		  - ( {\bf F}_\ell ( {\bf U} ) - {\bf F}_\ell ( \tilde{\bf U} ) ) \cdot {\bf n}^{\ast} 
		  &\ge -\alpha_\ell ( {\bf U}, \tilde {\bf U} ) 
		  \left(  ( {\bf U} +  \tilde {\bf U} ) \cdot {\bf n}^{\ast} + |{\bf B}^\ast|^2 \right) - ( B_\ell -\tilde B_\ell ) ( {\bf v}^\ast \cdot {\bf B}^\ast ),  
	\end{align}
	where the second inequality \eqref{WKL3133} follows from \cref{theo:MHD:LLFsplit} for any free auxiliary variables ${\bf v}^{\ast},{\bf B}^{\ast} \in \mathbb R^3$. 
	We reformulate the updated cell average 
	${\bf U}_{\Delta t}^{C}$ as 
	\begin{equation}\label{WKL3021}
		{\bf U}_{\Delta t}^{C} = (1-\theta) \overline{{\bf U}}_{ij}^{C} + \theta \overline{{\bf U}}_{ij}^{D}
		+ {\bf \Pi}_F,
	\end{equation}
	where  
	\begin{align*}
		{\bf \Pi}_F:= & - \frac{ \Delta t }{\Delta x}  \sum_{\sigma=\pm 1} \sum_{\mu=1}^N  \frac{\omega_{\mu} }{2}
	\Big(  {\bf F}_{1}({\bf U}_{h}^{D}(x_{i+\frac{1}{2}},{y}_{j+\frac{\sigma}4}^{(\mu)})) 
	- {\bf F}_{1}({\bf U}_{h}^{D}(x_{i-\frac{1}{2}},{y}_{j+\frac{\sigma}4}^{(\mu)})) \Big) 
	\\
	&  - \frac{\Delta t}{\Delta y}  \sum_{\sigma=\pm 1} \sum_{\mu=1}^N \frac{\omega_{\mu}}{2}
	\Big(  {\bf F}_{2} ({\bf U}_{h}^{D}({x}_{i+\frac{\sigma}4}^{(\mu)}, y_{j+\frac{1}{2}})) 
	- {\bf F}_{2} ({\bf U}_{h}^{D}({x}_{i+\frac{\sigma}4}^{(\mu)}, y_{j-\frac{1}{2}})) \Big)
	\end{align*}
	with ${\bf U}_{h}^{D}(x_{i\pm \frac{1}{2}},{y}_{j + \frac{\sigma}4}^{(\mu)}) \in G$ and ${\bf U}_{h}^{D}({x}_{i+\frac{\sigma}4}^{(\mu)}, y_{j\pm\frac{1}{2}}) \in G$ according to the hypothesis \eqref{pp-condition-2d}. 
	By applying \eqref{WKL3131}, one can estimate the lower bound of ${\bf \Pi}_F \cdot {\bf n}_1$ as 
	\begin{align*}
		{\bf \Pi}_F \cdot {\bf n}_1  \overset{\mbox{\eqref{WKL3131}}}{>} & - a_1 \frac{ \Delta t }{\Delta x}  \sum_{\sigma=\pm 1} \sum_{\mu=1}^N  \frac{\omega_{\mu} }{2}
		\Big(  {\bf U}_{h}^{D}(x_{i+\frac{1}{2}},{y}_{j+\frac{\sigma}4}^{(\mu)})
		+ {\bf U}_{h}^{D}(x_{i-\frac{1}{2}},{y}_{j+\frac{\sigma}4}^{(\mu)}) \Big) \cdot {\bf n}_1 
		\\
		&  - a_2 \frac{\Delta t}{\Delta y}  \sum_{\sigma=\pm 1} \sum_{\mu=1}^N \frac{\omega_{\mu}}{2}
		\Big(  {\bf U}_{h}^{D}({x}_{i+\frac{\sigma}4}^{(\mu)}, y_{j+\frac{1}{2}})
		+ {\bf U}_{h}^{D}({x}_{i+\frac{\sigma}4}^{(\mu)}, y_{j-\frac{1}{2}}) \Big) \cdot {\bf n}_1 
		\\
		= & -\lambda_1    \sum_{\sigma=\pm 1} \sum_{\mu=1}^N  \frac{\omega_{\mu} }{2}
		\Big(  {\bf U}_{h}^{D}( \hat{x}_{i + \frac{1}4}^{(L)} ,{y}_{j+\frac{\sigma}4}^{(\mu)})
		+ {\bf U}_{h}^{D}( \hat{x}_{i - \frac{1}4}^{(1)} ,{y}_{j+\frac{\sigma}4}^{(\mu)}) \Big) \cdot {\bf n}_1 
				\\
		&  - \lambda_2 \sum_{\sigma=\pm 1} \sum_{\mu=1}^N \frac{\omega_{\mu}}{2}
		\Big(  {\bf U}_{h}^{D}({x}_{i+\frac{\sigma}4}^{(\mu)},  \hat{y}_{j + \frac{1}4}^{(L)} )
		+ {\bf U}_{h}^{D}({x}_{i+\frac{\sigma}4}^{(\mu)}, \hat{y}_{j - \frac{1}4}^{(1)} ) \Big) \cdot {\bf n}_1 
		\\
		=& - \lambda \left( \mathbf{\Pi}_{ij}^{L,+ } + \mathbf{\Pi}_{ij}^{1,- } \right) \cdot {\bf n}_1, 
	\end{align*}
	where $\hat{x}_{i + \frac{1}4}^{(L)}=x_{i+\frac{1}{2}}$, $\hat{x}_{i - \frac{1}4}^{(1)}=x_{i-\frac{1}{2}}$, $\hat{y}_{j + \frac{1}4}^{(L)}=y_{j+\frac{1}{2}}$, and $\hat{y}_{j - \frac{1}4}^{(1)}=y_{j-\frac{1}{2}}$ are used. 
	It then follows from \eqref{WKL3021} that 
	 \begin{align*}
	 	{\bf U}_{\Delta t}^{C} \cdot {\bf n}_1  =~ & (1-\theta) \overline{{\bf U}}_{ij}^{C} \cdot {\bf n}_1  + \theta \overline{{\bf U}}_{ij}^{D} \cdot {\bf n}_1 
	 	+ {\bf \Pi}_F  \cdot {\bf n}_1  
	 	\\
	 	 >~  & \theta \overline{{\bf U}}_{ij}^{D} \cdot {\bf n}_1 - \lambda \left( \mathbf{\Pi}_{ij}^{L,+ } + \mathbf{\Pi}_{ij}^{1,- } \right) \cdot {\bf n}_1
	 	\\
	 	 \overset{\mbox{\eqref{eq:decomp}}}{=} &
	 	\theta \left(
	 	\sum\limits_{\nu = 1}^{L} \frac{\hat{\omega}_{\nu}}{2} \mathbf{\Pi}_{ij}^{\nu,-} 
	 	+  \sum\limits_{\nu = 1}^{L} \frac{\hat{\omega}_{\nu}}{2} \mathbf{\Pi}_{ij}^{\nu,+} \right) \cdot {\bf n}_1  
	 	- \lambda \left( \mathbf{\Pi}_{ij}^{L,+ } + \mathbf{\Pi}_{ij}^{1,- } \right) \cdot {\bf n}_1
	 	\\
	 	\ge~  &
	 	\left( \frac{\theta \hat{\omega}_{1}}{2} -  \lambda \right) \left( \mathbf{\Pi}_{ij}^{L,+ } + \mathbf{\Pi}_{ij}^{1,- } \right) \cdot {\bf n}_1 \overset{\mbox{\eqref{eq:2D-CFL}}}{>}  0, 
	 \end{align*}
	where we have used the identity \eqref{eq:decomp}, the CFL condition \eqref{eq:2D-CFL}, and 
	$\mathbf{\Pi}_{ij}^{\nu,\pm} \in G$ which follows from the convexity of $G$ and the hypothesis \eqref{pp-condition-2d}. 
	Next, we apply \eqref{WKL3133} to estimate the lower bound of ${\bf \Pi}_F \cdot {\bf n}^\ast$ for free auxiliary variables ${\bf v}^{\ast},{\bf B}^{\ast} \in \mathbb R^3$ as follows: 
	\begin{align} \notag
	{\bf \Pi}_F \cdot {\bf n}^\ast  \overset{\mbox{\eqref{WKL3133}}}{\ge} & - \lambda_{1} \sum_{\sigma=\pm 1} \sum_{\mu=1}^N \frac{\omega_{\mu}}{2} 
	\Big( {\bf U}_{h}^{D}(x_{i+\frac{1}{2}},{y}_{j+\frac{\sigma}4}^{(\mu)}) + 
	{\bf U}_{h}^{D}(x_{i-\frac{1}{2}},{y}_{j+\frac{\sigma}4}^{(\mu)}) \Big) \cdot {\bf n}^{\ast} - \lambda_{1} |{\bf B}^{\ast}|^2 \\ \notag
	& - \frac{\Delta t }{\Delta x}  \sum_{\sigma=\pm 1} \sum_{\mu=1}^N \frac{\omega_{\mu} }{2}
	\Big( B_{1,h}^{D}(x_{i+\frac{1}{2}},{y}_{j+\frac{\sigma}4}^{(\mu)})  - B_{1,h}^{D}(x_{i-\frac{1}{2}},{y}_{j+\frac{\sigma}4}^{(\mu)} ) \Big)
	({\bf v}^{\ast}\cdot{\bf B}^{\ast}) \\ \notag
	& - \lambda_{2}  \sum_{\sigma=\pm 1} \sum_{\mu=1}^N  \frac{\omega_{\mu}}{2}
	\Big( {\bf U}_{h}^{D}({x}_{i+\frac{\sigma}4}^{(\mu)}, y_{j+\frac{1}{2}}) + {\bf U}_{h}^{D}( {x}_{i+\frac{\sigma}4}^{(\mu)}, y_{j-\frac{1}{2}}) \Big)
	\cdot {\bf n}^{\ast} - \lambda_{2} |{\bf B}^{\ast}|^2 \\ \notag
	& -  \frac{\Delta t}{\Delta y} \sum_{\sigma=\pm 1} \sum_{\mu=1}^N \frac{\omega_{\mu} }{2}
	\Big( B_{2,h}^{D}({x}_{i+\frac{\sigma}4}^{(\mu)},y_{j+\frac{1}{2}})  - B_{2,h}^{D}({x}_{i+\frac{\sigma}4}^{(\mu)}, y_{j-\frac{1}{2}}) \Big)
	({\bf v}^{\ast}\cdot{\bf B}^{\ast})  \\ \label{WKL2001}
	= &  - \lambda  \Big( ( \mathbf{\Pi}_{ij}^{L,+ } + \mathbf{\Pi}_{ij}^{1,- } ) \cdot {\bf n}^{\ast}  + |{\bf B}^{\ast}|^2 \Big) 
	- \Delta t ( {\rm div}_{ij} {\bf B}_{h}^{D} ) ({\bf v}^{\ast}\cdot{\bf B}^{\ast}) \,.
\end{align}	 
Combining this estimate with \eqref{WKL3021} leads to 
	\begin{align*}
		 {\bf U}_{\Delta t}^{C} \cdot {\bf n}^{\ast} + \frac{|{\bf B}^{\ast}|^2}{2}  \overset{\mbox{\eqref{WKL3021}}}{=} &  
		(1-\theta) \left(  \overline{{\bf U}}_{ij}^{C} \cdot {\bf n}^{\ast} + \frac{|{\bf B}^{\ast}|^2}{2} \right) + \theta \left( \overline{{\bf U}}_{ij}^{D} \cdot {\bf n}^{\ast}  + \frac{|{\bf B}^{\ast}|^2}{2} \right) + {\bf \Pi}_F \cdot {\bf n}^{\ast}
		\\  
		\overset{\mbox{\eqref{eq:2D-CFL}}}{\ge}   &  
		  \theta \left( \overline{{\bf U}}_{ij}^{D} \cdot {\bf n}^{\ast}  + \frac{|{\bf B}^{\ast}|^2}{2} \right) + {\bf \Pi}_F \cdot {\bf n}^{\ast}
		\\
		\overset{\mbox{\eqref{WKL2001}}}{\ge} &  \theta \left( \overline{{\bf U}}_{ij}^{D} \cdot {\bf n}^{\ast}  + \frac{|{\bf B}^{\ast}|^2}{2} \right) - \lambda  \Big( ( \mathbf{\Pi}_{ij}^{L,+ } + \mathbf{\Pi}_{ij}^{1,- } ) \cdot {\bf n}^{\ast}  + |{\bf B}^{\ast}|^2 \Big) 
		- \Delta t ( {\rm div}_{ij} {\bf B}_{h}^{D} ) ({\bf v}^{\ast}\cdot{\bf B}^{\ast})
		\\
		 \overset{\mbox{\eqref{eq:decomp}}}{=} &
		\theta 
		\sum\limits_{\nu = 1}^{L} \frac{\hat{\omega}_{\nu}}{2} \left( \mathbf{\Pi}_{ij}^{\nu,-} \cdot {\bf n}^{\ast}  + \frac{|{\bf B}^{\ast}|^2}{2} \right) 
		+ \theta  \sum\limits_{\nu = 1}^{L} \frac{\hat{\omega}_{\nu}}{2} \left( \mathbf{\Pi}_{ij}^{\nu,+} 
		\cdot {\bf n}^{\ast}  + \frac{|{\bf B}^{\ast}|^2}{2} \right) 
		\\
		& - \lambda  \Big( ( \mathbf{\Pi}_{ij}^{L,+ } + \mathbf{\Pi}_{ij}^{1,- } ) \cdot {\bf n}^{\ast}  + |{\bf B}^{\ast}|^2 \Big) 
		- \Delta t ( {\rm div}_{ij} {\bf B}_{h}^{D} ) ({\bf v}^{\ast}\cdot{\bf B}^{\ast})
		\\
		\ge ~ &   \frac{\theta \hat{\omega}_{1}}{2}  \Big( ( \mathbf{\Pi}_{ij}^{L,- } + \mathbf{\Pi}_{ij}^{1,+ } ) \cdot {\bf n}^{\ast}  + |{\bf B}^{\ast}|^2 \Big)
		\\
		& + \left( \frac{\theta \hat{\omega}_{1}}{2} -  \lambda \right)  \Big( ( \mathbf{\Pi}_{ij}^{L,+ } + \mathbf{\Pi}_{ij}^{1,- } ) \cdot {\bf n}^{\ast}  + |{\bf B}^{\ast}|^2 \Big)  
		- \Delta t ( {\rm div}_{ij} {\bf B}_{h}^{D} ) ({\bf v}^{\ast}\cdot{\bf B}^{\ast})
		\\
		\overset{\mbox{\eqref{eq:2D-CFL}}}{>}   & \frac{\theta \hat{\omega}_{1}}{2}  \Big( ( \mathbf{\Pi}_{ij}^{L,- } + \mathbf{\Pi}_{ij}^{1,+ } ) \cdot {\bf n}^{\ast}  + |{\bf B}^{\ast}|^2 \Big)
		- \Delta t ( {\rm div}_{ij} {\bf B}_{h}^{D} ) ({\bf v}^{\ast}\cdot{\bf B}^{\ast}),
	\end{align*}
	which gives \eqref{eq:positivity-2d-high-primal} and further implies that 
	$$
	{\bf U}_{\Delta t}^{C} \cdot {\bf n}^{\ast} + \frac{|{\bf B}^{\ast}|^2}{2}  + \Delta t ( {\rm div}_{ij} {\bf B}_{h}^{D} ) ({\bf v}^{\ast}\cdot{\bf B}^{\ast}) >0. 
	$$
	Therefore, if ${\bf U}_{h}^{D}(x,y)$ further satisfies the discrete DF condition ${\rm div}_{ij} {\bf B}_{h}^{D}=0$, then we  obtain 
	$$
	{\bf U}_{\Delta t}^{C} \cdot {\bf n}^{\ast} + \frac{|{\bf B}^{\ast}|^2}{2} >0 \qquad \forall {\bf v}^\ast,{\bf B}^\ast \in \mathbb R^3,
	$$
	which along with ${\bf U}_{\Delta t}^{C} \cdot {\bf n}_1>0$ implies ${\bf U}_{\Delta t}^{C} \in G_*=G$, according to the GQL representation in \cref{theo:eqDefG}. 
	Similarly, one can derive ${\bf U}_{\Delta t}^{D} \cdot {\bf n}_1 >0$ and the estimate  \eqref{eq:positivity-2d-high-dual} for ${\bf U}_{\Delta t}^{D}$, which further lead to 
	${\bf U}_{\Delta t}^{D} \in G_*=G$ under the discrete DF condition \eqref{div-condition-2d}. 
	The proof is completed. 
\end{proof}

\begin{remark}
	\cref{theorem:positivity-2d-high} shows that the 
	PP property of 2D standard CDG method is closely related to a discrete DF condition \eqref{div-condition-2d}, 
	which is significantly different from both the trivial 1D version \eqref{eq:1D-DDF} 
	and the non-central DG version found in \cite{Wu2017a}.  
				As seen from  \eqref{eq:positivity-2d-high-primal} and \eqref{div-condition-2d}, 
	the discrete DF condition on the primal mesh is defined by the numerical solution on the dual mesh; see 
	\cref{eq:FigDDF}. 
\end{remark}

\begin{remark}	
	As the free auxiliary variables $\{ {\bf v}^\ast, {\bf B}^\ast \}$ are necessary 
	in \eqref{eq:positivity-2d-high-primal}--\eqref{eq:positivity-2d-high-dual}, 
	the GQL approach is essential for bridging 
	the PP and discrete DF properties. It seems very challenging (if not impossible) 
	to draw the connection between the PP and discrete DF properties without using the GQL approach. 
	Since the states at all the quadrature points in the CDG schemes are coupled by the discrete DF condition,  the PP analysis is very nontrivial, and some standard PP techniques, which typically rely on reformulating a 2D scheme into convex combination of formal 1D PP schemes \cite{zhang2010b,li2016maximum}, are inapplicable in our analysis. 
\end{remark}

\begin{theorem}[Necessity of discrete DF condition for standard CDG method]\label{thm:necessity}
	For any given CFL number ${\tt C}>0$ and any $\theta \in (0,1]$, the 2D standard CDG method, even under the condition \eqref{pp-condition-2d}, is not always PP in general, if the proposed 
	discrete DF condition \eqref{div-condition-2d} is violated. 
\end{theorem}

\begin{proof}
	It is proved by contradiction. 
Suppose there exists a CFL number ${\tt C}=\tau_{\max}(\hat a_1/\Delta x+\hat a_2/\Delta y)>0$, such that 
the PP property \eqref{2D-PP-cellave} always holds under the condition \eqref{pp-condition-2d}. 
Define the constant 
$\delta := \min \left \{ \frac{ {\tt C}}{8}, 1 \right \} \in (0,1].$  
Consider the ideal EOS, the $\mathbb P^0$-based CDG method with $\Delta x=\Delta y$ and piecewise constant data 
\begin{equation}\label{eq:counterEx2D}
	{\bf U}_h^C(x,y) \equiv {\bf U}_0 ~~ \forall (x,y)\in \Omega, \quad 
	{\bf U}_h^D(x,y) = 
	\begin{cases}
		{\bf U}_1,   &(x,y)  \in I_{i-\frac12,j-\frac12} \cup I_{i-\frac12,j+\frac12}, 
		\\[2mm]
		{\bf U}_2,  &(x,y)  \in I_{i+\frac12,j-\frac12} \cup I_{i+\frac12,j+\frac12}, 
		\\[2mm]
		{\bf U}_0, &{\rm otherwise},
	\end{cases}
\end{equation}
where the three constant admissible  states are defined by 
\begin{align*}
	{\bf U}_0 &=\left(1, 1 +  \delta \epsilon,~0,~0,~1+\frac{\epsilon}{2},~0,~0,~
 \frac{(1+   \delta \epsilon )^2}2  
+ \frac{(2+\epsilon)^2}{8} + \frac{\tt p}{\gamma-1}  
\right)^\top,
\\
{\bf U}_1 &= \left(1,~1,~0,~0,~1,~0,~0,~1+\frac{\tt p}{\gamma-1}\right)^\top,
\quad 
{\bf U}_2 = \left(1,~1,~0,~0,~1+ \epsilon,~0,~0,~\frac{1+(1+\epsilon)^2}2 + \frac{\tt p}{\gamma-1} \right)^\top
\end{align*}
with ${\tt p} \in \big(0,\frac{1}{\gamma}\big)$ and $\epsilon \in (0,\delta )$. 
Notice that ${\bf U}_0,{\bf U}_1,{\bf U}_2\in G$, so that 
the solutions \eqref{eq:counterEx2D} automatically satisfy the condition \eqref{pp-condition-2d}. However, they do  not meet the discrete DF condition \eqref{div-condition-2d}, because 
$
{\rm div}_{i,j} {\bf B}_{h}^{D}  = \epsilon/\Delta x \neq 0. 
$
Substituting \eqref{eq:counterEx2D} into ${\bf U}_{\Delta t}^{C} := \overline{ {\bf U} }_{ij}^{C} + \Delta t {\mathbfcal L}_{ij} ( {\bf U}_{h}^{C}, {\bf U}_{h}^{D} )$ gives 
\begin{align*}
	{\bf U}_{\Delta t}^{C} 
& = (1-\theta) {\bf U}_0 + \frac{\theta}2 ( {\bf U}_1 + {\bf U}_2 ) 
+ \frac{\theta {\tt C}}{\hat a_1 + \hat a_2 } ( {\bf F}_1( {\bf U}_1 ) - {\bf F}_1( {\bf U}_2 ) ). 
\end{align*}
According to the PP assumption, we have ${\bf U}_{\Delta t}^{C} \in G$, for any  ${\tt p} \in \big(0,\frac{1}{\gamma}\big)$ and any $\epsilon \in (0,\delta )$.  
For any ${\bf U},\tilde {\bf U} \in \{ {\bf U}_0, {\bf U}_1, {\bf U}_2 \}$, we observe from \eqref{eq:alpha_i} that
\begin{align*}
	\alpha_1 ( {\bf U}, \tilde {\bf U} ) & \le \|v_1\|_{\infty} + \| {\mathcal C}_1 \|_{\infty} + \max_{0\le \ell,s\le 2  } \frac{ | {\bf B}_\ell - {\bf B}_s | }{ \sqrt{\rho_\ell} + \sqrt{\rho_s} } 
=  1 +  \delta \epsilon + (1 + \epsilon) + \frac{ \epsilon }{2} < 5 =: \tilde a_1, 
\\
	\alpha_2 ( {\bf U}, \tilde {\bf U} ) & \le \|v_2\|_{\infty} + \| {\mathcal C}_2 \|_{\infty} + \max_{0\le \ell,s\le 2  } \frac{ | {\bf B}_\ell - {\bf B}_s | }{ \sqrt{\rho_\ell} + \sqrt{\rho_s} } 
	= \sqrt{ \frac12(\gamma-1){\tt p} + (1+\epsilon)^2 } + \frac{ \epsilon }{2} < \sqrt{ \gamma {\tt p} + 4 } + 1 =: \tilde a_2,
\end{align*}
which implies $ w := \frac{\hat a_1 + \hat a_2 }{\tilde a_1 + \tilde a_2} \in (0,1)$. 
Define 
\begin{align*}
	{\bf U}(p,\epsilon) := \, & (1-\theta) {\bf U}_0 + \frac{\theta}2 ( {\bf U}_1 + {\bf U}_2 ) 
+ \frac{\theta {\tt C}}{ \tilde a_1 + \tilde a_2 } ( {\bf F}_1( {\bf U}_1 ) - {\bf F}_1( {\bf U}_2 ) )
\\
= \, & w {\bf U}_{\Delta t}^{C} +  (1-w) \left( (1-\theta) {\bf U}_0 + \frac{\theta}2 ( {\bf U}_1 + {\bf U}_2 )  \right).  
\end{align*}
By the convexity of $G$, we have  
$
{\bf U}(p,\epsilon) \in G, 
$ 
which implies ${\mathcal E} ( {\bf U} ({\tt p},\epsilon) )>0$, for any $ {\tt p} \in \big(0,\frac{1}{\gamma}\big)$ and any $\epsilon \in (0,\delta)$. 
Define $\hat \delta :={\tt C}/8 \ge \delta > \epsilon$ and $\tilde \delta := \hat \delta - \delta \ge 0$. 
Observing  that ${\mathcal E} ( {\bf U})$ is continuous 
  with respect to $\bf U$ on $\mathbb{R}^+\times \mathbb{R}^7$, we obtain 
\begin{equation*}
	0 \le
	\mathop {\lim }\limits_{{\tt p} \to  0^+ } {\mathcal E} ( {\bf U} ({\tt p},\epsilon) ) =  {\mathcal E} \Big( \mathop {\lim }\limits_{{\tt p} \to  0^+ } {\bf U} ({\tt p},\epsilon) \Big)
	= \left( - \frac18 \theta \epsilon \right) \Big[   (8\hat \delta - \epsilon) +  
	\theta \epsilon( 2 \tilde \delta + \hat \delta   \epsilon )^2 + 4 \epsilon \left( \hat \delta + 
	 \delta \tilde \delta 
	 + \delta \hat \delta (1+\epsilon) \right)
	  \Big] < 0,
\end{equation*}
which is a contradiction. Hence the PP assumption is invalid. The proof is completed. 
\end{proof}

\begin{figure}[htbp]
	\centering
	\includegraphics[width=0.999\textwidth]{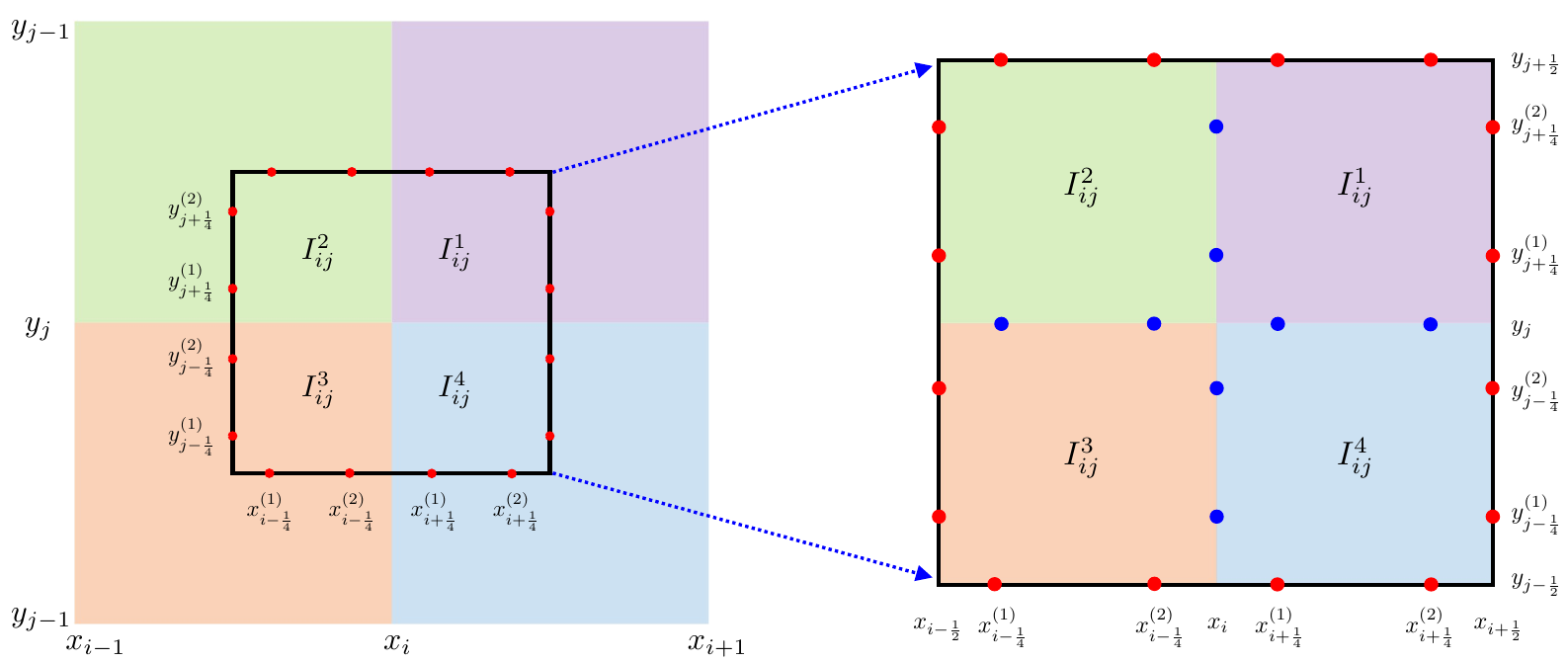}
	\captionsetup{belowskip=-12pt}
	\caption{\small Illustration of the 2D discrete divergence operator \eqref{eq:divijD} on a primal cell (solid lines) with $N=2$ and its relation to the dual mesh (the shadow cells). The red points are involved in \eqref{eq:divijD}, while the blue points are  involved in another discrete divergence operator \eqref{def:newDivB}. These two operators are equivalent when ${\bf B}_h^D$ is locally DF, as shown in the proof of \cref{theorem:new-positivity-2d-high}.} 
	\label{eq:FigDDF}
\end{figure}

\begin{remark}\label{rem:DF}
	The condition \eqref{pp-condition-2d} is a basic standard condition in PP DG type schemes and can be enforced by a local scaling limiter; see \cite{cheng,li2016maximum} and \cite{zhang2010,zhang2010b}. 
	However, unlike many other systems \cite{zhang2010,zhang2010b,li2016maximum}, only condition \eqref{pp-condition-2d} is {\em insufficient} for PP property in the MHD case. 
	\cref{thm:necessity} indicates that the 2D standard CDG method, even with the PP limiter to enforce
	condition \eqref{pp-condition-2d}, is {\em not} PP in general, 
	as it fails to meet the discrete DF condition \eqref{div-condition-2d}. 
	This implies the necessity of the discrete DF condition \eqref{div-condition-2d}, which is, unfortunately, not automatically satisfied by the standard CDG method \eqref{eq:CDG-2d-primal}--\eqref{eq:CDG-2d-dual}. 
	In fact, it is difficult to meet condition \eqref{div-condition-2d}, because it depends on coupling the numerical magnetic fields from the four neighboring cells on the dual mesh; see \cref{eq:FigDDF}. 
	If ${\bf B}_{h}^{D}(x,y)$ and ${\bf B}_{h}^{C}(x,y)$ are globally DF (see \cite{Li2011,Li2012} for a globally DF CDG method), then the condition \eqref{div-condition-2d} is met naturally. 
	 Unfortunately, using the local scaling PP limiter to enforce condition \eqref{pp-condition-2d}
	  will destroy the globally DF property. 
	Due to such incompatibility, it is difficult to meet conditions \eqref{pp-condition-2d} and \eqref{div-condition-2d} simultaneously. 
	We will overcome this obstacle in the next section by constructing new locally DF CDG schemes 
	based on the modified MHD equations \eqref{eq:MHD:GP}. 
\end{remark}

\section{New CDG schemes: provably PP and locally DF}\label{section:newpp2d} 
Our analysis in the last section shows that in order to achieve the provably PP property in the  standard 2D CDG framework, we require the corresponding discrete divergence terms ${\rm div}_{ij} {\bf B}_{h}^{D},{\rm div}_{i+\frac{1}{2},j+\frac{1}{2}} {\bf B}_{h}^{C}$ vanish. 
However, as discussed in \cref{rem:DF}, it is difficult to 
meet the discrete DF condition \eqref{pp-condition-2d} and the basic condition \eqref{div-condition-2d} simultaneously. 
In this section, we further propose and analyze a new locally DF CDG method
based on suitable discretization of the modified MHD equations \eqref{eq:MHD:GP} with the extra source term. 
We discover that if the numerical
magnetic fields ${\bf B}_{h}^{D}$ and ${\bf B}_{h}^{C}$ are locally DF within each cell, then a suitable discretization of the source term in \eqref{eq:MHD:GP} can bring some new  
discrete divergence terms which exactly offset ${\rm div}_{ij} {\bf B}_{h}^{D},{\rm div}_{i+\frac{1}{2},j+\frac{1}{2}} {\bf B}_{h}^{C}$ under the locally DF constraint. 
Moreover, the locally DF property is compatible with condition \eqref{div-condition-2d} and thus is not destroyed by the local scaling PP limiter. 
Notice that all our discussions in \cref{section:pp2d,section:newpp2d} are directly extensible to the 3D case.

In order to introduce our new CDG schemes for the modified MHD system \eqref{eq:MHD:GP}, 
we first define two locally DF spaces \cite{Li2005,Yakovlev2013} associated with the overlapping meshes
\begin{equation}\notag
	\mathbb{W}_{h}^{C,k} = \bigg\{ {\bf w} = (w_1, \dots ,w_8)^{\top} \in \mathbb{V}_{h}^{C,k}:~
	\bigg( \dfrac{\partial w_5}{\partial x} + \dfrac{\partial w_6}{\partial y} \bigg) \bigg|_{I_{ij}} = 0
	~~\forall i,j  \bigg\} \,,  
\end{equation}
\begin{equation}\notag
	\mathbb{W}_{h}^{D,k} =  \bigg\{ {\bf u} = (u_1, \dots ,u_8)^{\top} \in \mathbb{V}_{h}^{D,k}:~
	\bigg( \dfrac{\partial u_5}{\partial x} + \dfrac{\partial u_6}{\partial y} \bigg) \bigg|_{I_{i+\frac{1}{2}, j+\frac{1}{2}}} = 0 ~~ \forall i,j \bigg\} \,.
\end{equation}
Different from \cite{Li2005,Yakovlev2013}, our new locally DF CDG method seeks the numerical solutions ${\bf U}_{h}^{C} \in \mathbb{W}_{h}^{C,k}$ and ${\bf U}_{h}^{D} \in \mathbb{W}_{h}^{D,k}$ for the modified MHD system \eqref{eq:MHD:GP} such that 
\begin{align} \label{eq:newCDG-2d-primal} 
	\int_{ I_{ij} } \frac{\partial {\bf U}_{h}^{C}}{\partial t} \cdot {\bf w} {\rm d}x {\rm d}y
	&= {\mathbfcal G}_{ij} \big( {\bf U}_{h}^{C}, {\bf U}_{h}^{D}, {\bf w} \big) 
	+ {\mathbfcal H}_{ij} \big( {\bf B}_{h}^{D}, {\bf S} ( {\bf U}_{h}^{D} ) \cdot {\bf w} \big)
	\quad \forall {\bf w} \in \mathbb{V}_{h}^{C,k}, 
	\\ \label{eq:newCDG-2d-dual} 
	\int_{ I_{i+\frac{1}{2},j+\frac{1}{2}} } \frac{\partial {\bf U}_{h}^{D}}{\partial t} \cdot {\bf u} {\rm d}x {\rm d}y
	& = {\mathbfcal G}_{i+\frac12,j+\frac12} \big( {\bf U}_{h}^{D}, {\bf U}_{h}^{C}, {\bf u} \big) 
	+ {\mathbfcal H}_{i+\frac12,j+\frac12} \big( {\bf B}_{h}^{C}, {\bf S} ( {\bf U}_{h}^{C} ) \cdot {\bf u} \big)
	\quad \forall  {\bf u} \in \mathbb{V}_{h}^{D,k}, 
\end{align}
where ${\mathbfcal G}_{ij} \big( {\bf U}_{h}^{C}, {\bf U}_{h}^{D}, {\bf w} \big)$ and ${\mathbfcal G}_{i+\frac12,j+\frac12} \big( {\bf U}_{h}^{D}, {\bf U}_{h}^{C}, {\bf u} \big)$ are defined in  \eqref{eq:2DCDG-primal-operator}--\eqref{eq:2DCDG-dual-operator}, and 
${\mathbfcal H}_{ij} ( {\bf B}_{h}^{D}, {\bf S} ( {\bf U}_{h}^{D} ) \cdot {\bf w} )$ 
and ${\mathbfcal H}_{i+\frac12,j+\frac12} ( {\bf B}_{h}^{C}, {\bf S} ( {\bf U}_{h}^{C} ) \cdot {\bf u} )$ are suitable numerical approximations (discussed below) to the source terms
$$
\int_{ I_{ij} } \left( -\nabla \cdot {\bf B}_{h}^{D} \right) {\bf S} ( {\bf U}_{h}^{D} ) \cdot {\bf w} {\rm d}x {\rm d}y \quad \mbox{and} \quad  \int_{ I_{i+\frac12,j+\frac12} } \left( -\nabla \cdot {\bf B}_{h}^{C} \right) {\bf S} ( {\bf U}_{h}^{C} ) \cdot {\bf u} {\rm d}x {\rm d}y,
$$
respectively. Since ${\bf U}_{h}^{D} \in \mathbb{W}_{h}^{D,k}$, the numerical magnetic field  
 ${\bf B}_{h}^{D}$ is locally DF within every dual mesh cell. 
As shown in \cref{eq:FigDDF}, a primal mesh cell $I_{ij}$ consists of 
four quarters of dual mesh cells $I_{ij} = \mathop{\cup}_{1\le \ell \le 4} I_{ij}^{\ell} $, 
while ${\bf B}_{h}^{D}$ is locally DF within each of $\{I_{ij}^{\ell}\}_{\ell = 1}^4$. 
Therefore, to measure $\nabla \cdot {\bf B}_{h}^{D}$ on the primal mesh cell $I_{ij}$, 
we only need to consider the jump of normal magnetic component across the dual mesh interfaces 
$\{ (x_i, y): y_{j-\frac12} \le y \le y_{j+\frac12}  \}$ and $\{ (x, y_j): x_{i-\frac12} \le x \le x_{i+\frac12}  \}$ within the primal mesh cell $I_{ij}$; see \cref{eq:FigDDF}. 
Hereafter we employ the standard notations $\jump{ \cdot }$ and $\dgal{ \cdot }$ to respectively denote the jump and the average of the limiting values at a cell interface, for example,   
\begin{align*}
	&\jump{ B_{1,h}^D ( x_i,y )  }:= B_{1,h}^D ( x_i^+,y ) - B_{1,h}^D ( x_i^-,y ) ,  &\jump{ B_{2,h}^D ( x,y_j )  }:= B_{2,h}^D ( x,y_j^+ ) - B_{2,h}^D ( x,y_j^- ),	
	\\
	&\dgal{ {\bf U}_{h}^{D}( x_i,y ) }:= 
	\frac12\big( {\bf U}_{h}^{D}( x_i^-,y ) + {\bf U}_{h}^{D}( x_i^+,y ) \big),  
	&\dgal{ {\bf U}_{h}^{D}( x,y_j ) }:= 
	\frac12\big( {\bf U}_{h}^{D}( x,y_j^- ) + {\bf U}_{h}^{D}( x,y_j^+ ) \big).
\end{align*}
Then we carefully approximate the source term integral as follows:  
\begin{align} \nonumber
	 &  \int_{ I_{ij} } \left( - \nabla \cdot {\bf B}_{h}^{D} \right) {\bf S} ( {\bf U}_{h}^{D} ) \cdot {\bf w} {\rm d}x {\rm d}y 
	   \approx   
	 \int_{y_{j-\frac12}}^{ y_{j+\frac12} }  \left( - \jump{ B_{1,h}^D ( x_i,y )  } \right)
	{\bf S} \left( \dgal{ {\bf U}_{h}^{D}( x_i,y ) } \right) \cdot {\bf w} ( x_i,y )  {\rm d} y	
	\\
	& \qquad + \int_{x_{i-\frac12}}^{ x_{i+\frac12} }  \left( - \jump{ B_{2,h}^D ( x,y_j )  } \right)
	{\bf S} \left( \dgal{ {\bf U}_{h}^{D}( x,y_j ) } \right) \cdot {\bf w} ( x,y_j )  {\rm d} x
=: {\mathbfcal H}_{ij} ( {\bf B}_{h}^{D}, {\bf S} ( {\bf U}_{h}^{D} ) \cdot {\bf w} ). 
\label{eq:Hprimal}
\end{align}
Such a suitable discretization has carefully taken the PP property into account, as it will become clear in the proof of \cref{theorem:new-positivity-2d-high}. 
Similarly, we design  
\begin{align}\nonumber
{\mathbfcal H}_{i+\frac12,j+\frac12} ( {\bf B}_{h}^{C}, {\bf S} ( {\bf U}_{h}^{C} ) \cdot {\bf u} ) &= 
	 \int_{y_{j}}^{ y_{j+1} }  \left( - \jump{ B_{1,h}^C ( x_{i+\frac12},y )  } \right)
{\bf S} \left( \dgal{ {\bf U}_{h}^{C}( x_{i+\frac12},y ) } \right) \cdot {\bf u} ( x_{i+\frac12},y )  {\rm d} y	
\\ \label{eq:Hdual}
&  + \int_{x_{i}}^{ x_{i+1} }  \left( - \jump{ B_{2,h}^C ( x,y_{j+\frac12} )  } \right)
{\bf S} \left( \dgal{ {\bf U}_{h}^{C}( x,y_{j+\frac12} ) } \right) \cdot {\bf u} ( x,y_{j+\frac12} )  {\rm d} x.
\end{align}
Our new semi-discrete locally DF CDG method is defined by 
the weak formulation \eqref{eq:newCDG-2d-primal}--\eqref{eq:newCDG-2d-dual} 
with the approximate source terms 
\eqref{eq:Hprimal}--\eqref{eq:Hdual}. 
It is worth noting that the locally DF property and 
the above source term discretizations \eqref{eq:Hprimal}--\eqref{eq:Hdual} are essential for achieving PP property (see the proof of \cref{theorem:new-positivity-2d-high} and \cref{rem:GPimportance}), which are  
discovered through careful investigation via the GQL approach.

Next, we will present a rigorous PP analysis for our new locally DF CDG method \eqref{eq:newCDG-2d-primal}--\eqref{eq:newCDG-2d-dual} with \eqref{eq:Hprimal}--\eqref{eq:Hdual}. 
With the $N$-point Gauss quadrature rule approximating all the cell interface integrals, the semi-discrete equations 
for the cell averages in our new CDG method \eqref{eq:newCDG-2d-primal}--\eqref{eq:newCDG-2d-dual} can be written as
\begin{equation}\label{eq:newCDG-2d-high}
	\frac{ {\rm d} \overline{ {\bf U} }_{ij}^{C} }{ {\rm d} t} = {\mathbfcal L}_{ij}^{\rm new} \big( {\bf U}_{h}^{C}, {\bf U}_{h}^{D} \big), \qquad \frac{ {\rm d} \overline{ {\bf U} }_{i+\frac{1}{2},j+\frac{1}{2}}^{D} }{ {\rm d} t} = {\mathbfcal L}_{i+\frac{1}{2},j+\frac{1}{2}}^{\rm new} \big( {\bf U}_{h}^{D}, {\bf U}_{h}^{C} \big),
\end{equation}
where ${\mathbfcal L}_{ij}^{\rm new} ( {\bf U}_{h}^{C}, {\bf U}_{h}^{D} ) 
= {\mathbfcal L}_{ij} ( {\bf U}_{h}^{C}, {\bf U}_{h}^{D} )  + 
{\mathbfcal S}_{ij}^D$ and ${\mathbfcal L}_{i+\frac{1}{2},j+\frac{1}{2}}^{\rm new} ( {\bf U}_{h}^{D}, {\bf U}_{h}^{C} ) = {\mathbfcal L}_{i+\frac{1}{2},j+\frac{1}{2}}  ( {\bf U}_{h}^{D}, {\bf U}_{h}^{C} ) + 
{\mathbfcal S}_{i+\frac{1}{2},j+\frac{1}{2}}^C$, with ${\mathbfcal L}_{ij} ( {\bf U}_{h}^{C}, {\bf U}_{h}^{D} )$ and ${\mathbfcal L}_{i+\frac{1}{2},j+\frac{1}{2}}  ( {\bf U}_{h}^{D}, {\bf U}_{h}^{C} )$ defined in \eqref{eq:CDG-2d-primal-high}--\eqref{eq:CDG-2d-dual-high}, and 
\begin{align*}  
		 {\mathbfcal S}_{ij}^D &=   \sum_{\sigma=\pm 1} \sum_{\mu=1}^N \frac{\omega_{\mu}}{2} \Bigg( - \frac{ \jump{ B_{1,h}^D ( x_i,y_{j+\frac{\sigma}4}^{(\mu)} )  } }{\Delta x}  {\bf S} \left( \dgal{ {\bf U}_{h}^{D}( x_i, y_{j+\frac{\sigma}4}^{(\mu)} ) } \right) \Bigg)
		\\
		& \quad + \sum_{\sigma=\pm 1} \sum_{\mu=1}^N \frac{\omega_{\mu}}{2} 
		 \Bigg( - \frac{ \jump{ B_{2,h}^D ( x_{i+\frac{\sigma}4}^{(\mu)}, y_j )  }   }{\Delta y} 
		{\bf S} \left( \dgal{ {\bf U}_{h}^{D}( x_{i+\frac{\sigma}4}^{(\mu)}, y_j ) } \right) \Bigg),
		\\
		{\mathbfcal S}_{i+\frac{1}{2},j+\frac{1}{2}}^C &=   \sum_{\sigma=\pm 1} \sum_{\mu=1}^N \frac{\omega_{\mu}}{2}  \Bigg( - \frac{ \jump{ B_{1,h}^C ( x_{i+\frac12},y_{j+\frac12+\frac{\sigma}4}^{(\mu)} )  } }{\Delta x}  {\bf S} \left( \dgal{ {\bf U}_{h}^{C}( x_{i+\frac12}, y_{j+\frac12+\frac{\sigma}4}^{(\mu)} ) } \right) \Bigg)
		\\
		& \quad + \sum_{\sigma=\pm 1} \sum_{\mu=1}^N \frac{\omega_{\mu}}{2} \Bigg( - \frac{ \jump{ B_{2,h}^C ( x_{i+\frac12+\frac{\sigma}4}^{(\mu)}, y_{j+\frac12} )  }   }{\Delta y} 
		{\bf S} \left( \dgal{ {\bf U}_{h}^{C}( x_{i+\frac12+\frac{\sigma}4}^{(\mu)}, y_{j+\frac12} ) } \right) \Bigg).
\end{align*}

\begin{theorem}[PP property of new locally DF CDG method]\label{theorem:new-positivity-2d-high}	
	Assume $\overline{ {\bf U} }_{ij}^{C},\overline{ {\bf U} }_{i+\frac{1}{2},j+\frac{1}{2}}^{D}\in G$ and that the numerical solutions ${\bf U}_{h}^{C}(x,y), {\bf U}_{h}^{D}(x,y)$ satisfy the condition  \eqref{pp-condition-2d}. 
	Then our new locally DF CDG method \eqref{eq:newCDG-2d-primal}--\eqref{eq:newCDG-2d-dual} with \eqref{eq:Hprimal}--\eqref{eq:Hdual} is PP, namely, for all $i$ and $j$ the updated cell averages satisfy 
\begin{equation}\label{new-2D-PP-cellave}
	\overline{ {\bf U} }_{ij}^{C} + \Delta t {\mathbfcal L}_{ij}^{\rm new} \big( {\bf U}_{h}^{C}, {\bf U}_{h}^{D} \big) \in G,  \quad \overline{ {\bf U} }_{i+\frac{1}{2},j+\frac{1}{2}}^{D} + \Delta t {\mathbfcal L}_{i+\frac{1}{2},j+\frac{1}{2}}^{\rm new} \big( {\bf U}_{h}^{D}, {\bf U}_{h}^{C} \big) \in G  \quad \forall i,j,
\end{equation} 
	under the CFL condition	
	\begin{equation}\label{eq:new-2D-CFL}
		 \frac{  a_1 \Delta t}{\Delta x} +  \frac{  a_2 \Delta t}{\Delta y} 
		< \frac{\theta \hat{\omega}_{1}}{2} \,, \qquad
		\theta = \frac{\Delta t}{\tau_{\max}} \in (0,1] \,,
	\end{equation}
\end{theorem}
where $ a_\ell=\max\{ \hat a_\ell, \beta_\ell\}$, $\ell=1,2$, with $\{\hat a_\ell\}$ defined in \eqref{eq:defa1}--\eqref{eq:defa2} and 
\begin{align*}
&	\beta_1 :=  \max_{i,j,\mu,\sigma} 
\left\{  \frac{ \left|  { \jump{ B_{1,h}^D ( x_i,y_{j+\frac{\sigma}4}^{(\mu)} )  } }  \right| }{2\sqrt{ \dgal{ \rho_{h}^{D}( x_i, y_{j+\frac{\sigma}4}^{(\mu)} ) } }},~  \frac{ \left|  { \jump{ B_{1,h}^C ( x_{i+\frac12},y_{j+\frac12+\frac{\sigma}4}^{(\mu)} )  } }  \right| }{2\sqrt{ \dgal{ \rho_{h}^{C}( x_{i+\frac12}, y_{j+\frac12+\frac{\sigma}4}^{(\mu)} ) } }} 
\right\},
\\
&	\beta_2 :=  \max_{i,j,\mu,\sigma} 
\left\{  
\frac{ \left|  { \jump{ B_{2,h}^D  ( x_{i+\frac{\sigma}4}^{(\mu)}, y_j )  } }  \right| }{2\sqrt{ \dgal{ \rho_{h}^{D}  ( x_{i+\frac{\sigma}4}^{(\mu)}, y_j )  } }},~ 
\frac{ \left|  { \jump{ B_{2,h}^C  ( x_{i+\frac12+\frac{\sigma}4}^{(\mu)}, y_{j+\frac12} )  } }  \right| }{2\sqrt{ \dgal{ \rho_{h}^{C}  ( x_{i+\frac12+\frac{\sigma}4}^{(\mu)}, y_{j+\frac12} )  } }}
\right\}.
\end{align*}

\begin{proof}
	Define ${\bf U}_{\Delta t}^{C,{\rm new}}:=\overline{ {\bf U} }_{ij}^{C} + \Delta t {\mathbfcal L}_{ij}^{\rm new} ( {\bf U}_{h}^{C}, {\bf U}_{h}^{D} ) = {\bf U}_{\Delta t}^{C} + \Delta t {\mathbfcal S}_{ij}^D$, 
	where ${\bf U}_{\Delta t}^{C}=\overline{ {\bf U} }_{ij}^{C} + \Delta t {\mathbfcal L}_{ij} ( {\bf U}_{h}^{C}, {\bf U}_{h}^{D} )$ is the updated cell average of the 2D standard CDG method defined in 
	\cref{theorem:positivity-2d-high}. 
	Because the first component of ${\bf S}({\bf U})$ in \eqref{eq:MHD:GP} is zero, we have 
	${\mathbfcal S}_{ij}^D \cdot {\bf n}_1=0$. From \eqref{eq:2Ddensity} in \cref{theorem:positivity-2d-high}, one obtains    
	${\bf U}_{\Delta t}^{C,{\rm new}} \cdot {\bf n}_1= {\bf U}_{\Delta t}^{C} \cdot {\bf n}_1 + \Delta t {\mathbfcal S}_{ij}^D\cdot {\bf n}_1 =  {\bf U}_{\Delta t}^{C} \cdot {\bf n}_1 >0.$
	
	Next, we will prove ${\bf U}_{\Delta t}^{C,{\rm new}} \cdot {\bf n}^\ast + \frac{ |{\bf B}^*|^2 }2>0$ for auxiliary variables ${\bf v}^{\ast},{\bf B}^{\ast} \in \mathbb R^3$. 
	Notice that 
	\begin{equation}\label{keyg23}
		{\bf U}_{\Delta t}^{C,{\rm new}} \cdot {\bf n}^\ast + \frac{ |{\bf B}^*|^2 }2 
		= \left( {\bf U}_{\Delta t}^{C} \cdot {\bf n}^\ast + \frac{ |{\bf B}^*|^2 }2 \right) 
		+ \Delta t {\mathbfcal S}_{ij}^D \cdot {\bf n}^\ast,
	\end{equation}
    and  
    a tight lower bound of ${\bf U}_{\Delta t}^{C} \cdot {\bf n}^\ast + \frac{ |{\bf B}^*|^2 }2$ has been derived 
    in \eqref{eq:positivity-2d-high-primal} of \cref{theorem:positivity-2d-high}, i.e.,  
\begin{align} \label{eq:estimateU}
	    {\bf U}_{\Delta t}^{C} \cdot {\bf n}^{\ast} + \frac{|{\bf B}^{\ast}|^2}{2}  > 
    {\theta\hat{\omega}_{1}} \bigg(  \frac{ \mathbf{\Pi}_{ij}^{L,-} + \mathbf{\Pi}_{ij}^{1,+} }{2} \cdot {\bf n}^{\ast} + \frac{|{\bf B}^{\ast}|^2}2 \bigg)
    - \Delta t ({\bf v}^{\ast}\cdot{\bf B}^{\ast}) ( {\rm div}_{ij} {\bf B}_{h}^{D} )
\end{align}
    with 
\begin{equation}\label{avePi}
	    \frac{ \mathbf{\Pi}_{ij}^{L,-} + \mathbf{\Pi}_{ij}^{1,+} }{2} = 
    \frac{\lambda_{1}}{\lambda} \sum_{\sigma=\pm 1} \sum_{\mu=1}^N \frac{\omega_{\mu}}{2} 
    \dgal{ {\bf U}_{h}^{D}( {x}_i, {y}_{j+\frac{\sigma}4}^{(\mu)} ) }
    + \frac{\lambda_{2}}{\lambda}  \sum_{\sigma=\pm 1} \sum_{\mu=1}^N  \frac{\omega_{\mu}}{2} 
    \dgal{ {\bf U}_{h}^{D} ( {x}_{i+\frac{\sigma}4}^{(\mu)}, {y}_j ) }.
\end{equation}
    In the following, we will derive a suitable lower bound for $\Delta t {\mathbfcal S}_{ij}^D \cdot {\bf n}^\ast$, which exactly offsets the discrete divergence terms in \eqref{eq:estimateU}. 
	Thanks to \cite[Lemma 7]{WuShu2019}, for any ${\bf U} \in G$ and any $\xi \in \mathbb R$,  it holds that 
	\begin{equation}\label{key000}
		-\xi {\bf S}({\bf U}) \cdot {\bf n}^\ast \ge \xi ( {\bf v}^\ast \cdot {\bf B}^\ast ) - \frac{ |\xi| }{\sqrt{\rho}} \left( {\bf U} \cdot {\bf n}^\ast + \frac{ |{\bf B}^*|^2 }2 \right).
	\end{equation}
	The condition \eqref{pp-condition-2d} ensures   
	${\bf U}_{h}^{D}( x_i^\pm, y_{j+\frac{\sigma}4}^{(\mu)} )  \in G$, 
	which implies the average  
$\dgal{ {\bf U}_{h}^{D}( x_i, y_{j+\frac{\sigma}4}^{(\mu)} ) } \in G$ 
according to the convexity of $G$. 
Applying inequality \eqref{key000} to $\dgal{ {\bf U}_{h}^{D}( x_i, y_{j+\frac{\sigma}4}^{(\mu)} ) }$ and ${ \jump{ B_{1,h}^D ( x_i,y_{j+\frac{\sigma}4}^{(\mu)} )  } } $ gives 
\begin{align*}\nonumber
	&- { \jump{ B_{1,h}^D ( x_i,y_{j+\frac{\sigma}4}^{(\mu)} )  } }  {\bf S} \left( \dgal{ {\bf U}_{h}^{D}( x_i, y_{j+\frac{\sigma}4}^{(\mu)} ) } \right) \cdot {\bf n}^\ast 
	\\
	& \qquad 
	\ge 
	{ \jump{ B_{1,h}^D ( x_i,y_{j+\frac{\sigma}4}^{(\mu)} )  } } 
	( {\bf v}^\ast \cdot {\bf B}^\ast ) 
	- \frac{ \left|  { \jump{ B_{1,h}^D ( x_i,y_{j+\frac{\sigma}4}^{(\mu)} )  } }  \right| }{\sqrt{ \dgal{ \rho_{h}^{D}( x_i, y_{j+\frac{\sigma}4}^{(\mu)} ) } }} \left(  
	\dgal{ {\bf U}_{h}^{D}( x_i, y_{j+\frac{\sigma}4}^{(\mu)} ) } \cdot {\bf n}^\ast + \frac{ |{\bf B}^*|^2 }2 \right)
	\\ 
	 & \qquad \ge  { \jump{ B_{1,h}^D ( x_i,y_{j+\frac{\sigma}4}^{(\mu)} )  } } 
	( {\bf v}^\ast \cdot {\bf B}^\ast ) - 2 \beta_1 \left(  
	\dgal{ {\bf U}_{h}^{D}( x_i, y_{j+\frac{\sigma}4}^{(\mu)} ) } \cdot {\bf n}^\ast + \frac{ |{\bf B}^*|^2 }2 \right). 
\end{align*}
Similarly, one has 
\begin{align*} \nonumber
	&- { \jump{ B_{2,h}^D  ( x_{i+\frac{\sigma}4}^{(\mu)}, y_j )  } }  {\bf S} \left( \dgal{ {\bf U}_{h}^{D}  ( x_{i+\frac{\sigma}4}^{(\mu)}, y_j ) } \right) \cdot {\bf n}^\ast & 
	\\ 
	& \qquad \ge 
	{ \jump{ B_{2,h}^D   ( x_{i+\frac{\sigma}4}^{(\mu)}, y_j )  } } 
	( {\bf v}^\ast \cdot {\bf B}^\ast ) 
	- 2\beta_2 \left(  
	\dgal{ {\bf U}_{h}^{D}  ( x_{i+\frac{\sigma}4}^{(\mu)}, y_j ) } \cdot {\bf n}^\ast + \frac{ |{\bf B}^*|^2 }2 \right).
\end{align*}
Therefore, 
\begin{align}\nonumber
			 \Delta t {\mathbfcal S}_{ij}^D \cdot {\bf n}^\ast & =  \frac{\Delta t}{\Delta x} \sum_{\sigma=\pm 1} \sum_{\mu=1}^N \frac{\omega_{\mu}}{2} \bigg( - { \jump{ B_{1,h}^D ( x_i,y_{j+\frac{\sigma}4}^{(\mu)} )  } }  {\bf S} \left( \dgal{ {\bf U}_{h}^{D}( x_i, y_{j+\frac{\sigma}4}^{(\mu)} ) } \right) \cdot {\bf n}^\ast  \bigg)
	\\ \nonumber
	& \quad + \frac{\Delta t}{\Delta y} \sum_{\sigma=\pm 1} \sum_{\mu=1}^N \frac{\omega_{\mu}}{2} 
	\bigg( - { \jump{ B_{2,h}^D ( x_{i+\frac{\sigma}4}^{(\mu)}, y_j )  }   } 
	{\bf S} \left( \dgal{ {\bf U}_{h}^{D}( x_{i+\frac{\sigma}4}^{(\mu)}, y_j ) } \right) \cdot {\bf n}^\ast  \bigg)
	\\ \nonumber
	& \ge \frac{\Delta t}{\Delta x} \sum_{\sigma=\pm 1} \sum_{\mu=1}^N \frac{\omega_{\mu}}{2} 
	\bigg[ { \jump{ B_{1,h}^D ( x_i,y_{j+\frac{\sigma}4}^{(\mu)} )  } } 
	( {\bf v}^\ast \cdot {\bf B}^\ast ) - 2\beta_1 \bigg(  
	\dgal{ {\bf U}_{h}^{D}( x_i, y_{j+\frac{\sigma}4}^{(\mu)} ) } \cdot {\bf n}^\ast + \frac{ |{\bf B}^*|^2 }2 \bigg)
	\bigg]
	\\ \nonumber
	& \quad + \frac{\Delta t}{\Delta y} \sum_{\sigma=\pm 1} \sum_{\mu=1}^N \frac{\omega_{\mu}}{2}  
	\bigg[ { \jump{ B_{2,h}^D   ( x_{i+\frac{\sigma}4}^{(\mu)}, y_j )  } } 
	( {\bf v}^\ast \cdot {\bf B}^\ast ) 
	- 2\beta_2 \bigg(  
	\dgal{ {\bf U}_{h}^{D}  ( x_{i+\frac{\sigma}4}^{(\mu)}, y_j ) } \cdot {\bf n}^\ast + \frac{ |{\bf B}^*|^2 }2 \bigg) \bigg]
	\\ \nonumber
	& = \Delta t ({\bf v}^{\ast}\cdot{\bf B}^{\ast}) \left( \widetilde {\rm div}_{ij} {\bf B}_{h}^{D}  \right)
	- \frac{2\beta_1 \Delta t}{\Delta x} \sum_{\sigma=\pm 1} \sum_{\mu=1}^N \frac{\omega_{\mu}}{2}  
	\left(  
	\dgal{ {\bf U}_{h}^{D}( x_i, y_{j+\frac{\sigma}4}^{(\mu)} ) } \cdot {\bf n}^\ast + \frac{ |{\bf B}^*|^2 }2 \right)
	\\ \label{eq:estS}
	& \quad - \frac{2\beta_2 \Delta t}{\Delta y} \sum_{\sigma=\pm 1} \sum_{\mu=1}^N \frac{\omega_{\mu}}{2} 
	\left(  
	\dgal{ {\bf U}_{h}^{D}  ( x_{i+\frac{\sigma}4}^{(\mu)}, y_j ) } \cdot {\bf n}^\ast + \frac{ |{\bf B}^*|^2 }2 \right)
\end{align}
with 
\begin{equation}\label{def:newDivB}
	\widetilde {\rm div}_{ij} {\bf B}_{h}^{D} :=  \sum_{\sigma=\pm 1} \sum_{\mu=1}^N  \frac{\omega_{\mu}}{2} 
	\left( \frac{ \jump{ B_{1,h}^D ( x_i,y_{j+\frac{\sigma}4}^{(\mu)} )  } } { \Delta x }
	 +   
	  \frac{  \jump{ B_{2,h}^D   ( x_{i+\frac{\sigma}4}^{(\mu)}, y_j )  }  }{\Delta y} \right).  
\end{equation}
Substituting the estimates \eqref{eq:estS} and \eqref{eq:estimateU} with \eqref{avePi} into \eqref{keyg23}, 
we obtain  
\begin{equation}\label{key633}
	{\bf U}_{\Delta t}^{C,{\rm new}} \cdot {\bf n}^\ast + \frac{ |{\bf B}^*|^2 }2  >  \Phi 
	+ \Delta t ({\bf v}^{\ast}\cdot{\bf B}^{\ast}) \left( \widetilde {\rm div}_{ij} {\bf B}_{h}^{D}  - {\rm div}_{ij} {\bf B}_{h}^{D} \right)
\end{equation}
with 
\begin{align*}
	 \Phi  & := \left( \theta \hat \omega_1  \frac{\lambda_1}{\lambda} -\frac{2\beta_1 \Delta t}{\Delta x}  \right) 
\sum_{\sigma=\pm 1} \sum_{\mu=1}^N \frac{\omega_{\mu}}{2}  
\left(  
\dgal{ {\bf U}_{h}^{D}( x_i, y_{j+\frac{\sigma}4}^{(\mu)} ) } \cdot {\bf n}^\ast + \frac{ |{\bf B}^*|^2 }2 \right)
\\
& + \left( \theta \hat \omega_1  \frac{\lambda_2}{\lambda} -\frac{2\beta_2 \Delta t}{\Delta y}  \right) 
\sum_{\sigma=\pm 1} \sum_{\mu=1}^N \frac{\omega_{\mu}}{2}  
\left(  
\dgal{ {\bf U}_{h}^{D}  ( x_{i+\frac{\sigma}4}^{(\mu)}, y_j ) } \cdot {\bf n}^\ast + \frac{ |{\bf B}^*|^2 }2 \right).
\end{align*}
Under the CFL condition \eqref{eq:new-2D-CFL}, we have $\theta \hat \omega_1  \frac{\lambda_1}{\lambda} \ge 2 \lambda_1 =  \frac{ 2a_1 \Delta t}{\Delta x} \ge  \frac{2\beta_1 \Delta t}{\Delta x}  $, and similarly, $\theta \hat \omega_1  \frac{\lambda_2}{\lambda} \ge \frac{2\beta_2 \Delta t}{\Delta y}$. 
Hence $\Phi\ge 0$, and then the estimate \eqref{key633} yields 
\begin{equation}\label{key634}
	{\bf U}_{\Delta t}^{C,{\rm new}} \cdot {\bf n}^\ast + \frac{ |{\bf B}^*|^2 }2  >  
	\Delta t ({\bf v}^{\ast}\cdot{\bf B}^{\ast}) \left( \widetilde {\rm div}_{ij} {\bf B}_{h}^{D}  - {\rm div}_{ij} {\bf B}_{h}^{D} \right).
\end{equation}
Combining \eqref{def:newDivB} with \eqref{eq:divijD} gives  
\begin{align*}
			{\rm div}_{ij} {\bf B}_{h}^{D} - \widetilde {\rm div}_{ij} {\bf B}_{h}^{D}
	&=  \frac{1 }{\Delta x}  \sum_{\sigma=\pm 1} \sum_{\mu=1}^N  \frac{\omega_{\mu} }{2}
	\Big( B_{1,h}^{D} {\color{red} (  x_{i+\frac{1}{2}}, {y}_{j+\frac{\sigma}4}^{(\mu)}  )}  - B_{1,h}^{D} {\color{red}(  x_{i-\frac{1}{2}},  {y}_{j+\frac{\sigma}4}^{(\mu)}  )} \Big)  \\
	&+  \frac{ 1 }{\Delta y}  \sum_{\sigma=\pm 1} \sum_{\mu=1}^N   \frac{\omega_{\mu} }{2}
	\Big( B_{2,h}^{D} {\color{red} (  {x}_{i+\frac{\sigma}4}^{(\mu)} ,y_{j+\frac{1}{2}}  )}  - B_{2,h}^{D} {\color{red}(   {x}_{i+\frac{\sigma}4}^{(\mu)} , y_{j-\frac{1}{2}}  )} \Big) 
	\\
	& - \sum_{\sigma=\pm 1} \sum_{\mu=1}^N  \frac{\omega_{\mu}}{2} 
	\Bigg( \frac{ \jump{ B_{1,h}^D {\color{blue} (  x_i,y_{j+\frac{\sigma}4}^{(\mu)} )}  } } { \Delta x }
	+   
	\frac{  \jump{ B_{2,h}^D  {\color{blue} (  x_{i+\frac{\sigma}4}^{(\mu)}, y_j )}  }  }{\Delta y} \Bigg),
\end{align*}
where for clarification we have colored the points which correspond to the red and blue points illustrated in \cref{eq:FigDDF} for $N=2$.  
{\bf A key observation is that thanks to the locally DF property of ${\bf B}_{h}^{D}$, the two discrete divergence operators 
	$\widetilde {\rm div}_{ij}$ and ${\rm div}_{ij}$ are exactly equivalent for ${\bf B}_{h}^{D}$.} 
In fact, using 
the exactness of $N$-point Gauss quadrature ($N=k+1$) for polynomials of degree $k$, we have 
\begin{align*}
	{\rm div}_{ij} {\bf B}_{h}^{D} - \widetilde {\rm div}_{ij} {\bf B}_{h}^{D}
	=   \frac{1}{\Delta x \Delta y} \bigg( &
	\int_{y_{j-\frac12}}^{ y_{j+\frac12} } \left( B_{1,h}^{D} {\color{red} (  x_{i+\frac{1}{2}}, y  )}  - B_{1,h}^{D} {\color{red}(  x_{i-\frac{1}{2}},  y  )} \right) {\rm d}y 
	\\
	 + & \int_{x_{i-\frac12}}^{ x_{i+\frac12} }  \left( B_{2,h}^{D} {\color{red} (  x ,y_{j+\frac{1}{2}}  )}  - B_{2,h}^{D} {\color{red}(  x , y_{j-\frac{1}{2}}  )} \right)  {\rm d} x
	\\
	 - & \int_{y_{j-\frac12}}^{ y_{j+\frac12} }  \jump{ B_{1,h}^D {\color{blue} (  x_i,y )}  }  {\rm d}y 
	- \int_{x_{i-\frac12}}^{ x_{i+\frac12} }  \jump{ B_{2,h}^D  {\color{blue} (  x, y_j )}  }   {\rm d} x
	\bigg)
	\\
	= \frac{1}{\Delta x \Delta y} \bigg( & \sum_{\ell = 1}^4 \int_{ \partial I_{ij}^\ell }  {\bf B}_h^D \cdot   {\bm n}_{\partial I_{ij}^\ell} {\rm d} s  \bigg) = \frac{1}{\Delta x \Delta y} \sum_{\ell = 1}^4 \iint_{  I_{ij}^\ell } 
	\nabla \cdot {\bf B}_h^D {\rm d}x{\rm d}y, 
\end{align*}
where we have utilized the {\em divergence theorem} within the four subcells $I_{ij}^{\ell}$ shown \cref{eq:FigDDF}, namely, $I_{ij}^{1}=[x_{i},x_{i+\frac12}] \times [y_{j},y_{j+\frac12}]$, 
$I_{ij}^{2}=[x_{i-\frac12},x_{i}] \times [y_{j},y_{j+\frac12}]$, $I_{ij}^{3}=[x_{i-\frac12},x_{i}] \times [y_{j-\frac12},y_{j}]$ and $I_{ij}^{4}=[x_{i},x_{i+\frac12}] \times [y_{j-\frac12},y_{j}]$. Since ${\bf B}_h^D$ is locally DF, we have $\nabla \cdot {\bf B}_h^D=0$ within each of these four subcells. 
Thus ${\rm div}_{ij} {\bf B}_{h}^{D} - \widetilde {\rm div}_{ij} {\bf B}_{h}^{D}=0$. It then follows from \eqref{key634} that 
 ${\bf U}_{\Delta t}^{C,{\rm new}} \cdot {\bf n}^\ast + \frac{ |{\bf B}^*|^2 }2>0$ for any auxiliary variables ${\bf v}^{\ast},{\bf B}^{\ast} \in \mathbb R^3$. This together with ${\bf U}_{\Delta t}^{C,{\rm new}} \cdot {\bf n}_1>0$ implies ${\bf U}_{\Delta t}^{C,{\rm new}} \in G_*=G$, according to the GQL representation in \cref{theo:eqDefG}. 
Similarly, one can show $\overline{ {\bf U} }_{i+\frac{1}{2},j+\frac{1}{2}}^{D} + \Delta t {\mathbfcal L}_{i+\frac{1}{2},j+\frac{1}{2}}^{\rm new} ( {\bf U}_{h}^{D}, {\bf U}_{h}^{C} ) \in G$. 
The proof is completed. 
\end{proof}

\begin{remark}\label{rem:GPimportance}
	As seen from the proof of \cref{theorem:new-positivity-2d-high}, 
	the locally DF property and the suitable source term discretizations 
	\eqref{eq:Hprimal}--\eqref{eq:Hdual} are essential for achieving 
	the PP property. 
	Our carefully discretized source terms \eqref{eq:Hprimal}--\eqref{eq:Hdual} provide the 
	 discrete divergence terms $\Delta t ({\bf v}^{\ast}\cdot{\bf B}^{\ast}) \widetilde {\rm div}_{ij} {\bf B}_{h}^{D} $, 
	 which, under the locally DF constraint, exactly cancel out the ``superfluous''  discrete divergence terms $-\Delta t ({\bf v}^{\ast}\cdot{\bf B}^{\ast})  {\rm div}_{ij} {\bf B}_{h}^{D} $ arising from the standard CDG method. 
	  The GQL approach with auxiliary variables has played a critical role in the above PP analysis and numerical design.   
\end{remark}

\begin{remark}
  The estimate wave speed $ a_\ell=\max\{ \hat a_\ell, \beta_\ell\}$ in  \cref{theorem:new-positivity-2d-high} is comparable to the standard one $a_\ell^{\rm std}:= \max\{ \| {\mathscr{R}}_\ell ({\bf U}^C_h)\|_\infty,\|{\mathscr{R}}_\ell({\bf U}_h^D) \|_\infty \} $. 
  In fact, for smooth solutions, 
  one has $\hat a_\ell \le a_\ell^{\rm std} + {\mathcal O}(h)$ from  \eqref{eq:defa1}--\eqref{eq:defa2} and \eqref{eq:aaaaWKL}, where $h=\max\{\Delta x,\Delta y\}$, 
  and $\beta_\ell = {\mathcal O}(h^{k+1})$ is much smaller than $\hat a_\ell$, so that  $a_\ell \le a_\ell^{\rm std} + {\mathcal O}(h)$. 
  Even in the discontinuous cases, $\beta_\ell$ does not cause strict restriction on $\Delta t$, 
  as justified theoretically by \cref{WU:prop} and verified numerically. 
  Moreover, our numerical results in \cref{section:numerical-example} show that our CDG 
  schemes with a standard CFL number are still PP in most cases, which indicates the theoretical CFL condition  \eqref{eq:new-2D-CFL} is sufficient rather than necessary.  
\end{remark}

\begin{proposition}\label{WU:prop}
	For any ${\bf U}, \tilde{\bf U} \in G$, define $\dgal{ \rho }:=\frac12 (\rho + \tilde \rho)$ and $\jump{ B_\ell} :=  \tilde B_\ell - B_\ell $, then it holds that 
	$$
 \frac{ \left|\jump{ B_\ell} \right| }{2 \sqrt{  \dgal{ \rho } } } \le \frac12 \alpha_{\ell} ( {\bf U}, \tilde{\bf U} ), \qquad \ell \in \{1,2,3\}.
	$$
\end{proposition}

\begin{proof}
	Using Jensen's inequality for the concave function $\sqrt{x}$ gives $\sqrt{ \dgal{ \rho } }  \ge \frac12 ( \sqrt{ \rho } + \sqrt{ \tilde \rho } )$. Thus 
\begin{equation}\label{keyeeww}
		 \frac{ \left|\jump{ B_\ell} \right| }{2 \sqrt{  \dgal{ \rho } } } \le \frac{ \left|\jump{ B_\ell} \right| }{ \sqrt{ \rho } + \sqrt{ \tilde \rho } }  \le \frac{ | {\bf B}- \tilde{\bf B} | } { \sqrt{ \rho } + \sqrt{ \tilde \rho } }. 
\end{equation}
On the other hand, the first inequality in \eqref{keyeeww} also implies that 
\begin{equation}\label{key3366}
	\frac{ \left|\jump{ B_\ell} \right| }{2 \sqrt{  \dgal{ \rho } } } \le \frac{ \left|\jump{ B_\ell} \right| }{ \sqrt{ \rho } + \sqrt{ \tilde \rho } }  \le \frac{ | B_\ell | + | \tilde B_\ell | }{ \sqrt{ \rho } + \sqrt{ \tilde \rho } } \le \max \left\{ \frac{| B_\ell |}{\sqrt{ \rho } }, \frac{| \tilde B_\ell |}{\sqrt{ \tilde \rho } } \right\} \le \max \left\{ {\mathcal C}_\ell,  \tilde {\mathcal C}_\ell \right\},
\end{equation}
where 
the last step follows from 
\begin{align*}
	\frac{ |B_\ell| }{\sqrt{\rho}} & \le  \sqrt{ \max \left\{ \frac{ |{\bf B}|^2 }{\rho}, {\mathcal C}_s   \right\} } =   \left[ \frac12\left(  \frac{ |{\bf B}|^2}{\rho} + {\mathcal {C}}_s^2 + \bigg|  \frac{ |{\bf B}|^2}{\rho} - {\mathcal {C}}_s^2 \bigg|  
	\right) \right]^\frac12
	\\
	& = \frac{1}{\sqrt{2}}  \left[   \frac{ |{\bf B}|^2}{\rho} + {\mathcal {C}}_s^2 + \sqrt{ \left(   \frac{ |{\bf B}|^2}{\rho} + {\mathcal {C}}_s^2 \right)^2 - 4 \frac{ | {\bf B} |^2 {\mathcal {C}}_s^2 }{\rho}  } \right]^\frac12 
 \le  {\mathcal {C}}_\ell.
\end{align*}
Combining \eqref{keyeeww} with \eqref{key3366} gives 
	$\frac{ \left|\jump{ B_\ell} \right| }{2 \sqrt{  \dgal{ \rho } } } 
	\le \frac12 \left( \max \big\{ {\mathcal C}_\ell,  \tilde {\mathcal C}_\ell \big\} + \frac{ | {\bf B}- \tilde{\bf B} | } { \sqrt{ \rho } + \sqrt{ \tilde \rho } }  \right) \le \frac12 \alpha_{\ell} ( {\bf U}, \tilde{\bf U} ).$
\end{proof}

\section{Numerical experiments}\label{section:numerical-example}
This section carries out several benchmark or demanding numerical tests on 1D and 2D MHD problems to verify the accuracy, robustness, and effectiveness of the proposed (locally) DF and PP CDG methods. We focus on the proposed third-order accurate PP CDG schemes ($k=2$) coupled with the explicit third-order accurate SSP Runge--Kutta time discretization \cite{GottliebShuTadmor2001}. Unless mentioned otherwise, we use the ideal EOS $p=(\gamma-1)\rho e$ with $\gamma  = 5/3$, the CFL number of $0.25$, and  
 $\theta = {\Delta t}/{\tau_{\max}}  = 1$. 

\subsection{1D  near-vacuum Riemann problem}
Consider 
a Riemann problem from \cite{Christlieb}. Its initial conditions, which involve very low density and low pressure, are given by 
\begin{equation*}
	(\rho, p, {\bf v}, {\bf B})(x,0)=
	\begin{cases}
		(10^{-12}, ~ 10^{-12}, ~ 0,  ~ 0,  ~ 0, ~ 0,  ~ 0,  ~ 0) \,,  ~  &  x < 0\,,  \\
		(1, ~ 0.5, ~ 0,  ~ 0,  ~ 0,  ~ 0,  ~ 1,  ~ 0) \,,  &  x > 0  \,.
	\end{cases} 
\end{equation*}
The computational domain is $[-0.5, 0.5]$ with outflow 
boundary conditions. Figure \ref{fig:Vacuum}  displays the density and thermal pressure at $t = 0.1$ simulated by our 
PP CDG method with $100$ cells, 
along with a reference solution with $1000$ cells. 
One can observe that the near-vacuum wave structures well captured by our scheme and agree with the results reported in \cite{Christlieb,WuShu2019}. Our numerical scheme maintains the positivity of density and pressure and is very robust in the whole simulation.

\begin{figure}[thbp]
	\centering
		\includegraphics[width=0.35\textwidth]
		{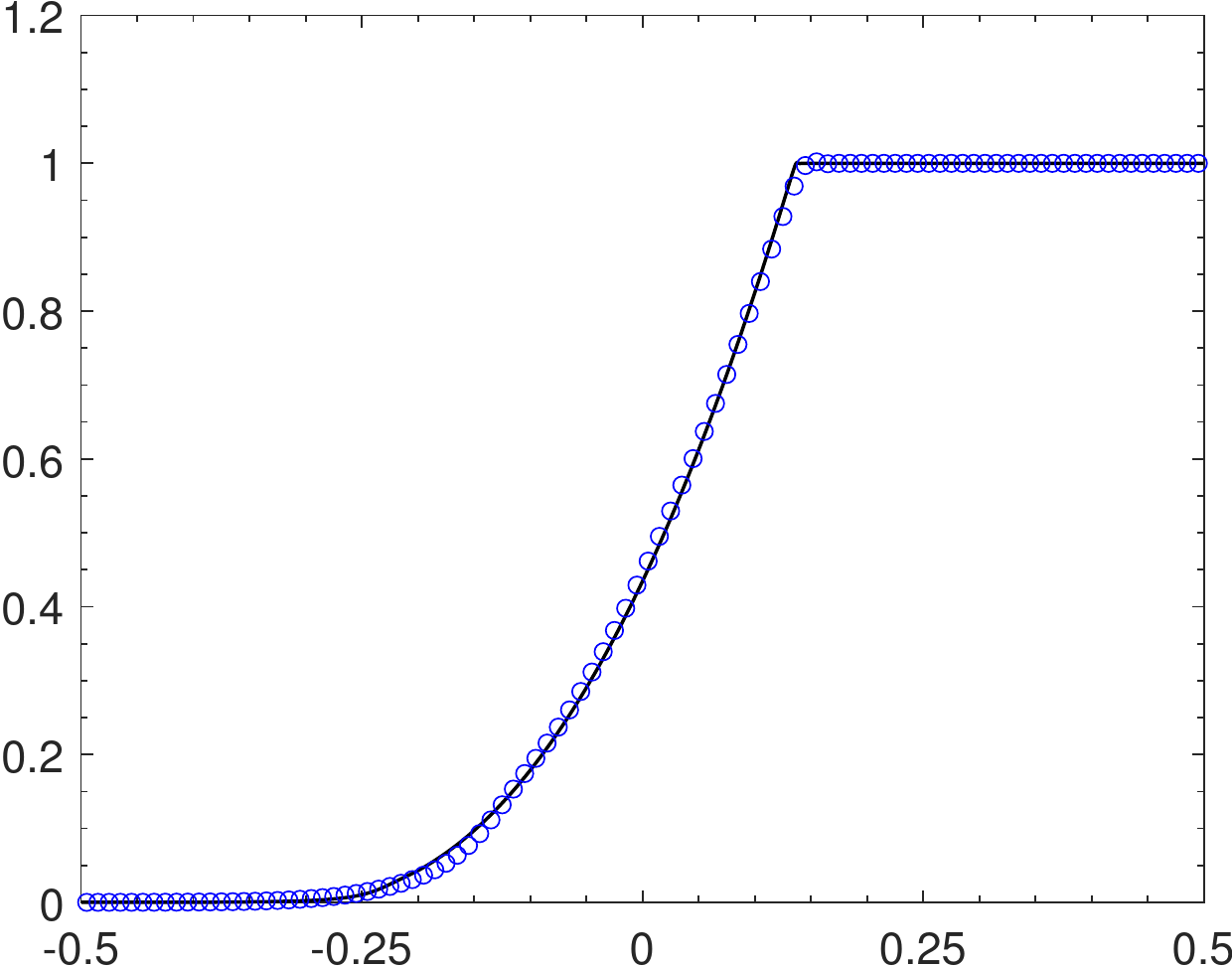}\qquad \qquad 
		\includegraphics[width=0.35\textwidth]
		{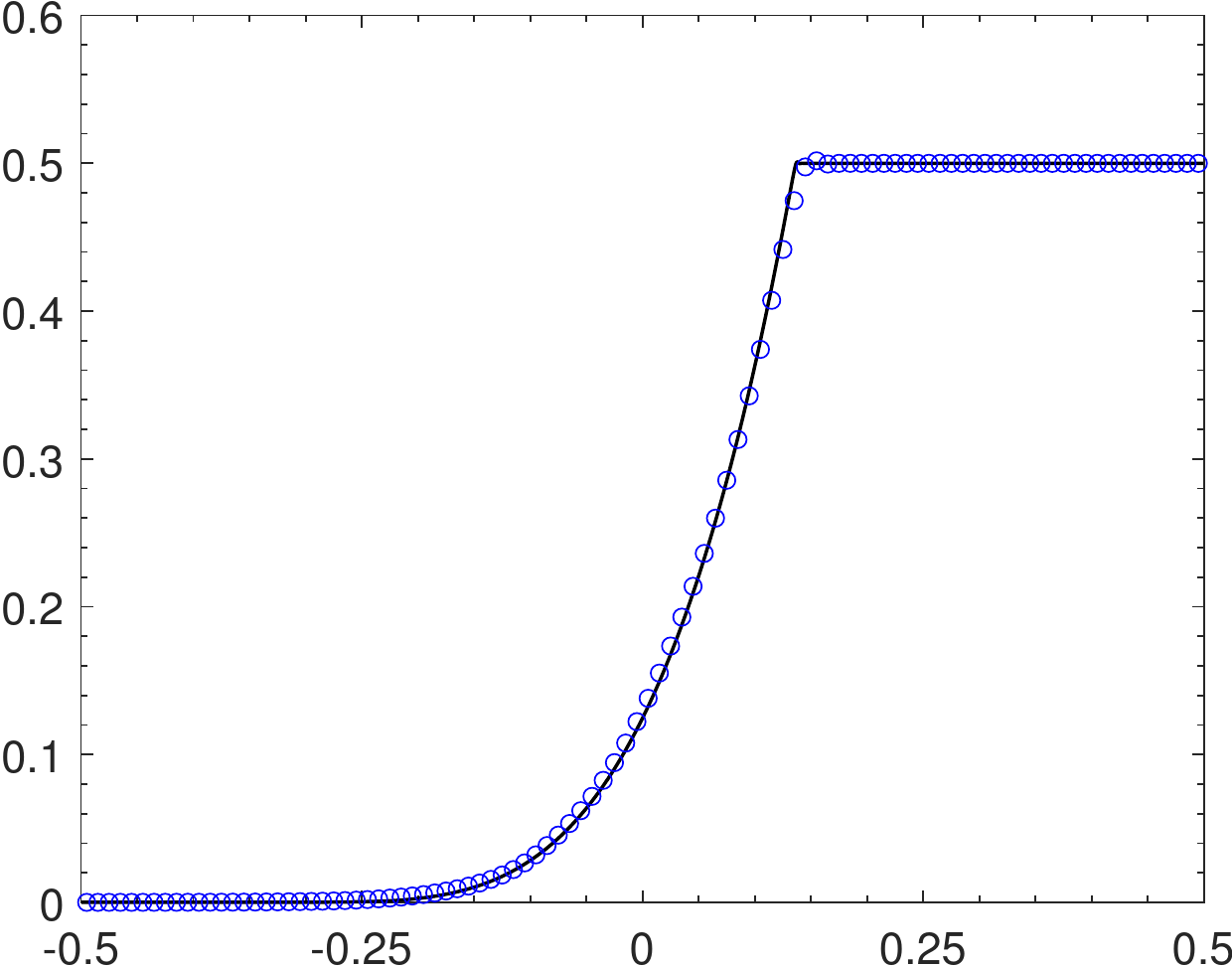}
		\captionsetup{belowskip=-12pt}
	\caption{Near-vacuum Riemann problem: density (left) and pressure (right) computed by the third-order PP CDG scheme with $100$ cells (circles)  and $1000$ cells (solid lines),
		respectively.}
	\label{fig:Vacuum}
\end{figure}

\subsection{Vortex problem with low pressure}
This example simulates a smooth MHD vortex problem \cite{Christlieb,WuShu2018} with very low pressure in the domain $[-10, 10]^2$ with periodic
boundary conditions.
The initial conditions are $(\rho, {\bf v}, p, {\bf B})(x,y,0) = (1, 1+\delta v_{1}, 1+\delta v_{2}, 0, 1+\delta p, \delta B_{1}, \delta B_{2}, 0)$ 
with vortex perturbations 
$( \delta v_{1} ,  \delta v_{2} ) = \dfrac{\mu}{\sqrt{2}\pi}e^{0.5(1-r^2)}(-y,x)$, $\delta p = - \dfrac{\mu^2(1+r^2)}{8\pi^2}e^{1-r^2}$, and 
$( \delta B_{1} ,  \delta B_{2} ) = \dfrac{\mu}{2\pi}e^{0.5(1-r^2)}(-y,x)
$, 
where $r = \sqrt{x^2 + y^2}$,  
and the vortex strength is set as $\mu = 5.389489439$. The lowest thermal pressure 
 is very small (about $5.3\times10^{-12}$) in the vortex center. 
As such, the CDG method would fail due to negative pressure, if we do not enforce the condition \eqref{pp-condition-2d} 
with the PP limiter. 
To assess the accuracy, 
we list in Table \ref{tab:Vortex} the errors in the momentum and the magnetic field  
at $t = 0.05$ for our third-order locally DF PP scheme. 
The results confirm that the  third order of convergence is 
achieved in $l^1$ norm. 

\begin{table}[h]
	\centering
	\captionsetup{belowskip=-3pt}
	\caption{Vortex problem: $l^{1}$ errors at $t = 0.05$ and the approximate rates of convergence 
		for the third-order locally DF PP CDG scheme.}
	\begin{tabular}{c|c|c|c|c|c|c|c|c}
		\hline Mesh  &\multicolumn{2}{|c|}{$m_1$}  &\multicolumn{2}{|c|}{$m_2$ }  &\multicolumn{2}{|c|}{$B_1$ }
		&\multicolumn{2}{|c}{$B_2$ }    \\
		\hline   $N \times N$  &$l^{1}$-error    & Rate      &$l^{1}$-error     & Rate       &$l^{1}$-error      & Rate  
		&$l^{1}$-error    & rate \\
		\hline $10\times10$  &4.65e-3	& -- 	&4.66e-3	& -- 	&3.34e-3	& -- 	&3.34e-3	& -- 	 \\
		\hline $20\times20$  &8.39e-4	& 2.47	&8.36e-4	& 2.48	&5.89e-4	&2.50	&5.89e-4	&2.50	 \\
		\hline $40\times40$  &1.16e-4	& 2.85	&1.16e-4	& 2.85	&8.14e-5	& 2.86	&8.14e-5	& 2.86	 \\
		\hline $80\times80$  &1.21e-5	& 3.27	&1.20e-5	& 3.27	&8.55e-6	& 3.25	&8.55e-6	& 3.25	 \\
		\hline $160\times160$ &1.28e-6	& 3.24	&1.27e-6	& 3.24	&9.04e-7	& 3.24	&9.04e-7	& 3.24	 \\
		\hline $320\times320$ &1.49e-7	& 3.10	&1.49e-7	& 3.10	&1.06e-7	& 3.10	&1.06e-7	&3.10    \\
		\hline $640\times 640$ &1.85e-8	& 3.01	&1.85e-8	& 3.01	&1.30e-8	&3.02	&1.30e-8	&3.02    \\
		\hline
	\end{tabular}
	\label{tab:Vortex}
\end{table}

We also quantitatively investigate the numerical  divergence error in the magnetic field. As in \cite{WuShu2020NumMath}, we measure the global relative divergence error in ${\bf B}_{h}^{C}$ on the primal mesh ${\mathcal T}_h^C$ by   
\begin{equation}
	\varepsilon_{\rm div} =  {\lVert \nabla\cdot{\bf B}_{h}^{C} \lVert}/{\lVert {\bf B}_{h}^{C} \lVert} \,,
\end{equation}
with  
\begin{align*}
	& \lVert \nabla\cdot{\bf B}_{h}^{C} \lVert  ~ := \sum\limits_{ {\mathcal E}_h^C \in {\mathcal T}_h^C } 
	\int_{ {\mathcal E}_h^C } \left| \jump{\langle{ {\bf n}, \bf B}_{h}^{C} \rangle} \right| {\rm d}s + 
	\sum\limits_{ i,j } 
	\int_{ I_{ij} } | \nabla\cdot{\bf B}_{h}^{C} | {\rm d}x{\rm d}y \,,
	\\
	&	\lVert {\bf B}_{h}^{C} \lVert  ~ := \sum\limits_{ {\mathcal E}_h^C  \in {\mathcal T}_h^C } 
	\int_{ {\mathcal E}_h^C } \dgal{ |{\bf B}_{h}^{C} | } {\rm d}s + \sum\limits_{ i,j } 
	\int_{ I_{ij} } \left| {\bf B}_{h}^{C} \right| {\rm d}x{\rm d}y \,.
\end{align*}
where $\jump{\langle{ {\bf n}, \bf B}_{h}^{C} \rangle}$ denotes the jump of the normal component of  ${\bf B}_{h}^{C}$ across the cell interfaces ${\mathcal E}_h^C$ of the primal mesh $\in {\mathcal T}_h^C$. 
Table \ref{tab:VortexDiv} lists the global divergence errors $\varepsilon_{\rm div}$ computed at different grid resolutions. 
It is seen that the errors $\varepsilon_{\rm div}$ decreases, as the mesh refines, at an approximately third-order rate.

\begin{table}[h]
	\centering
	\captionsetup{belowskip=-5pt}
	\caption{Vortex problem: global divergence errors $\varepsilon_{\rm div}$ at $t = 0.05$ and the approximate rates of convergence 
		for the third-order locally DF PP CDG scheme with increasing grid resolution.}
	\begin{tabular}{c|c|c|c|c|c|c|c}
		\hline   Mesh  & $10 \times 10$  &  $20\times 20$      & $40\times 40$  &  $80 \times 80$      & $160 \times 160$      &  $320 \times 320$
		& $640 \times 640$     \\
		\hline  $\varepsilon_{\rm div}$ &1.04e-1	& 2.13e-2	& 3.48e-3	&  4.56e-4	& 5.92e-5	& 7.58e-6	& 9.58e-7		 \\
		\hline  Rate & --	& 2.28	& 2.62	& 2.93	& 2.94	& 2.96 & 2.98		 \\
		\hline
	\end{tabular}
	\label{tab:VortexDiv}
\end{table}

\subsection{Orszag-Tang problem}

The Orszag--Tang problem \cite{Li2005} is a benchmark test for MHD codes. Although it does not involve low  pressure or density, we take it to verify the effectiveness and correct resolution of our scheme. 
The initial solution is given by $\rho = \gamma^2$, ${\bf v} = ( - \sin y, \sin x,  0)$, ${\bf B} = (- \sin y,   \sin 2x,  0)$, and $p  = \gamma$. 
The computational domain $\Omega = [0, 2\pi]\times[0, 2\pi]$ is divided into $200\times200$ cells with periodic boundary conditions on 
$\partial\Omega$. 
Figure \ref{fig:Orszag-Tang} plots the contours of $\rho$ at $t = 0.5$ 
and $t = 2$ computed by our third-order locally DF CDG method. 
As time evolves, the initial smooth flow develops into the complicated structures involving multiple shocks.  
 Our results are in good agreement with those  in \cite{Li2005,WuShu2018} by the non-central DG schemes, and the wave structures are correctly captured with high resolution by  
 our new locally DF CDG method. 
 Figure \ref{fig:Div-OT} displays 
the time evolution of the global divergence error $\varepsilon_{\rm div}$. We find that the magnitude of $\varepsilon_{\rm div}$ is kept below $10^{-3}$ during the whole simulation.

\begin{figure}[htbp]
	\centering
	{
		\includegraphics[width=0.39\textwidth]
		{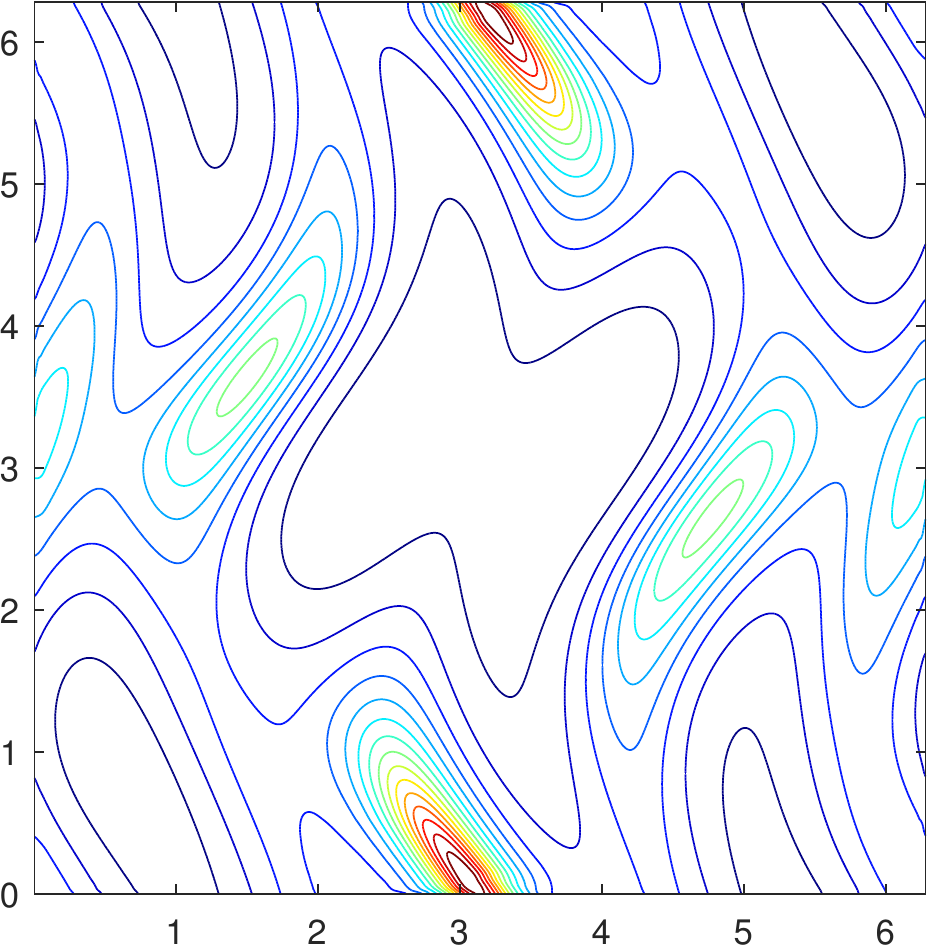}
	}
	\qquad \quad 
	{
		\includegraphics[width=0.39\textwidth]
		{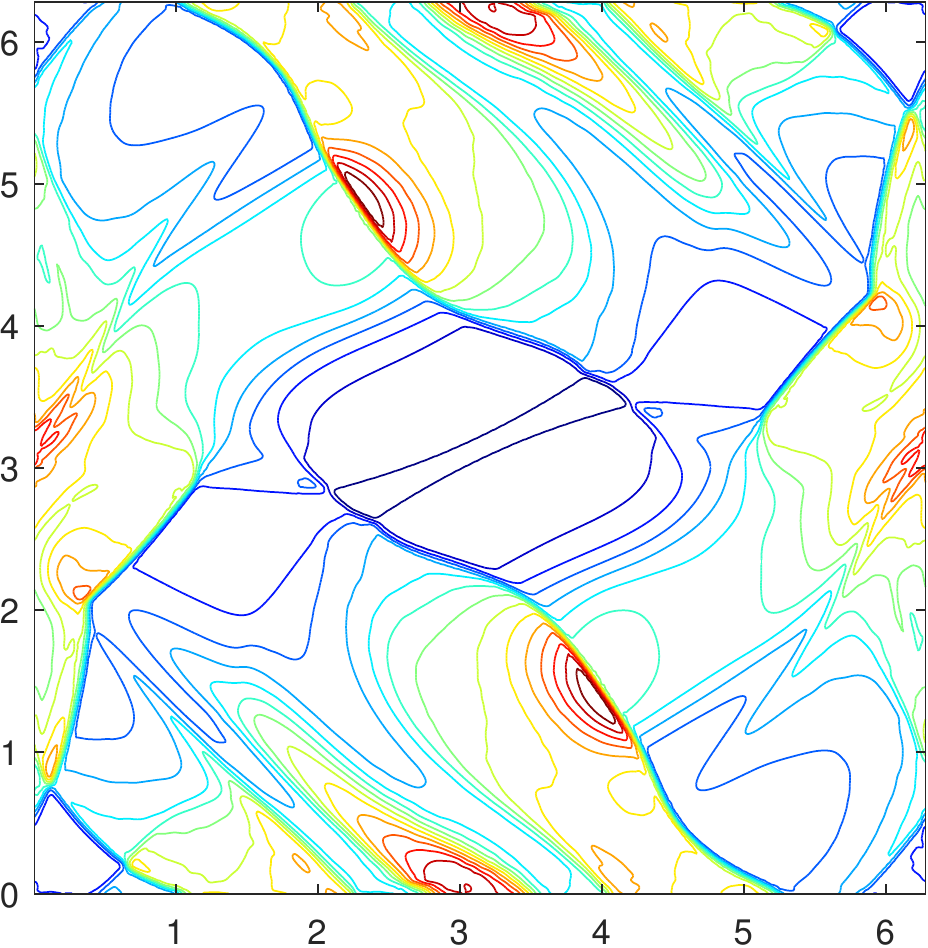}
	}
\captionsetup{belowskip=-12pt}
	\caption{Orszag--Tang problem:  density at $t = 0.5$ (left) and $t = 2$ (right). }
	\label{fig:Orszag-Tang}
\end{figure}

\begin{figure}[htbp]
	\centering
	\begin{subfigure}[b]{0.3\textwidth}
		\centering
		\includegraphics[width=\textwidth]
		{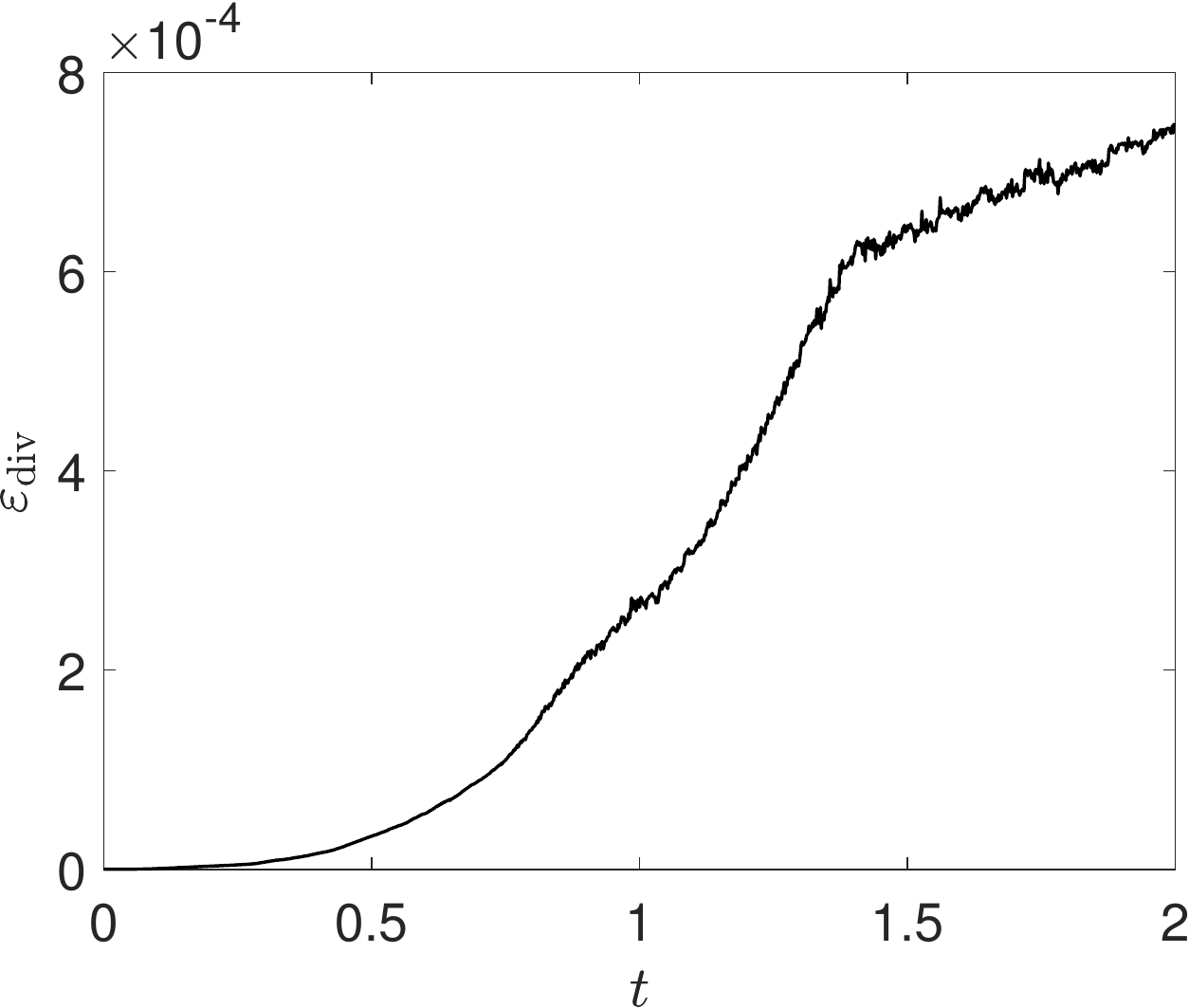}
		\captionsetup{belowskip=-8pt}
		\caption{Orszag--Tang problem}
		\label{fig:Div-OT}
	\end{subfigure}
	\hfill 
	\begin{subfigure}[b]{0.32\textwidth}
		\centering
		\includegraphics[width=\textwidth]
		{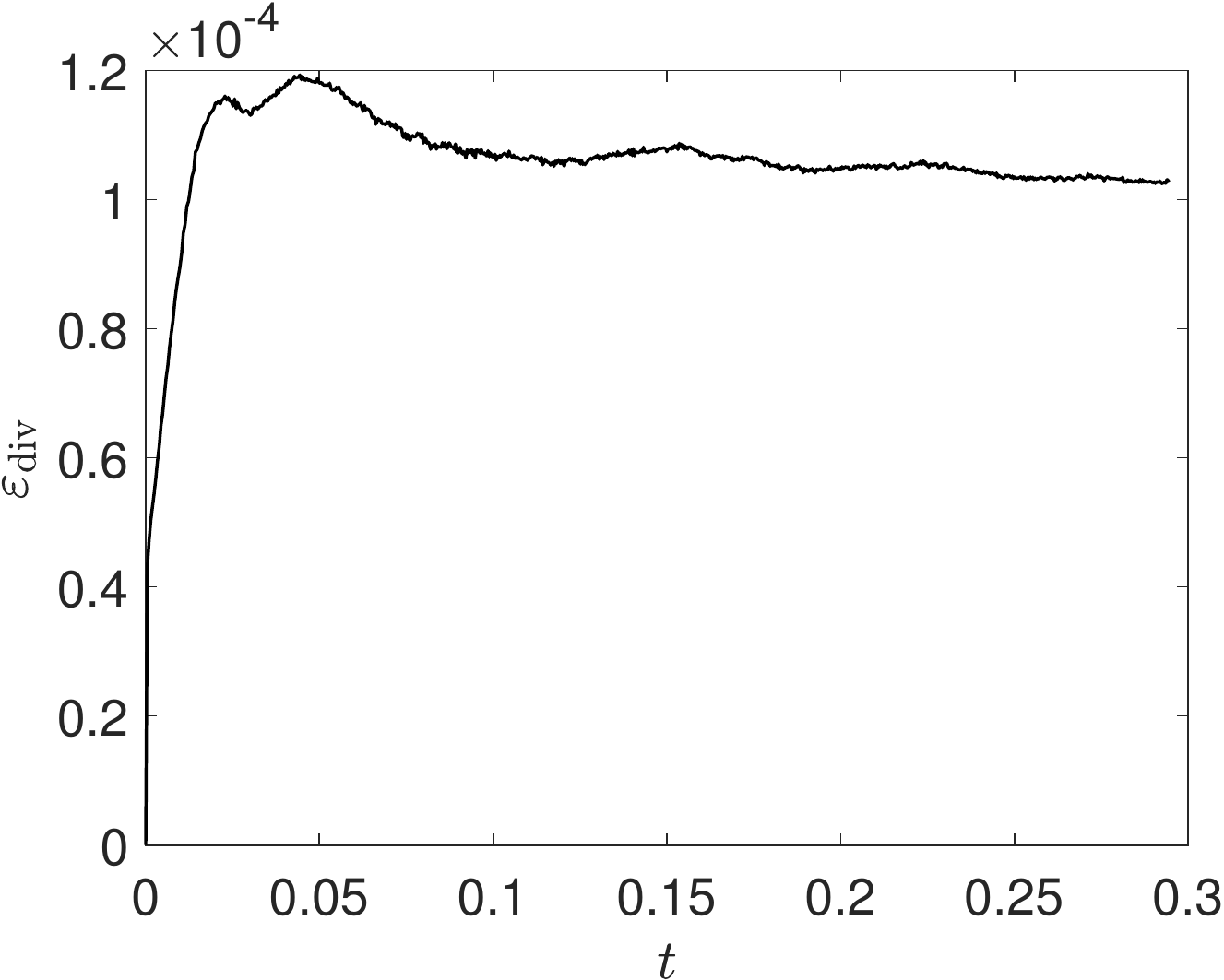}
		\captionsetup{belowskip=-8pt}
		\caption{Rotor problem}
		\label{fig:Div-RO}
	\end{subfigure}
	\begin{subfigure}[b]{0.325\textwidth}
		\centering
		\includegraphics[width=\textwidth]
		{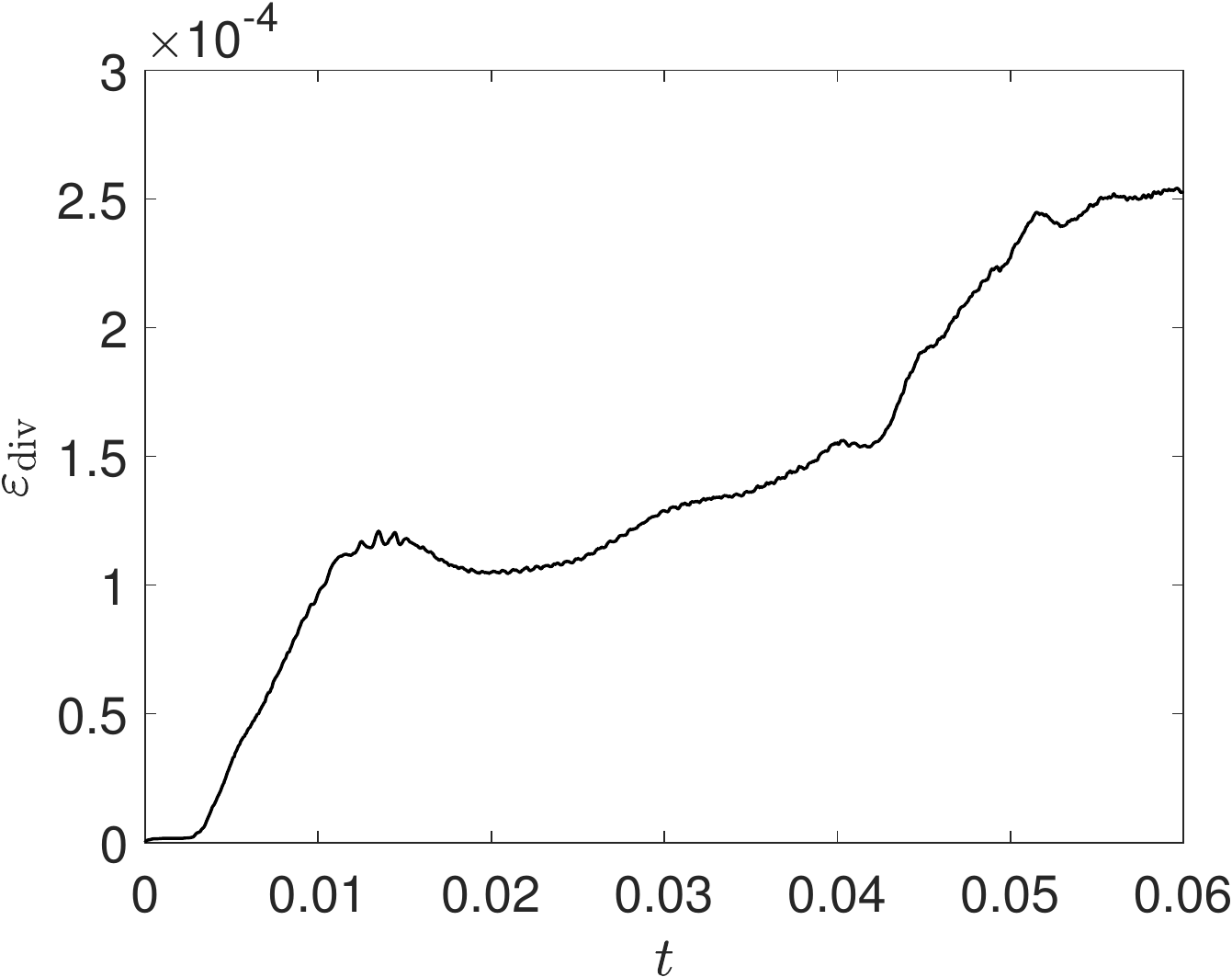}
		\captionsetup{belowskip=-8pt}
		\caption{Shock-cloud interaction}
		\label{fig:Div-SC}
	\end{subfigure}
	\captionsetup{belowskip=-12pt}
	\caption{Time evolution of the global divergence error $\varepsilon_{\rm div}$. } 
	\label{fig:3Div}
\end{figure}

\subsection{Rotor problem}
This is also a benchmark test \cite{BalsaraSpicer1999}, which describes a dense disk of fluid 
rotating in a ambient fluid, with the initial conditions given by  
\[
( p, ~ v_3, ~ B_1, ~ B_2, ~ B_3 ) = (0.5, ~ 0, ~ 2.5/\sqrt{4\pi}, ~ 0, ~ 0  ), 
\]
and
\begin{equation*} 
	( \rho, ~ v_1, ~ v_2 ) =
	\begin{cases}
		( 10, ~ -(y - 0.5)/r_0, ~ (x - 0.5)/r_0 )   \qquad  &{\rm if} ~~ r < r_0 \,,  \\
		( 1 + 9\lambda, ~ -\lambda (y - 0.5)/r, ~ \lambda(x - 0.5)/r )   \qquad  &{\rm if} ~~  r_0 < r < r_1 \,,  \\
		( 1, ~ 0, ~ 0 )   \qquad  &{\rm if} ~~ r > r_1 \,, 
	\end{cases} 
\end{equation*}
with $r = \sqrt{ (x-0.5)^2 + (y-0.5)^2 }$, $r_{0} = 0.1$, $r_{1} = 0.115$, $\lambda = (r_{1} - r) / (r_{1} - r_{0})$. The computational domain $\Omega = [ 0,1]\times[0,1]$ is divided into $200\times200$ uniform cells with 
outflow boundary conditions on $\partial\Omega$.
Figure \ref{fig:Rotor-Problem} shows the contour plots of the thermal pressure $p$ and the Mach number $|{\bf v}|/c_s$ 
at $t = 0.295$. 
Our results are consistent with those reported in \cite{BalsaraSpicer1999,WuShu2018}. 
Figure \ref{fig:Div-RO} plots the global divergence error $\varepsilon_{\rm div}$, which remains 
small and at order $\mathcal{O}(10^{-4})$.

\begin{figure}[htbp]
	\centering
	{
		\includegraphics[width=0.43\textwidth]
		{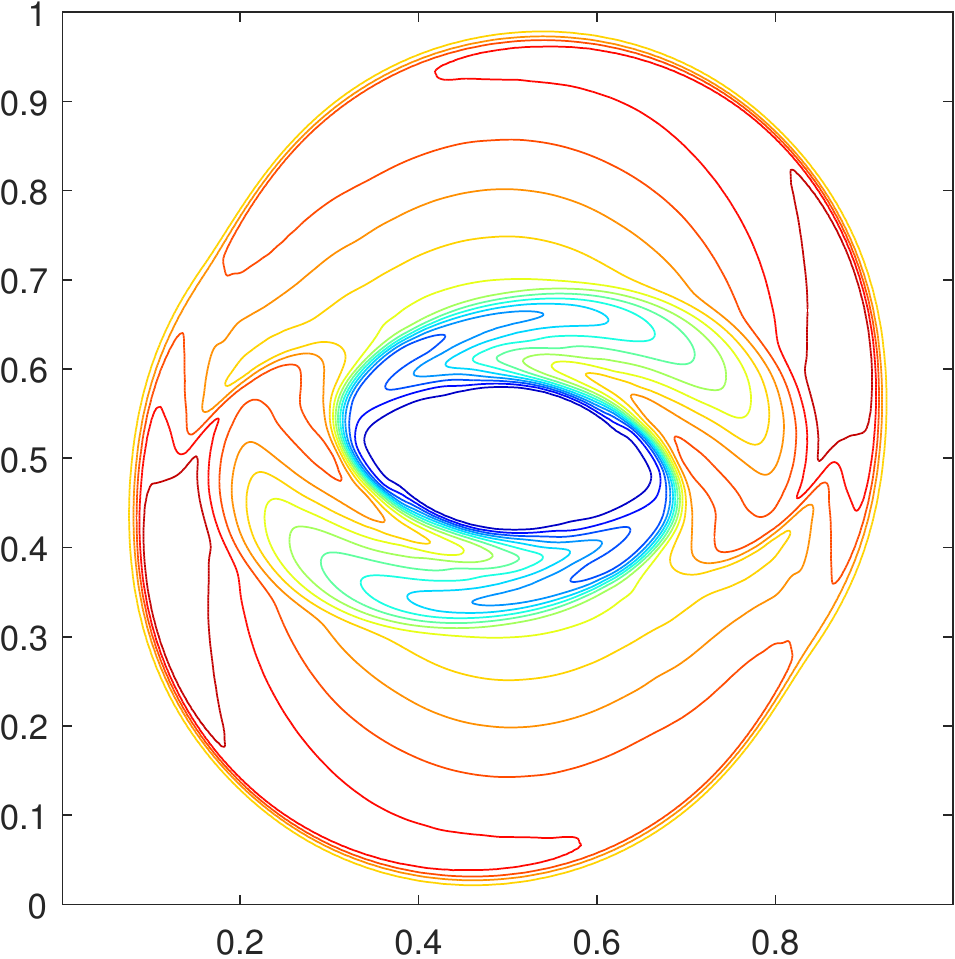}
	}
\quad 
	{
		\includegraphics[width=0.43\textwidth]
		{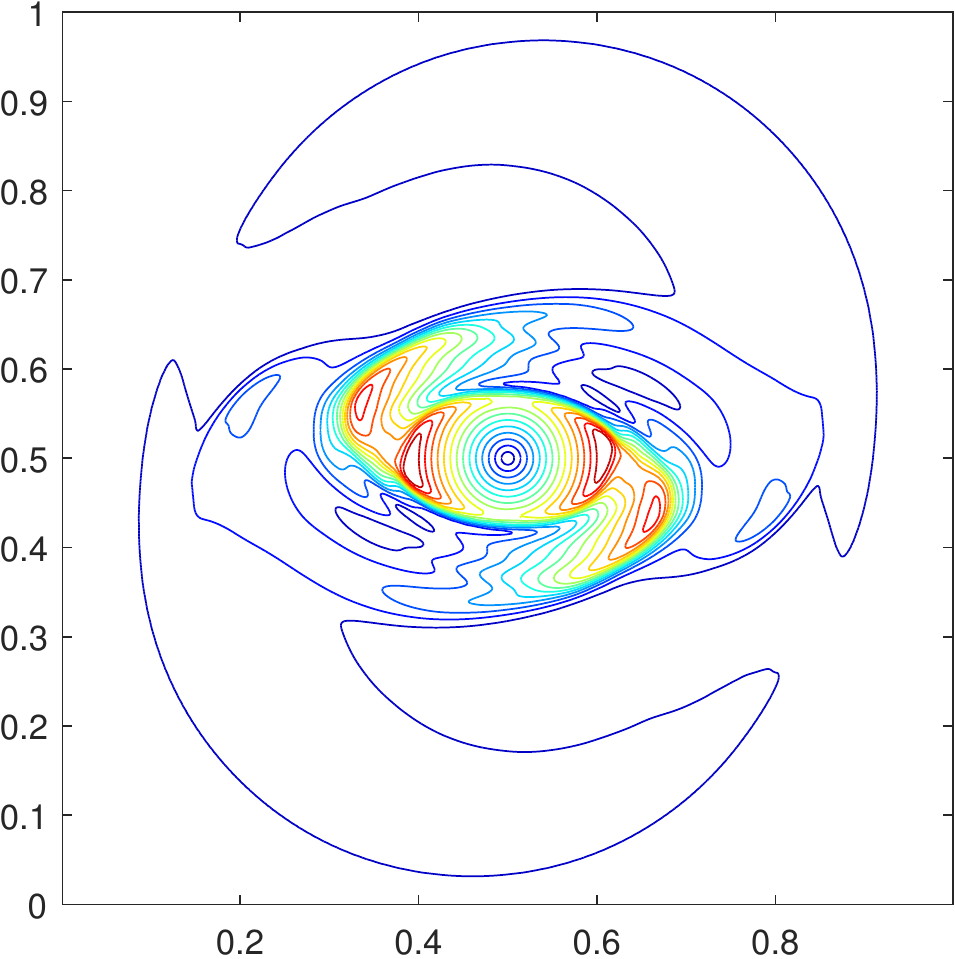}
	}
\captionsetup{belowskip=-12pt}
	\caption{Rotor problem: Contour plots of the thermal pressure (left) and Mach number (right) 
		at $t = 0.295$. }
	\label{fig:Rotor-Problem}
\end{figure}

\subsection{Shock cloud interaction} 
This test simulates the interaction of a high density cloud and a strong shock wave. 
It was originally introduced in \cite{Dai1998} and has become a benchmark for examining MHD schemes \cite{Toth2000,WuShu2018,WuShu2019}. 
Initially, there is a strong shock at $x = 0.6$, which is parallel to the $y$-axis. The left and right states of the shock are specified as 
$\rho_L = 3.86859$, $p_L=167.345$, ${\bf v}_L={\bf 0}$, ${\bf B}_L=(0, 2.1826182, -2.1826182)$, 
$\rho_R=1$, $p_R=1$, ${\bf v}_R=(-11.2536,0,0)$, and ${\bf B}_R=(0,0.56418958, 0.56418958)$, with a rotational discontinuity in the magnetic field. 
In front of the shock, a stationary circular cloud of radius $0.15$ is centered at $(0.8, 0.5)$. 
The cloud has a higher density of $10$ and the same pressure and magnetic field as the surrounding plasma.  
The computational domain $\Omega = [0, 1]^2$ is divided into $400\times400$ uniform rectangular cells, with the 
inflow condition on the 
 right boundary and the outflow conditions on the others. 
Figure \ref{fig:Shock-Cloud} presents the numerical thermal pressure and the magnitude of the magnetic pressure 
at $t = 0.06$ simulated by our locally DF PP CDG method. 
It is observed that the complicated flow structures and the discontinuities are resolved and agree with the 
results computed in \cite{Toth2000,WuShu2018,WuShu2019}. 
Figure \ref{fig:Div-SC} shows the evolution of the global divergence error $\varepsilon_{\rm div}$, which remains 
small and at order $\mathcal{O}(10^{-4})$. 
We also notice that if we do not enforce condition \eqref{pp-condition-2d} 
by using the PP limiter, the CDG solution will go outside the set $G$ and break down at time $t \approx 0.0366$.

\begin{figure}[htbp]
	\centering
	{
		\includegraphics[width=0.43\textwidth]
		{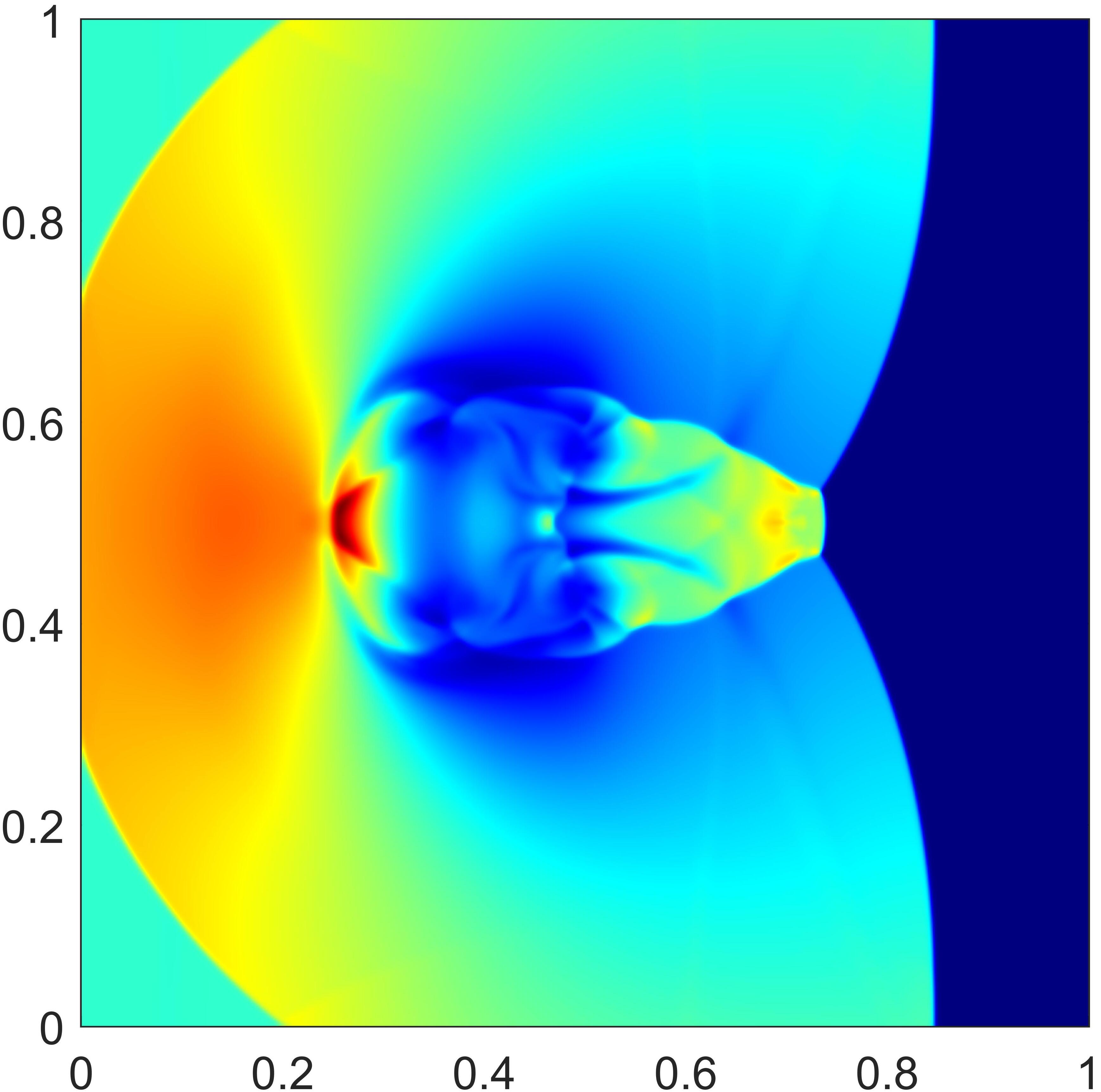}
	}
	\quad 
	{
		\includegraphics[width=0.43\textwidth]
		{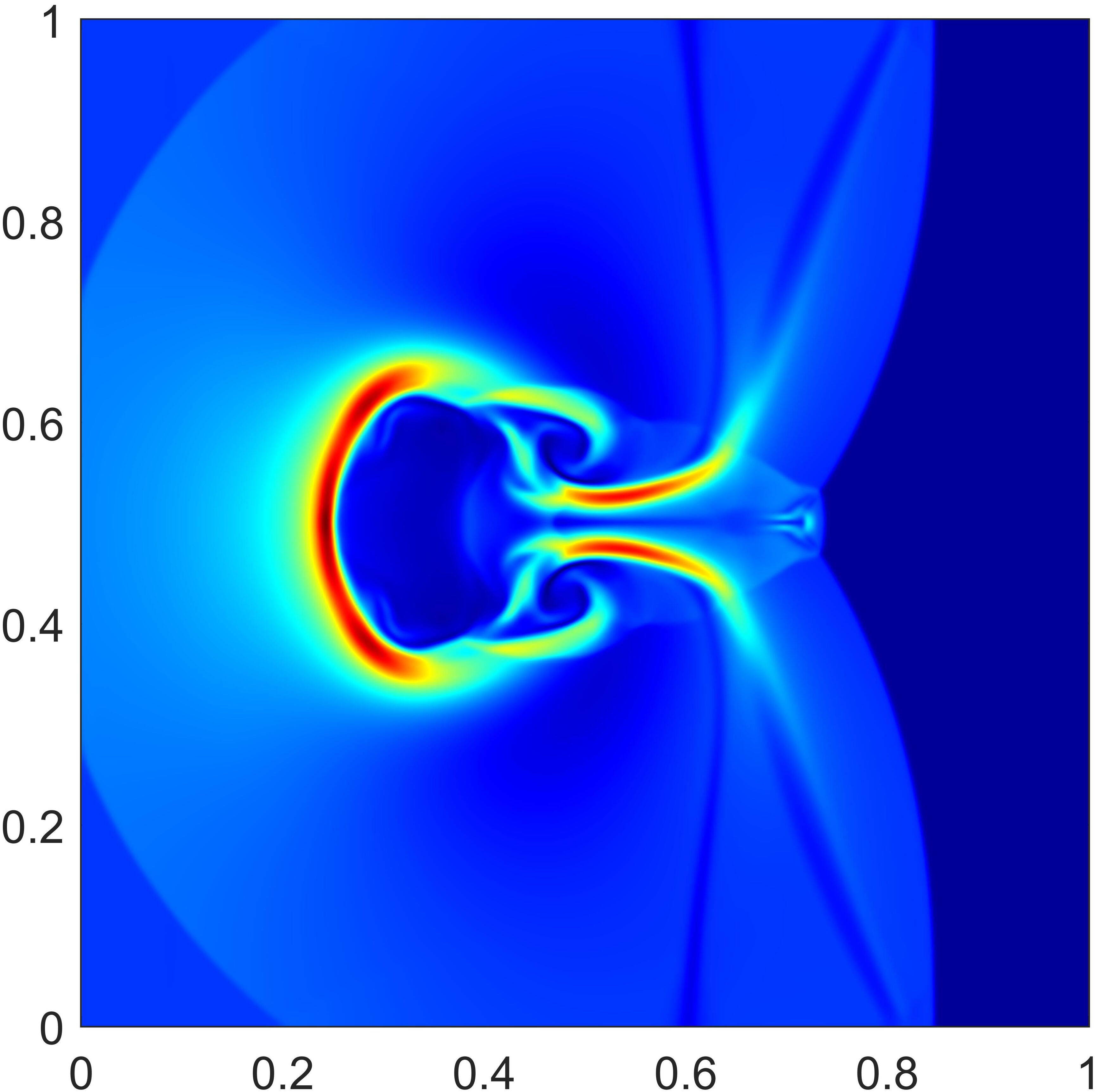}
	}
\captionsetup{belowskip=-12pt}
	\caption{Shock cloud interaction: the thermal pressure (left) and the magnitude of magnetic field (right). }
	\label{fig:Shock-Cloud}
\end{figure}

\subsection{Blast problems} 
The classical MHD blast wave  problem, originally proposed in \cite{BalsaraSpicer1999},  
represents a quite demanding test widely adopted to examine 
the positivity of numerical MHD schemes; see \cite{BalsaraSpicer1999,Christlieb,Wu2017a,WuShu2018,WuShu2019,WuShu2020NumMath}. 
The adiabatic index is taken as $\gamma = 1.4$. 
The computational domain is $\Omega = [-0.5,0.5]^2$ with outflow boundary conditions on $\partial\Omega$. 
Initially, $\Omega$ is filled with stationary fluid with ${\bf v}={\bf 0}$, $\rho=1$, and ${\bf B}=(B_0,0,0)$. The initial pressure $p$ is piecewise constant and has a circular jump on $x^2 + y^2=0.1^2$, with $p=p_e$ inside the circle and $p=0.1$ outside. 
We consider two blast problems: the classical version \cite{BalsaraSpicer1999} with 
\{$p_e = 10^3, B_0 = 100/\sqrt{4\pi} \}$,  
 and a much more extreme version \cite{WuShu2018} with $\{p_e = 10^4, B_0 = 1000/\sqrt{4\pi} \}$ (larger discontinuity in $p$ and stronger magnetic field).  
The plasma-beta $\beta$ is very small for both cases ($\beta \approx 2.51 \times 10^{-4} $ for the classical blast problem and $\beta = 2.51\times10^{-6}$ for the extreme blast problem), rendering their simulations highly challenging.  
Our locally DF PP CDG method works very robustly for both blast problems. The numerical results computed on the mesh of $200\times200$ cells are given in 
Figure \ref{fig:Blast12}. 
One can see that, for the classical blast problem, our simulation results are in good agreement with those reported in 
\cite{BalsaraSpicer1999,  Christlieb, Li2011,Wu2017a,WuShu2018,WuShu2019}, 
and our density profile does not have the numerical oscillations that were observed in \cite{BalsaraSpicer1999,  Christlieb}. 
Our flow patterns of the extreme blast problem are consistent with those in \cite{WuShu2018} simulated by a PP non-central DG method. 
{\em It is noticed that without the proposed PP techniques, the CDG code would break down quickly within a few time steps.}

\begin{figure}[htbp]
	\centering
		\includegraphics[width=0.32\textwidth]
		{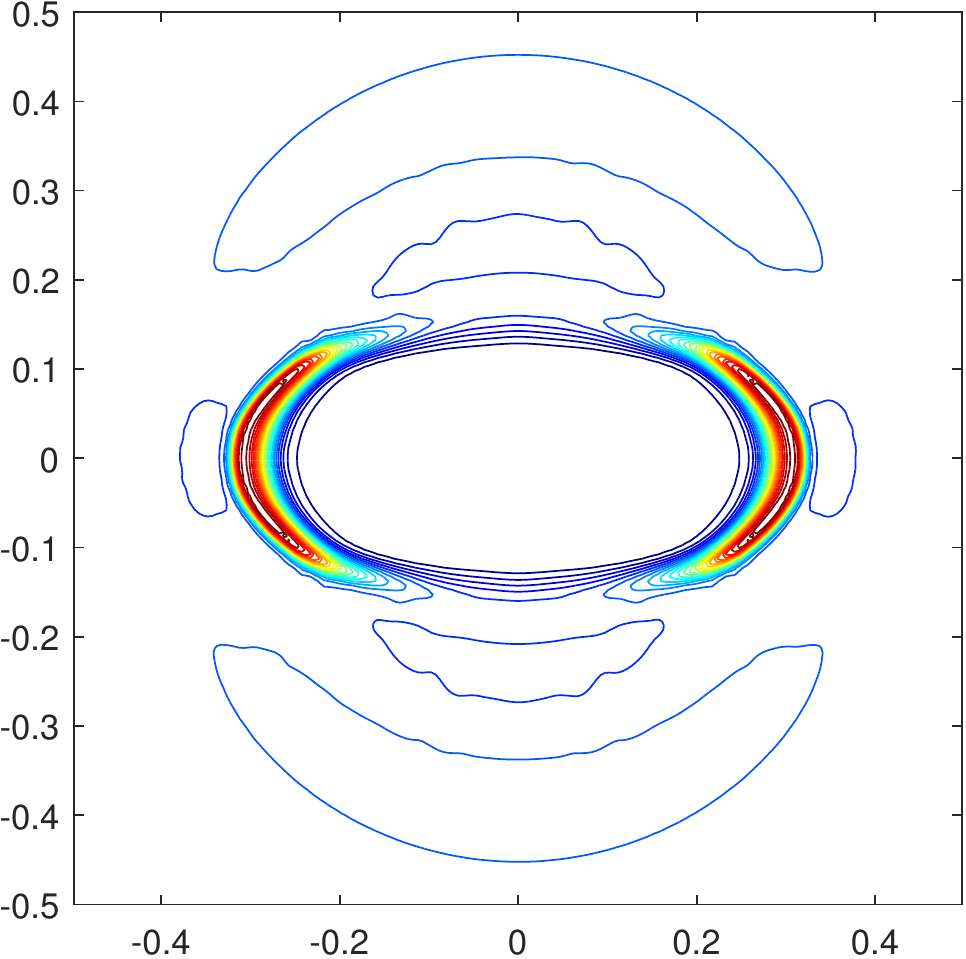}
		\includegraphics[width=0.32\textwidth]
		{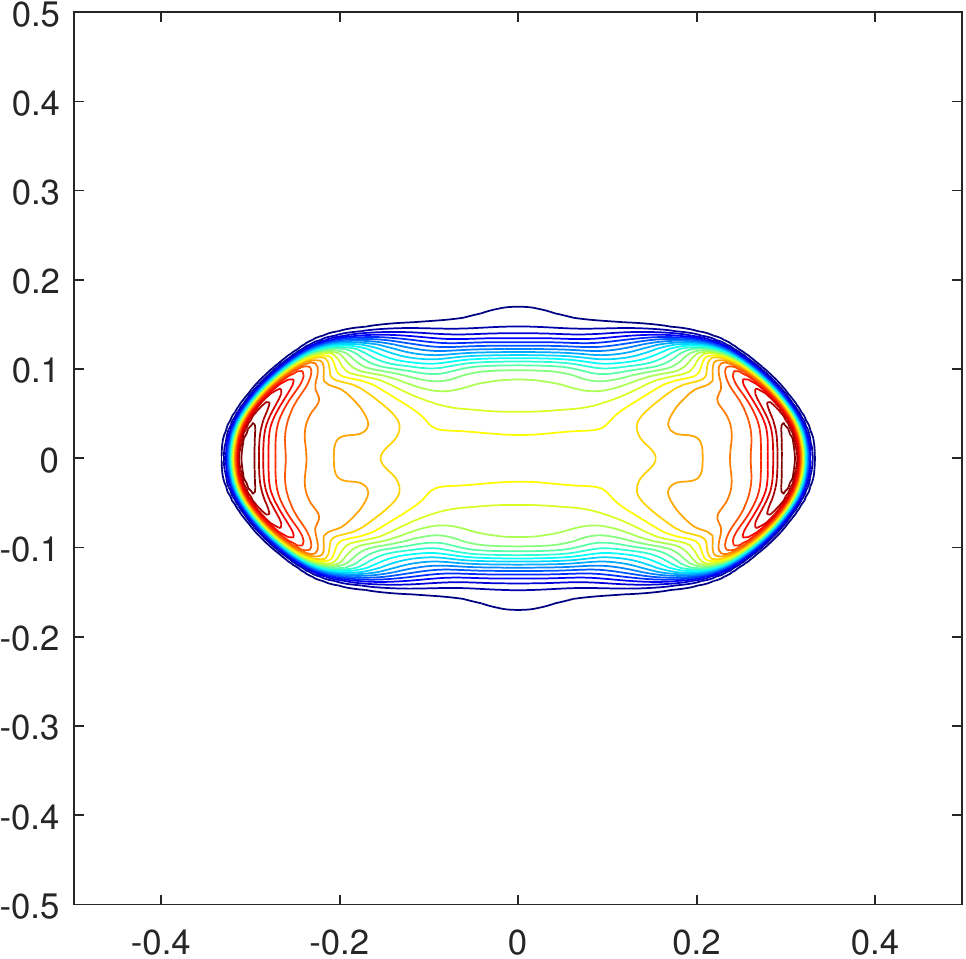}
		\includegraphics[width=0.32\textwidth]
		{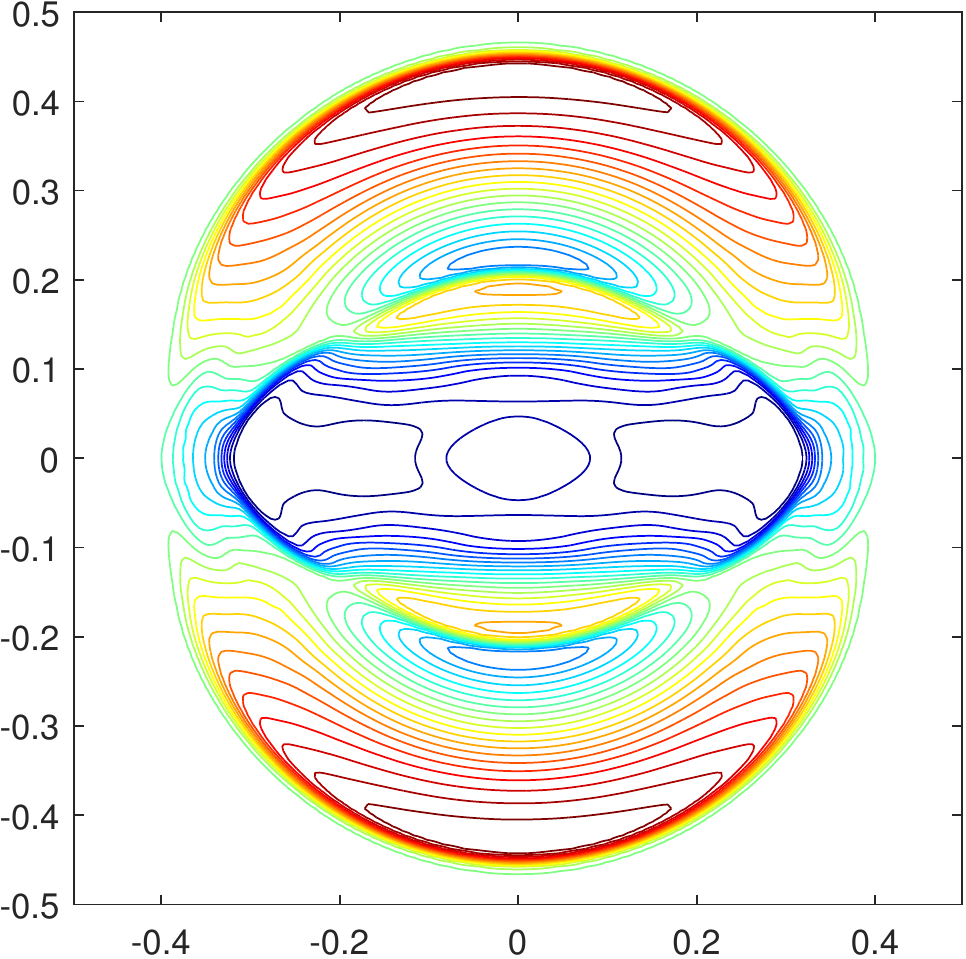}
\\
		\includegraphics[width=0.32\textwidth]
		{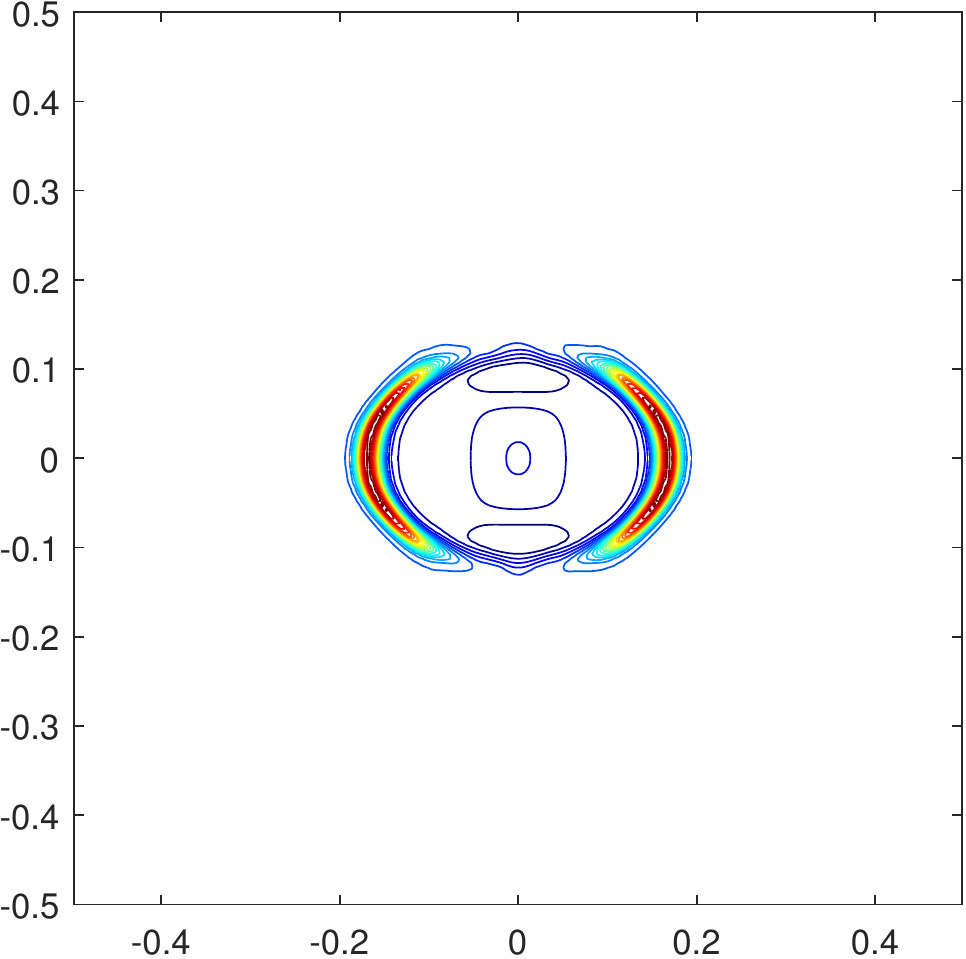}
		\includegraphics[width=0.32\textwidth]
		{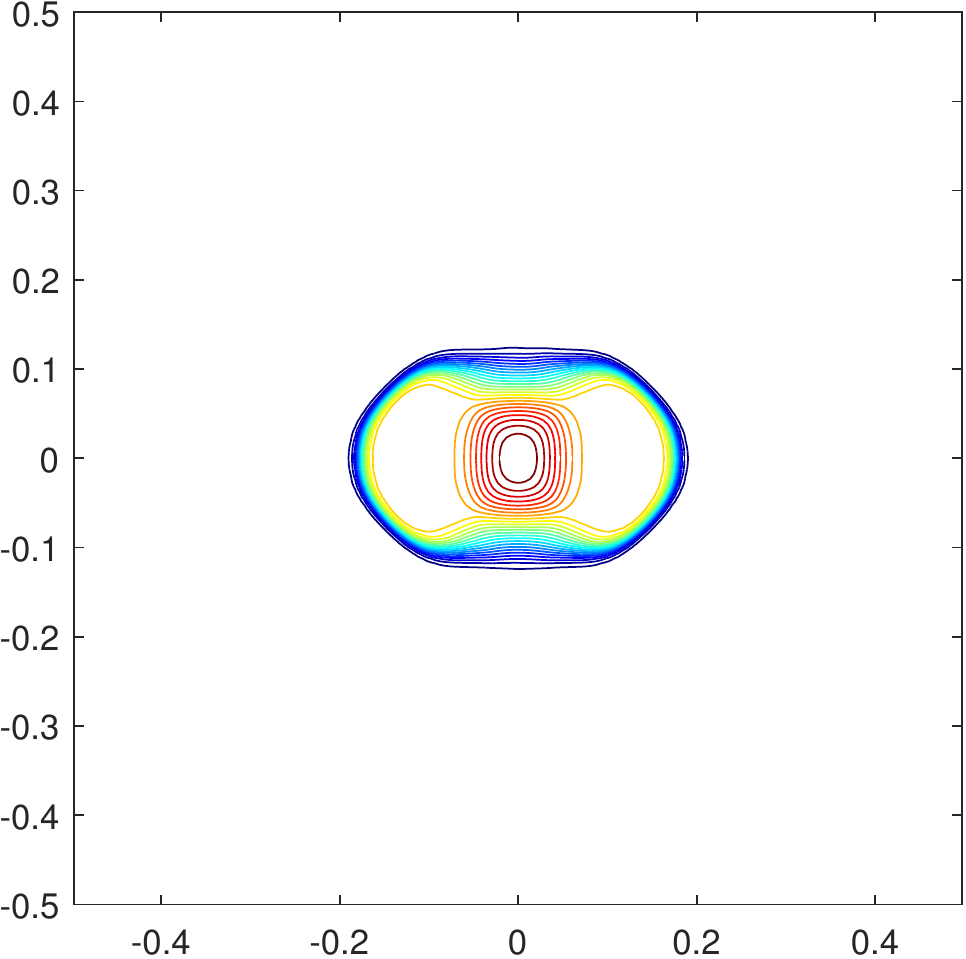}
		\includegraphics[width=0.32\textwidth]
		{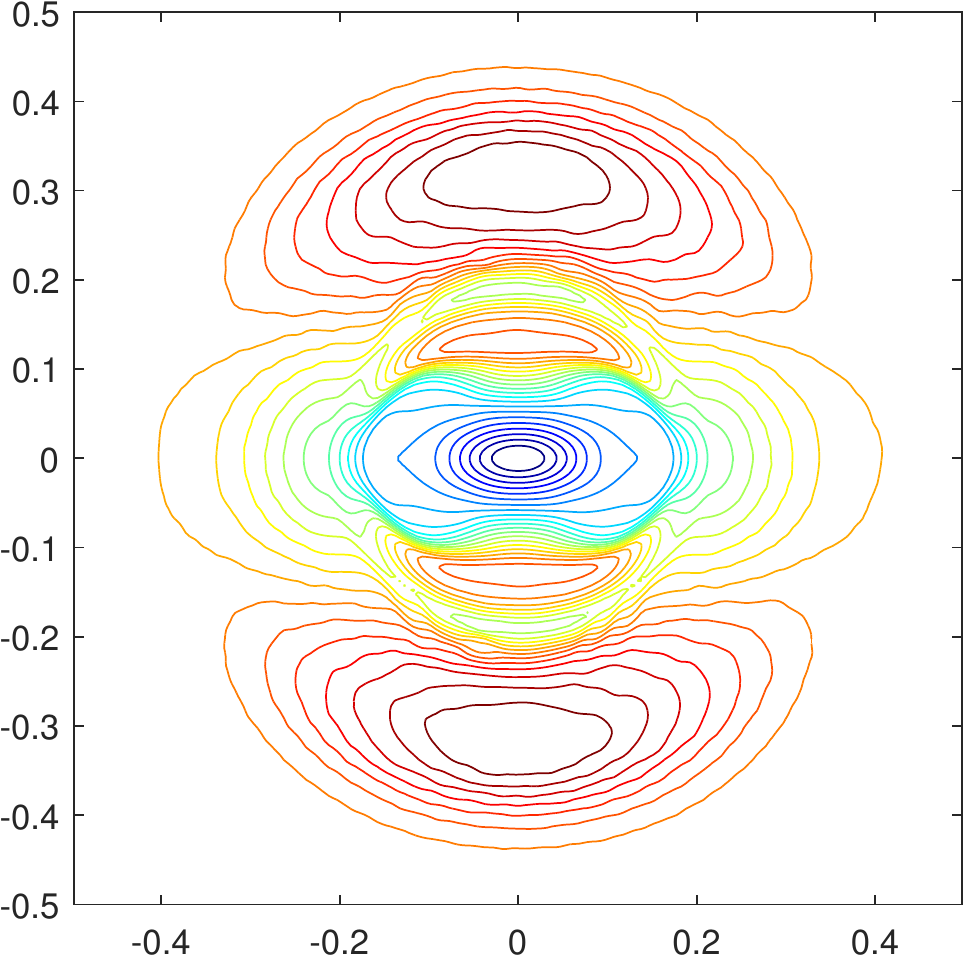}
		\captionsetup{belowskip=-12pt}
	\caption{Contour plots of density (left), thermal pressure (middle), and magnetic pressure (right). Top: results of the classical blast problem at $t = 0.01$. Bottom: results of the extreme blast problem at $t = 0.001$. }
	\label{fig:Blast12}
\end{figure}

\subsection{Astrophysical jets}

This test simulates three very challenging jet problems involving very high Mach number and strong magnetic fields. 
The setup is the same as in \cite{WuShu2018} and similar to the gas dynamical case in \cite{Balsara2012} with $\gamma = 1.4$. 
The domain $[-0.5, 0.5] \times [0, 1.5]$ is initially filled with the ambient plasma with 
 ${\bf v}=0$, $p=1$, and $\rho=0.14$. 
On the bottom boundary, the inflow jet condition ($\rho=1.4$, $p=1$, ${\bf v}=(0,800,0)$) is fixed for $x\in [-0.05,0.05]$ and $y=0$. 
All the other boundaries are set as outflow. 
The magnetic field is initialized as $(0,B_0,0)$ in the entire domain. 
We consider three 
configurations based on different strengths of $B_0$: {\em Case 1}: $B_{0} = \sqrt{200}$, and the plasma-beta $\beta =10^{-2}$; {\em Case 2}: $B_{0} = \sqrt{2000}$, and the plasma-beta $\beta=10^{-3}$; {\em Case 3}: $B_{0} = \sqrt{20000}$, and the plasma-beta $\beta=10^{-4}$. 
Since the jet Mach number is as high as $800$ 
and the magnetic field is very strong (especially in Case 3), so that the internal energy is much smaller than the kinetic/magnetic energy and negative numerical pressure can be easily produced.  
Without the proposed PP techniques the CDG code would break down within a few time steps. 
In the computation, we take the computational  domain as $[0, 0.5]\times[0, 1.5]$, divide it into $200\times600$ cells, 
use reflecting boundary condition on $x = 0$.  
The numerical results computed by our third-order locally DF PP CDG method are displayed in 
Figures \ref{fig:Jet-Density} within the domain $[-0.5, 0.5]\times[0, 1.5]$. 
We clearly see that the flow patterns are different for different strengths of $B_0$. 
The cocoons, bow shock, shear flows, and jet head location are well captured and agree with those in \cite{WuShu2018}, demonstrating the high resolution and excellent robustness of our locally DF PP CDG scheme. 
{\em It is worth mentioning that if we 
either remove our proposed discretization of the extra source term 
or neglect condition \eqref{pp-condition-2d} without using the PP limiter, 
then the simulation would fail due to the appearance of negative pressure.}

\begin{figure}[htb]
	\centering
	{
		\includegraphics[width=0.285\textwidth]
		{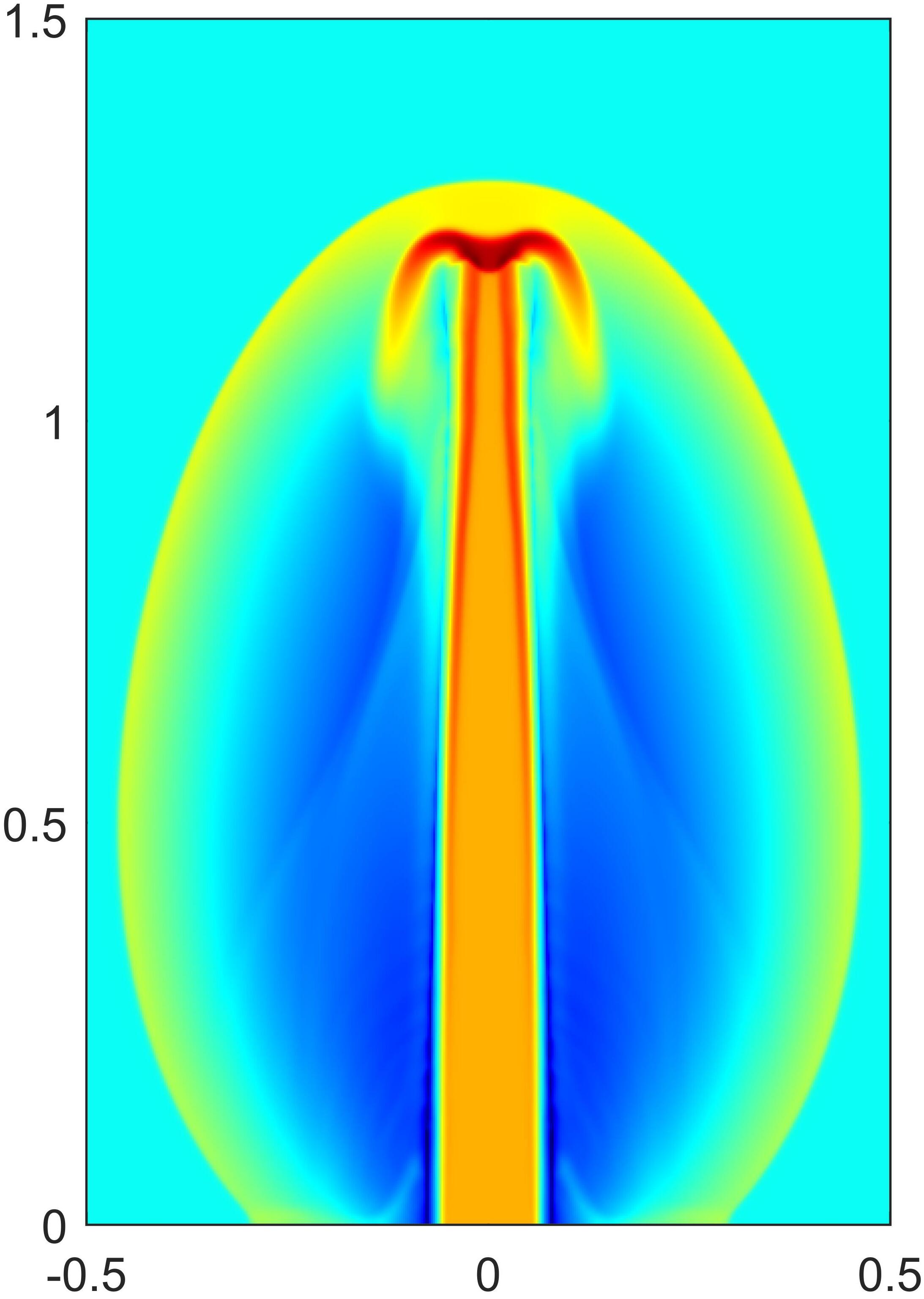}
	}
	\hfill
	{
		\includegraphics[width=0.285\textwidth]
		{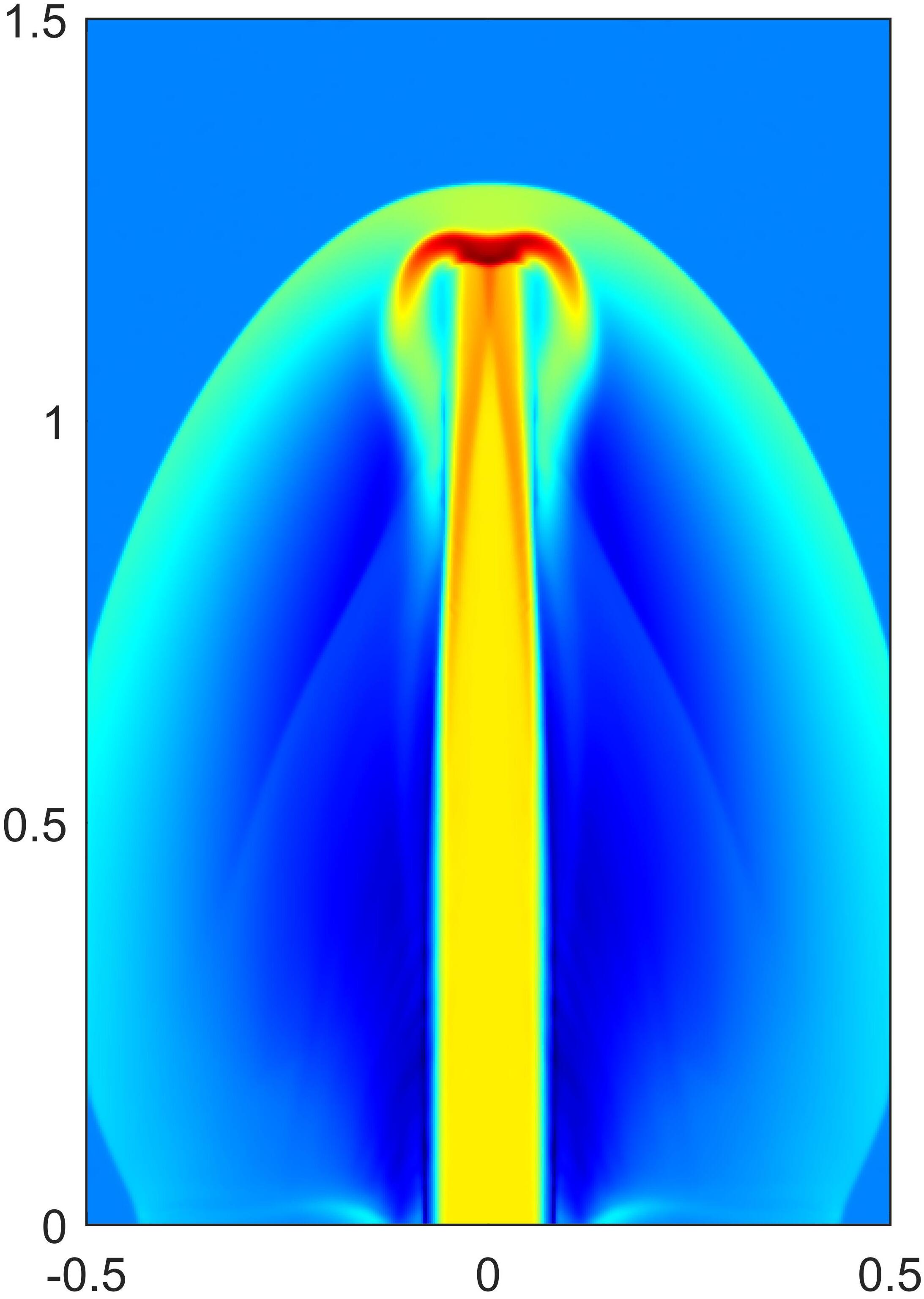}
	}
	\hfill
	{
		\includegraphics[width=0.285\textwidth]
		{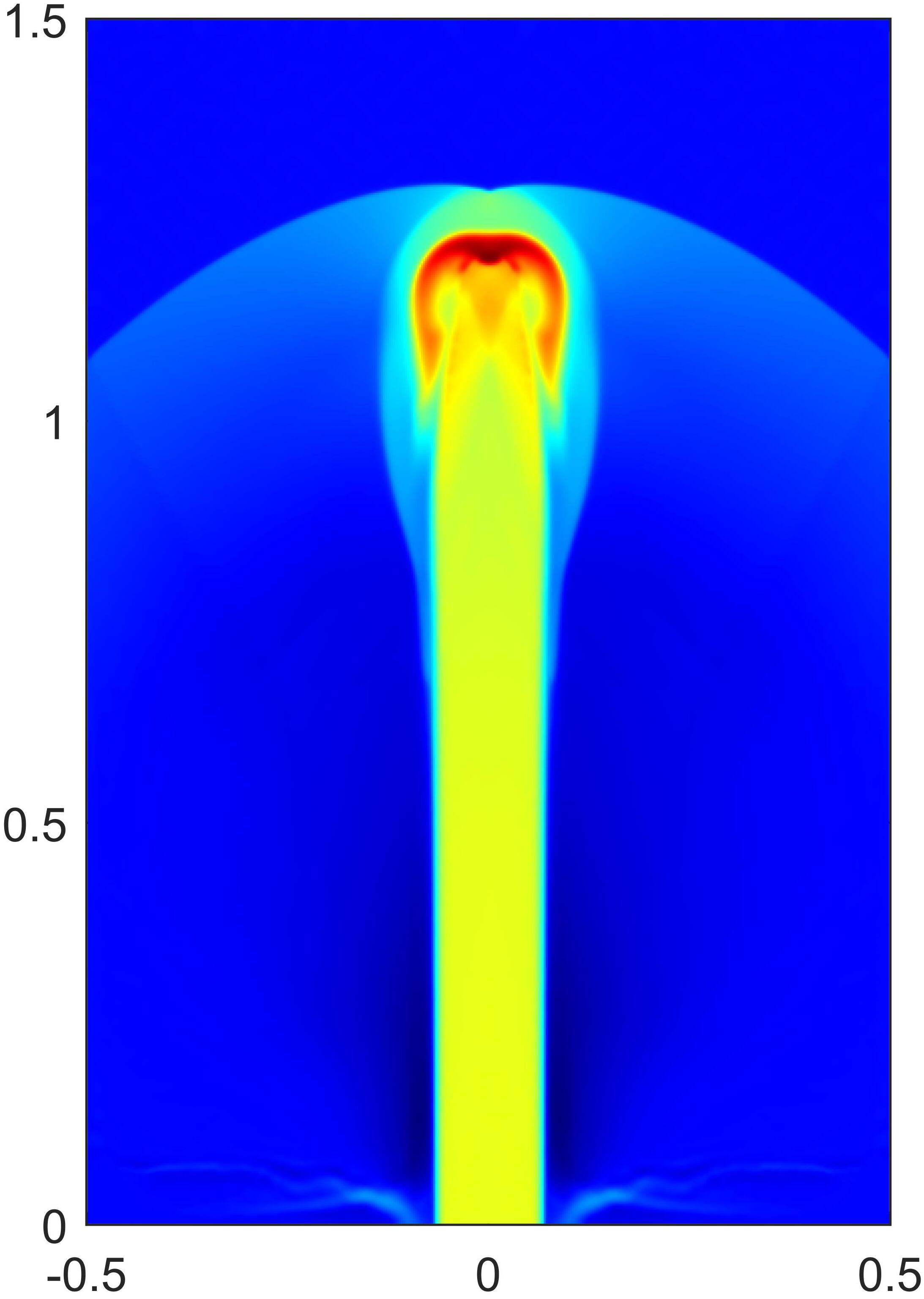}
	}
	\\
	{
		\includegraphics[width=0.285\textwidth]
		{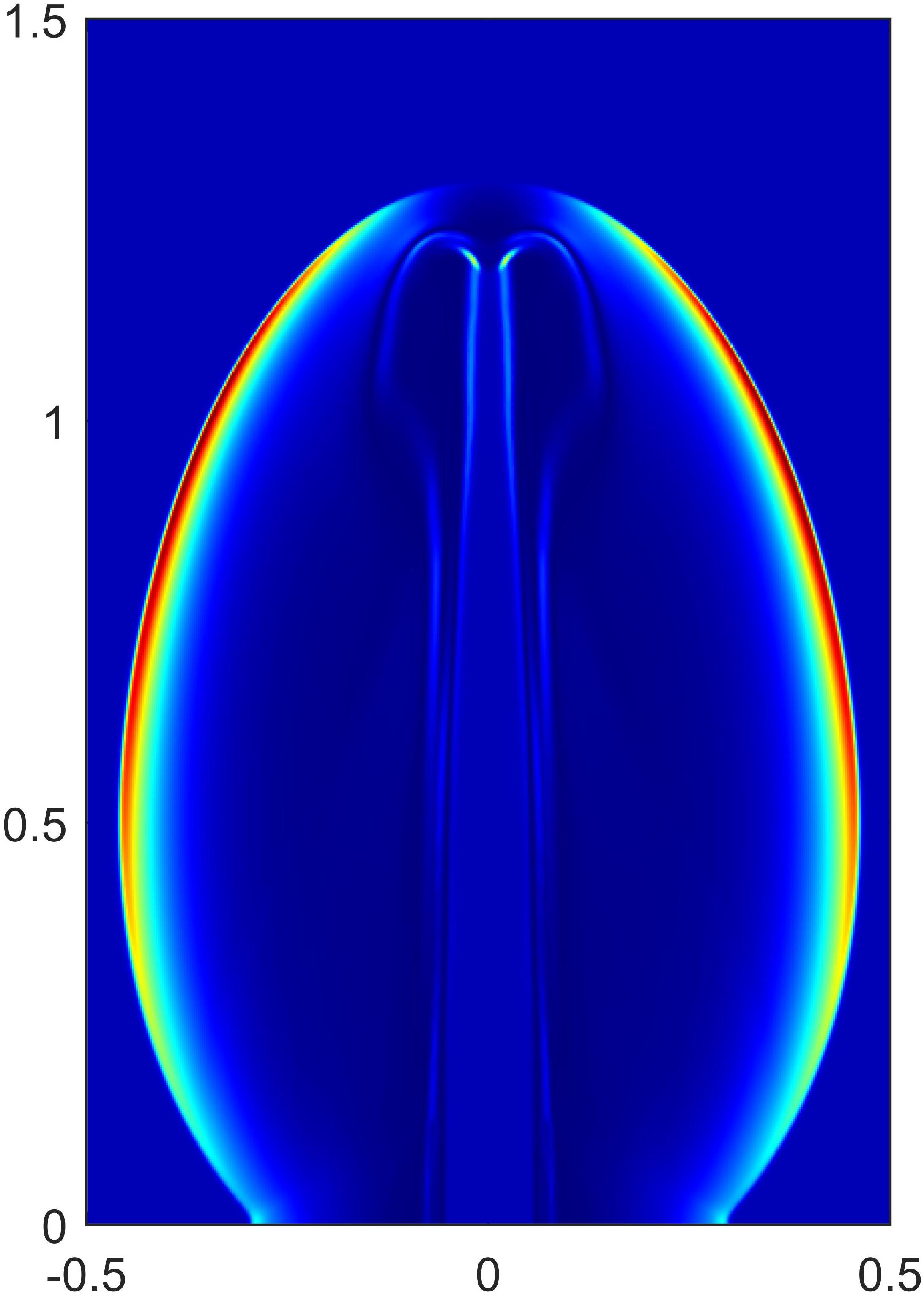}
	}
	\hfill
	{
		\includegraphics[width=0.285\textwidth]
		{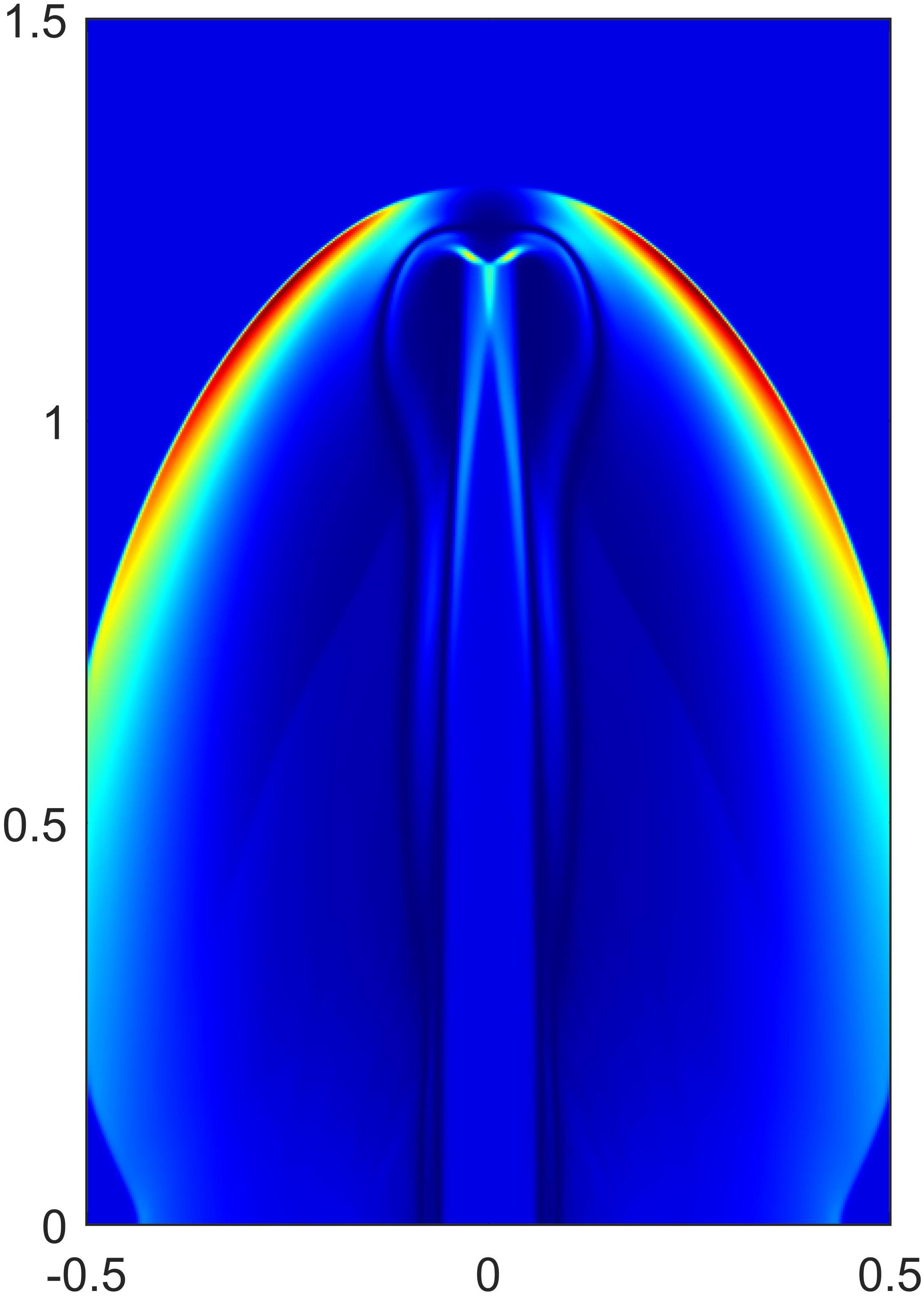}
	}
	\hfill
	{
		\includegraphics[width=0.285\textwidth]
		{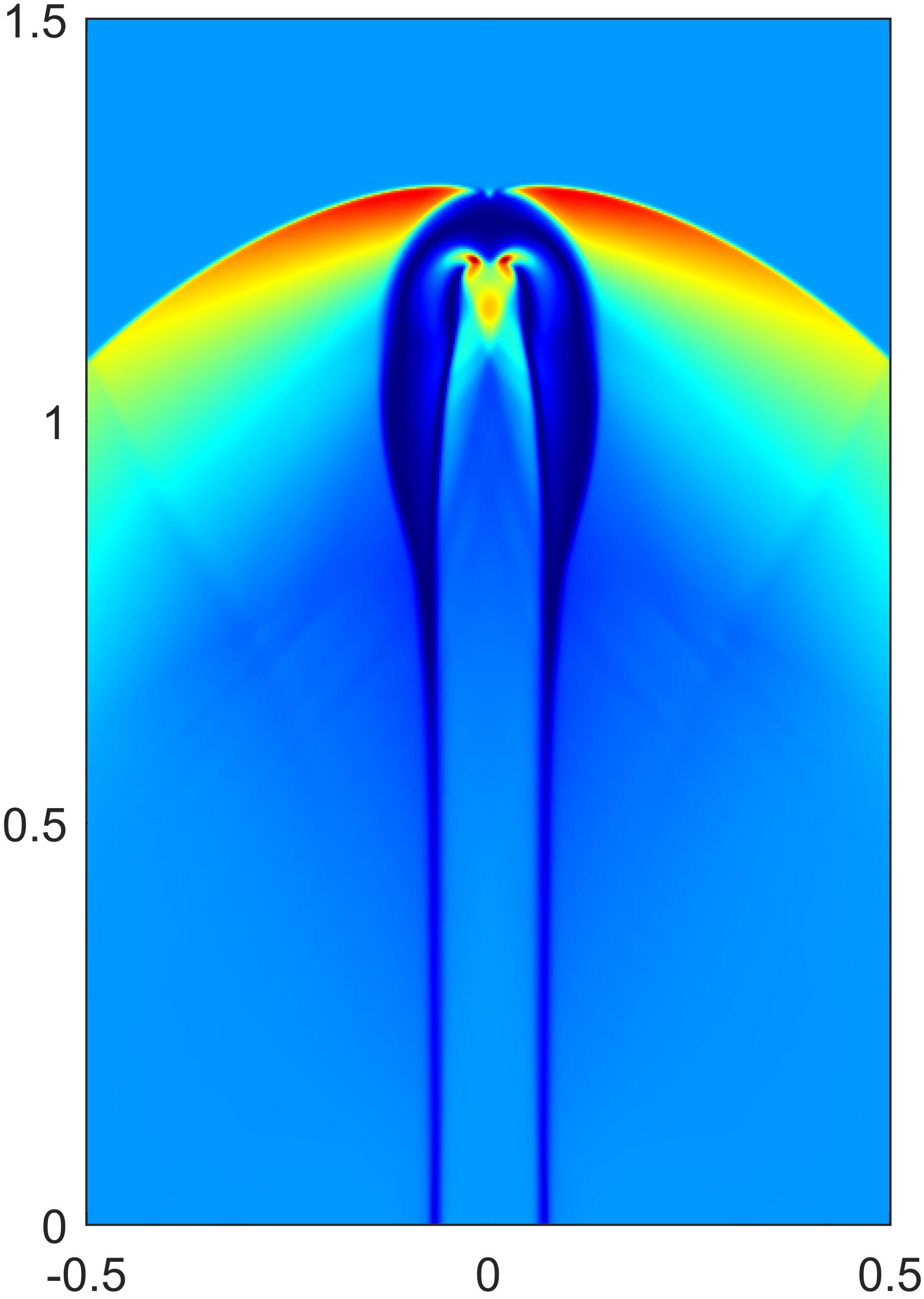}
	}
	\captionsetup{belowskip=-12pt}
	\caption{Astrophysical jets: the density logarithm (top) and the magnetic pressure (bottom) at $t = 0.002$ for Cases 1 to 3 (from left to right).} 
	\label{fig:Jet-Density}
\end{figure}

\section{Conclusions} \label{section:conclusion}

This paper has presented the first rigorous analysis on the positivity-preserving (PP) property of the central discontinuous Galerkin (CDG) approach  
for ideal magnetohydrodynamics (MHD). The analysis has further led to our design of arbitrarily high-order provably PP, (locally) divergence-free (DF) CDG schemes for 1D and 2D MHD systems. 
We have found that the PP property of the standard CDG methods is closely related to 
a discrete DF condition, 
which differs from the non-central DG case. 
This finding laid the foundation for the design of our PP CDG schemes. 
In the 1D case, the discrete DF  
condition is naturally satisfied, and we have rigorously proved that the   
standard CDG method is PP under a condition satisfied easily using an existing PP limiter \cite{cheng}. 
However, in the multidimensional cases, the corresponding discrete DF  
condition is highly nontrivial yet critical, and we have analytically proved that  
the standard CDG method, even with the PP limiter, is not PP in general, 
as it generally fails to meet the discrete DF condition. 
We have addressed this issue by carefully analyzing the structure of the discrete divergence terms and 
then 
constructing new locally DF CDG schemes for Godunov's modified MHD equations \eqref{eq:MHD:GP}. 
A challenge we have settled is to find out the suitable discretization of the source term in \eqref{eq:MHD:GP} such that it exactly offsets the divergence terms in the discovered discrete DF condition.
Based on the geometric quasilinearization approach, we have proved in theory   
the PP property of the new multidimensional CDG schemes under a CFL condition. 
Extensive benchmark and demanding numerical tests have been conducted to validate the performance of the proposed PP CDG schemes. 

In the future, we hope to further explore high-order numerical schemes preserving both the positivity and the globally DF property simultaneously. We hope our findings and newly developed analysis techniques may motivate future developments in this direction as well as the exploration of other PP central type schemes for MHD and related equations.

\renewcommand\baselinestretch{0.85}

\bibliographystyle{siamplain}
\bibliography{references_short,references_short_CDG}

\end{document}